\documentclass[letterpaper]{article}
\usepackage[latin1]{inputenc}
\usepackage{amssymb}
\usepackage{mathrsfs, amsmath}
\usepackage{amsfonts}
\usepackage{amsthm}
\usepackage{bbm}
\usepackage{graphicx}
\usepackage{theoremref, thmtools}
\usepackage{url}
\usepackage{hyperref}
\usepackage[paper=letterpaper,right=3cm,left=3cm,top=2.5cm,bottom=2.5cm]{geometry}
\usepackage{enumitem}

\usepackage{verbatim}

\DeclareFontFamily{U}{mathx}{\hyphenchar\font45}
\DeclareFontShape{U}{mathx}{m}{n}{
      <5> <6> <7> <8> <9> <10>
      <10.95> <12> <14.4> <17.28> <20.74> <24.88>
      mathx10
      }{}
\DeclareSymbolFont{mathx}{U}{mathx}{m}{n}
\DeclareFontSubstitution{U}{mathx}{m}{n}
\DeclareMathAccent{\widecheck}{0}{mathx}{"71}
\DeclareMathAccent{\wideparen}{0}{mathx}{"75}

\title{Reconstruction of the magnetic field for a 
Schr$\ddot{\textrm{o}}$dinger operator in a cylindrical setting}

\author{Daniel Campos}

\date{}

\newtheorem{theorem}{Theorem}[section]
\newtheorem{lemma}[theorem]{Lemma}
\newtheorem{proposition}[theorem]{Proposition}
\newtheorem{corollary}[theorem]{Corollary}
\newtheorem{definition}[theorem]{Definition}
\newtheorem*{remark}{Remark}
\newtheorem*{example}{Example}

\newcommand{\ov}{\overline} 			
\newcommand{\vep}{\varepsilon}  		
\newcommand{\bC}{\mathbbm{C}} 	 	
\newcommand{\bR}{\mathbbm{R}}  		
\newcommand{\bQ}{\mathbbm{Q}}  		
\newcommand{\bN}{\mathbbm{N}}  		
\newcommand{\bZ}{\mathbbm{Z}}  		
\newcommand{\bT}{\mathbbm{T}}  		
\newcommand{\cD}{\mathcal{D}}		
\newcommand{\cS}{\mathcal{S}}		
\newcommand{\ra}{\rightarrow} 		
\newcommand{\sse}{\subseteq}			

\newcommand{\pa}{\partial}			
\newcommand{\sm}{\setminus}			
\newcommand{\tr}{\textrm{tr}}			
\newcommand{\what}{\widehat}			
\renewcommand{\Re}{\textrm{Re}}		
\renewcommand{\Im}{\textrm{Im}}		
\newcommand{\wilde}{\widetilde}		
\newcommand{\lan}{\langle}			
\newcommand{\ran}{\rangle}			
\newcommand{\spec}{\textrm{Spec}}		
\newcommand{\wra}{\rightharpoonup}		
\newcommand{\supp}{\textrm{supp}}		
\newcommand{\cgeq}{\gtrsim}			
\newcommand{\cleq}{\lesssim}			
\newcommand{\Op}{\textrm{Op}}		
\newcommand{\sgn}{\textrm{sgn}}		
\newcommand{\curl}{\textrm{curl }}		
\renewcommand{\div}{\textrm{div}}		
\newcommand{\hb}{\hbar}			

\begin{document}

\maketitle

\let\thefootnote\relax\footnotetext
{\textsc{Department of Mathematics, University of Chicago, Chicago, IL, 60637, USA} \par
\ \ \textsc{Escuela de Matem\'atica, Universidad de Costa Rica, 2060 San Jos\'e, Costa Rica} \par  
\ \ \textit{E-mail address:} \texttt{campos@math.uchicago.edu}}


\begin{abstract} 
\noindent
In this thesis we consider a magnetic Schr$\ddot{\textrm{o}}$dinger 
inverse problem over a compact domain contained in an infinite 
cylindrical manifold. We show that, under certain conditions on the 
electromagnetic potentials, we can recover the magnetic field 
from boundary measurements in a constructive way. A fundamental 
tool for this procedure is a global Carleman estimate for the magnetic 
Schr$\ddot{\textrm{o}}$dinger operator. We prove this
by conjugating the magnetic operator essentially 
into the Laplacian, and using the 
Carleman estimates for it proven by Kenig--Salo--Uhlmann in the 
anisotropic setting, see \cite{KSaU1}. The conjugation is 
achieved through pseudodifferential operators over the cylinder, for 
which we develop the necessary results.

\bigskip
The main motivations to attempt this 
question are the following results concerning the magnetic
Schr$\ddot{\textrm{o}}$dinger operator: first, the solution to 
the uniqueness problem in the cylindrical setting in \cite{DKSaU}, 
and, second, the reconstruction algorithm in the 
Euclidean setting from \cite{Sa1}. We will also borrow ideas 
from the reconstruction of the electric potential in the cylindrical 
setting from \cite{KSaU2}. These two new results 
answer partially the Carleman estimate problem (Question 4.3.) 
proposed in \cite{Sa2} and the reconstruction for the 
magnetic Schr$\ddot{\textrm{o}}$dinger operator mentioned
in the introduction of \cite{KSaU2}. To our
knowledge, these are the first global Carleman estimates
and reconstruction procedure for the magnetic 
Schr$\ddot{\textrm{o}}$dinger operator available
in the cylindrical setting.
\end{abstract}


\tableofcontents


\section{Introduction}


Let us present the notion of an inverse problem through the 
following contrasting settings. A direct problem aims to determine, 
from the knowledge of the internal properties of a system, the 
reaction of it to certain stimuli. For example, 
knowing the conductivity of a medium and the voltage 
potential at the boundary we can determine the voltage induced 
in the interior of the domain and, therefore, the current flowing 
through the boundary. In contrast, an inverse problem looks to 
deduce properties of the system from the knowledge of the reactions 
to the stimuli. For instance, in his seminal paper \cite{C}, 
Calder\'on proposes to study the uniqueness and the subsequent 
reconstruction of the conductivity of a medium from the 
voltage--to--current measurements at the boundary. This problem
came to be known as the Calder\'on inverse conductivity problem. 
Since then, this and other related problems have attracted a great 
deal of attention; see the survey \cite{U1}. Various
examples of inverse problems are also presented in 
\cite{U2} and \cite{I}. 

\bigskip
For a domain $M \sse \bR^{d}$, the isotropic conductivity equation 
can be expressed as the boundary value problem 
\begin{equation*}
\biggl\{
\begin{array}{rll}
\div(\gamma \nabla u) \hspace{-2mm} & = 0 & \ \textrm{in} \ \ M, \\
u \hspace{-2mm} & = f & \ \textrm{on} \ \ \pa M,
\end{array}
\end{equation*}

where the unknown conductivity $\gamma$ is a function in $M$. 
The known data is the boundary measurement 
$\Lambda_{\gamma}: f \mapsto \gamma \pa_{\nu}u|_{\pa M}$, 
which maps the voltage 
potential at the boundary to the current flowing through the 
boundary due to the induced voltage in the interior. As mentioned 
before, the Calder\'on inverse problem consists in recovering the function 
$\gamma$ from the map $\Lambda_{\gamma}$. 

\bigskip
After a change of variables the conductivity equation can be expressed 
in the form $H_{0,W}v := (D^{2} + W)v = 0$, where $D = -i\nabla$ is the 
gradient, $D^{2} := D\cdot D = -\div\cdot \nabla$ is the (negative)
Laplacian, and $W$ is a function; we refer to $H_{0,W}$ as a 
(\textit{electric}) Schr$\ddot{\textrm{o}}$dinger operator. In greater 
generality, we can consider a \textit{magnetic} 
Schr$\ddot{\textrm{o}}$dinger operator, which has a structure similar
to the previous operator but contains first order terms in the form 
$H_{V,W} := (D + V)^{2} + W$. In any of these cases, the inverse 
problem consists in recovering information about either (or both) of 
the electromagnetic potentials $V$ and $W$, in the interior of the domain, 
from boundary measurements. We elaborate this with more detail in 
the following section. One of the reasons why this problem
is interesting and relevant is its relation to the 
inverse scattering problem at fixed energy from quantum mechanics;
see the introduction of the Ph.D. thesis by Haberman \cite{H} 
for a detailed presentation on this.

\bigskip
As mentioned before, there is a significant body of work surrounding 
these problems. In the Euclidean setting, the uniqueness (or identifiability) 
problem for the \textit{electric} Schr$\ddot{\textrm{o}}$dinger operator
was explicitly addressed by Nachman--Sylvester--Uhlmann in \cite{NSyU}, but
it was implicitly used in the proof of the uniqueness for the conductivity 
problem by Sylvester--Uhlmann in \cite{SyU}. Their proof uses the construction 
of many special solutions inspired by the complex exponential solutions 
introduced by Calder\'on in \cite{C}; this method of construction
relies on a global Carleman estimate for the Laplacian. The Carleman
estimates are a kind of parameter--dependent weighted inequalities, 
originally introduced in the setting of unique continuation problems. 
The reconstruction of the electric potential is due to Nachman, see 
\cite{N}, and uses the uniqueness for the global Carleman 
estimate from \cite{SyU} in two ways. First, it is shown that the 
uniqueness ``at infinity'' implies a uniqueness property at the 
boundary, and this allows to determine the boundary values of the 
special solutions. Second, the smallness that is established in the 
estimate makes it possible to disregard certain correction terms. Later we 
will elaborate more carefully on this. For the \textit{magnetic} operator, 
the uniqueness has been established in a series of papers under 
different assumptions. This was started with the work of Sun, in 
\cite{Sun}, under smallness conditions on the magnetic field; then 
the smallness  condition was replaced by a smoothness condition by 
Nakamura--Sun--Uhlmann in \cite{NkSU}. Further improvements of 
these include the results by Salo in \cite{Sa1} and Krupchyk--Uhlmann 
in \cite{KrU}. For a more detailed account of the available results, see 
\cite{H}. Moreover, in \cite{Sa1}, Salo carries out a 
constructive procedure to recover the electromagnetic parameters.
As before, the reconstruction uses the existence of many special 
solutions which are constructed through a Carleman estimate for
the magnetic Schr$\ddot{\textrm{o}}$dinger operator. We will 
follow closely the arguments from this paper.

\bigskip
Moving away from the Euclidean setting, the Calder\'on problem, 
or its corresponding problem for the Schr$\ddot{\textrm{o}}$dinger 
operator, can also be formulated in the context of Riemannian manifolds.
This problem arises as a model for electrical imaging in anisotropic 
media, and it is one of the most basic inverse problems in a geometric 
setting; for the basic results in this context we refer to 
\cite{Sa2}. Motivated by the results in the Euclidean setting, we 
are interested in proving analogous Carleman estimates on manifolds. 
Looking to deduce  such an estimate, in \cite{DKSaU} it is proven 
that the existence of a limiting Carleman weight implies some kind 
of product structure on the manifold. Since then, it has been usual 
to consider a cylindrical manifold, as we will do with 
$T = \bR \times \bT^{d}$, and the Carleman weight $x_{1}$; for 
instance, see \cite{KSaU1} or \cite{KSaU2}. 
Our setting will be slightly different from the so--called \textit{admissible 
Riemannian manifolds} from \cite{DKSaU}. The solution to the 
uniqueness problem for the magnetic operator was established by 
Dos Santos Ferreira--Kenig--Salo--Uhlmann in \cite{DKSaU}, and
the reconstruction problem for the electric 
Schr$\ddot{\textrm{o}}$dinger operator is elaborated 
in \cite{KSaU2}. For a more complete exposition
of the results either in the Euclidean or Riemannian setting we
refer to the surveys \cite{U1} and \cite{U2}.

\bigskip
In this thesis we prove a global
Carleman estimate for the magnetic Schr$\ddot{\textrm{o}}$dinger
operator and propose a reconstruction procedure for the magnetic
field. The main motivations to attempt this 
question are the following results concerning the magnetic
Schr$\ddot{\textrm{o}}$dinger operator: first, the solution to 
the uniqueness problem in the cylindrical setting in \cite{DKSaU}, 
and, second, the reconstruction algorithm in the 
Euclidean setting from \cite{Sa1}. We will also borrow ideas 
from the reconstruction of the electric potential in the cylindrical 
setting from \cite{KSaU2}. These two new results 
answer partially the Carleman estimate problem (Question 4.3.) 
proposed in \cite{Sa2} and the reconstruction for the 
magnetic Schr$\ddot{\textrm{o}}$dinger operator mentioned
in the introduction of \cite{KSaU2}. To our
knowledge, these are the first global Carleman estimates
and reconstruction algorithms for the magnetic 
Schr$\ddot{\textrm{o}}$dinger operator available
in the cylindrical setting.


\subsection{Setting and main results}


Let $\bT^{d} = \bR^{d}/\bZ^{d}$ be the $d$-dimensional torus with standard metric 
$g_{0}$ and let $e$ be the Euclidean metric on $\bR$. Consider the cylinder 
$T = \bR \times \bT^{d}$ with the standard product metric $g = e \oplus g_{0}$. 
We denote the points in the cylinder $T$ by $(x_{1},x')$, meaning that $x_{1} \in \bR$ 
and $x' \in \bT^{d}$. Let $(M,g) \sse T$ be a smooth connected compact 
$(d+1)$-submanifold. Let $\pa M$ denote its smooth $d$-dimensional boundary,
let $M_{-} := M \sm \pa M$ and $M_{+} = T\sm M$. We call $M_{-}$ 
and $M_{+}$ the interior and exterior of $M$, respectively.

\bigskip
Let $D_{x_{1}} = \pa_{x_{1}}/(2\pi i)$ and $D_{x'} = \nabla_{x'}/(2\pi i)$, and 
define the gradient $D := (D_{x_{1}},D_{x'})$ and Laplacian 
$-\Delta_{g} := D^{2} = D_{x_{1}}^{2} + D_{x'}^{2}$. We denote 
$-\Delta_{g_{0}} := D_{x'}^{2}$, so that its eigenvalues on $\bT^{d}$
consist of the set $\spec(-\Delta_{g_{0}}) := \{|k|^{2} : k \in \bZ^{d}\}$.

\bigskip
Let $F$, $G_{1}, \ldots, G_{d}$, $W$ be functions in $M$, and consider the vector 
field $V := (F,G) := (F,G_{1}, \ldots, G_{d})$. We call $V$ and $W$ the magnetic 
and electric potentials, respectively. 
We consider the magnetic Schr$\ddot{\textrm{o}}$dinger operator 
\[
H_{V,W} := (D + V)^{2} + W = D^{2} + 2V\cdot D + (V^{2} + D\cdot V + W),
\]

and its associated Dirichlet problem
\begin{equation*}
\label{defn: Dirichlet}
\tag{$\ast$}
\biggl\{
\begin{array}{rll}
H_{V,W}u \hspace{-2mm} & = 0 & \ \textrm{in} \ \ M_{-}, \\
u \hspace{-2mm} & = f & \ \textrm{on} \ \ \pa M.
\end{array}
\end{equation*}

In \textit{Chapter 2. Prelimaries} we will introduce the necessary notation
and motivate the following definitions. For $f \in H^{1/2}(\pa M)$, we say 
that $u \in H^{1}(M)$ is a weak solution 
to the Dirichlet problem \eqref{defn: Dirichlet} if $\tr^{-}(u) = f$ and	
\begin{equation}
\label{defn: magnetic_weak}
\int_{M}-Du\cdot D\varphi + V\cdot (\varphi Du - u D\varphi) 
+ (V^{2} + W)u\varphi = 0,
\end{equation}

for all test functions $\varphi \in H^{1}_{0}(M)$. 
Under certain conditions on the potentials, which we later elaborate,
there exists a unique weak solution to the Dirichlet problem 
\eqref{defn: Dirichlet}. We define the Dirichlet-to-Neumann (DN) map 
$\Lambda_{V,W}$ as follows: if $f, g \in H^{1/2}(\pa M)$ and 
$v \in H^{1}(M)$ is any function extending $g$, 
i.e. $\tr^{-}(v) = g$, then
\begin{equation}
\label{defn: DN_magnetic}
\lan \Lambda_{V,W}f,g \ran := \int_{M}-Du\cdot Dv + V\cdot (vDu  - uDv)
+ (V^{2} + W)uv,
\end{equation}

where $u \in H^{1}(M)$ is the weak solution of 
\eqref{defn: Dirichlet}. Formally, the DN map corresponds
to the boundary measurement 
\[
\Lambda_{V,W}f = \frac{i}{2\pi}\nu \cdot (D + V)u\bigl|_{\pa M}.
\]

The reconstruction problem then consists in using measurements
at the boundary of the domain, such as the DN map $\Lambda_{V,W}$, 
to recover information about the potentials in the interior of it.

\bigskip
Before we proceed to formulate the results,
let us recall the gauge invariance of the DN map observed in \cite{Sun}.
The conjugation identity $e^{-2\pi i\varphi}De^{2\pi i\varphi} = D + \nabla \varphi$ 
gives that $H_{V + \nabla \varphi,W} = e^{-2\pi i\varphi}H_{V,W}e^{2\pi i\varphi}$, 
which implies that if $0$ is not an eigenvalue of $H_{V,W}$ on $M$ and 
$\varphi \in C^{\infty}(T)$, then $0$ is also not an eigenvalue of the operator 
$H_{V + \nabla \varphi, W}$. Indeed, $\wilde{u}$ is a solution of the Dirichlet 
problem 
\[
\biggl\{
\begin{array}{rll}
H_{V + \nabla \varphi,W}\wilde{u} \hspace{-2mm} & = 0 & \ \textrm{in} \ \ M_{-}, \\
\wilde{u} \hspace{-2mm} & = g & \ \textrm{on} \ \ \pa M,
\end{array}
\]

if and only if $u = e^{2\pi i\varphi}\wilde{u}$ solves
\[
\biggl\{
\begin{array}{rll}
H_{V,W}u \hspace{-2mm} & = 0  & \ \textrm{in} \ \ M_{-}, \\
u \hspace{-2mm} & = e^{2\pi i\varphi}g & \ \textrm{on} \ \ \pa M.
\end{array}
\]

A routine computation yields that 
$\Lambda_{V + \nabla \varphi,W} 
= e^{-2\pi i\varphi}|_{\pa M}\Lambda_{V,W}e^{2\pi i\varphi}$
which gives the gauge invariance 
$\Lambda_{V + \nabla \varphi,W} = \Lambda_{V,W}$ if 
$\varphi|_{\pa M} = 0$. Therefore, it is not possible to determine 
the magnetic potential $V$ from the knowledge of $\Lambda_{V,W}$. 
Let us note, however, that the magnetic fields are the same, i.e. 
$\curl V = \curl(V + \nabla \varphi)$. The main result from the thesis 
is that it is possible to reconstruct the magnetic field $\curl V$ under the 
following smoothness, support, and vanishing moment conditions:
\begin{equation*}
\label{eqn: conditions_reconstruction}
\tag{$\dagger$}
V \in C^{\infty}_{c}(M_{-}), \ \ W \in L^{\infty}(M), \ \ \supp(W) \sse M,
\ \ \int_{\bR}V(x_{1},x')dx_{1} = 0
\ \textrm{for all} \ x' \in \bT^{d}.
\end{equation*}


\begin{theorem}
\label{thm: curl_reconstruction}
Let $M \sse T$ be as before, with $d \geq 3$. Assume that the potentials $V,W$ 
satisfy \eqref{eqn: conditions_reconstruction} and $0$ is not an eigenvalue of $H_{V,W}$
in $M$. Then the magnetic field $\curl V$ can be reconstructed from the knowledge 
of the Dirichlet-to-Neumann map $\Lambda_{V,W}$.
\end{theorem}


A fundamental step in the reconstruction of $\curl V$ from the DN map 
$\Lambda_{V,W}$ is the construction of many special solutions to the 
equation $H_{V,W}u = 0$. Following Sylvester--Uhlmann's method of 
complex geometric optics (CGOs), see \cite{SyU}, the solutions consist 
in appropriate corrections, depending on a large parameter, of harmonic 
functions. The standard technique to perform these constructions has been 
the use of Carleman estimates. Following \cite{N} and 
\cite{KSaU2}, we use a uniqueness result for these kind 
of estimates, as mentioned in the previous section, for a twofold purpose: 
first, to characterize the boundary values of the CGOs from the DN map;
second, to ``disregard'' the correction terms  as the parameter grows.
The other main result of the thesis is the following Carleman estimate, which 
holds under the following conditions
on the potentials:
\begin{equation*}
\label{eqn: conditions}
\tag{$\star$}
V \in C^{\infty}_{c}(T), \ \ \supp(V) \sse [-R,R]\times \bT^{d}, 
\ \ \lan x_{1}\ran^{2\delta}W \in L^{\infty}(T), 
\ \ \int_{\bR}V(x_{1},x')dx_{1} = 0
\ \textrm{for all} \ x' \in \bT^{d}.
\end{equation*}


\begin{theorem}
\label{thm: carleman_full}
Let $1/2 < \delta < 1$ and let $V,W$ satisfy \eqref{eqn: conditions}. 
There exists $\tau_{0} \geq 1$, such that if $|\tau|\geq \tau_{0}$ and
$\tau^{2} \notin \spec(-\Delta_{g_{0}})$, then for any $f \in L^{2}_{\delta}(T)$
there exists a unique $u \in H^{2}_{-\delta}(T)$ which solves 
\[
e^{2\pi\tau x_{1}}H_{V,W}e^{-2\pi\tau x_{1}}u = f.
\]
 
Moreover, this solution satisfies the estimates 
\[
\|u\|_{H^{s}_{-\delta}(T)} \cleq |\tau|^{s - 1}\|f\|_{L^{2}_{\delta}(T)},
\]

for $s = 0,1,2$. The constant of the inequality is independent of $\tau$.
\end{theorem}


In the next chapter we introduce the weighted Sobolev spaces 
$L^{2}_{\delta}(T)$ and $H^{s}_{-\delta}(T)$.
The solution to this equation is based on a reduction to the case of the
Laplacian, i.e. when there are no electromagnetic potentials. The gain of 
one derivative in the estimate, meaning the constant $\tau^{-1}$, allows 
to deduce the estimate of Theorem \ref{thm: carleman_full}
in the presence of an electric potential alone through perturbative methods. The 
reconstruction procedure of the electric potential has been given in \cite{N}
for the Euclidean case and in \cite{KSaU2} for the cylindrical case. 
However, the gain of one derivative is not enough to deal with the magnetic potential 
beyond the perturbative regime, i.e. when the norm of the magnetic potential may
not be small. Following the ideas in \cite{NkU}, \cite{NkSU}, and especially \cite{Sa1},
we prove this by conjugating the equation through pseudodifferential operators in 
order to ``essentially eliminate'' the magnetic potential. To do this we consider the 
small parameter $\hb = \tau^{-1}$ and use the results for semiclassical analysis on $\bR$.


\subsection{Structure of the thesis}


In \textit{Chapter 2. Preliminaries} we recall some definitions and results on Fourier analysis, 
introduce the function spaces that will appear through the problem, and present the basic facts
necessary to formulate the magnetic Schr$\ddot{\textrm{o}}$dinger inverse problem.

\bigskip
\noindent
In \textit{Chapter 3. Semiclassical pseudodifferential operators over $\bR \times \bT^{d}$} we 
define these operators over $T$ and prove the usual results specific to our cylindrical setting. 
These results do not seem to be explicitly stated in the standard 
references, see \cite{T2} or \cite{Z}, so, we elaborate the necessary theory for it. For zero 
order pseudodifferential operators we prove an analog of the Calder\'on--Vaillancourt 
$L^{2}$--boundedness theorem, as well as a norm estimate for the first order expansion 
of the composition of two such operators. 

\bigskip
In \textit{Chapter 4. Conjugation and Carleman estimate} we carry out the construction 
of the conjugation as well as the proof of 
Theorem \ref{thm: carleman_full}. The conjugation requires the solution of a first order
differential equation, together with the appropriate estimates. In our cylindrical setting,
through the expansion in Fourier series, this equation can be reduced to the solution of 
multiple ODEs. The ideas follow closely the results from \cite{Sa1}.

\bigskip
In \textit{Chapter 5. Equivalent formulations and boundary characterization} we use 
Theorem \ref{thm: carleman_full} to construct many solutions of the equation $H_{V,W}u = 0$.
Starting from a harmonic solution, we construct a unique solution (CGO) to the equation that  
``behaves like'' it at infinity. We show that the uniqueness at infinity implies a uniqueness property
at the boundary, and so the boundary values of the CGOs can be characterized as solutions 
to boundary integral equations involving only the knowledge of the DN map $\Lambda_{V,W}$ 
and not the unknown electromagnetic potentials. We follow the presentation from 
\cite{KSaU2}.

\bigskip
In \textit{Chapter 6. Reconstruction of the magnetic field} we restrict the attention to
CGOs that result from correcting the harmonic functions $e^{\pm 2\pi|m|x_{1}}e_{m}(x')$. 
We prove that such CGOs can also be written in the form 
$e^{\pm 2\pi|m|x_{1}}e_{m}(x')a_{m} + e^{-2\pi\tau x_{1}}r_{m,\tau}$, for an 
appropriate amplitude $a_{m}$ making the correction term have better estimates. 
Then we define an analog of the scattering transform from \cite{N} and \cite{Sa1},
and use it together with the correction estimates to obtain integrals that are basically
a mixed (Laplace--Fourier) transform of terms involving the magnetic potential. Finally, 
we show that it is possible to recover the magnetic field $\curl V$ from these integrals. 
These steps require some linear algebra lemmas over $\bQ$ and the reconstruction 
formula for an entire function, which we prove in the appendix of the chapter. This is
perhaps the most interesting chapter: not only the methods require playful ideas, 
but the results obtained are somewhat different from analogous previous ones.


\section{Preliminaries}


Consider the cylinder $T = \bR \times \bT^{d}$ with 
standard product metric $g = e \oplus g_{0}$. The points
in $T$ are denoted by $x = (x_{1},x')$, meaning that 
$x_{1} \in \bR$ and $x' \in \bT^{d}$. 
Let $(M,g) \sse T$ is a smooth connected compact 
$(d+1)$-submanifold with boundary $\pa M$. We denote 
the volume element in $T$ and $M$ by $dx = dx_{1}dx'$ and 
the surface measure in $\pa M$ by $d\sigma$.

\bigskip
Let $D_{x_{1}} = \pa_{x_{1}}/(2\pi i)$ and $D_{x'} = \nabla_{x'}/(2\pi i)$, 
and define the gradient $D := (D_{x_{1}},D_{x'})$ and Laplacian 
$-\Delta_{g} := D^{2} = D_{x_{1}}^{2} + D_{x'}^{2}$.
For a multiindex $\alpha = (\alpha_{1},\alpha') 
= (\alpha_{1},\alpha'_{1},\ldots,\alpha'_{d})$,
we denote $|\alpha| = \alpha_{1} + \alpha'_{1} + \ldots + \alpha'_{d}$
and $D^{\alpha} = D_{x_{1}}^{\alpha_{1}}
D_{x'_{1}}^{\alpha'_{1}}\ldots D_{x'_{d}}^{\alpha'_{d}}$.

\bigskip
In what follows we define several functions spaces over 
$T$, and we mention when the 
definitions allow for analogous spaces over $\bR$, $\bT^{d}$, 
or $M$. Most of the definitions and results from this chapter can 
be found in \cite{Sa1}, \cite{T1}, \cite{T2}, 
\cite{Z}.


\subsection{Fourier analysis and distributions}


\subsubsection{Distributions and Sobolev spaces}

We consider the space of smooth compactly supported functions
$\cD(T) := C^{\infty}_{c}(T)$ with the family of seminorms 
\[
\|f\|_{k,l} = \sup\{|D^{\alpha}f(x_{1},x')| : |x_{1}| \leq k, \
x' \in \bT^{d}, \ |\alpha| \leq l\},
\]

with $k,l \in \bN$. We say a linear functional $\varphi: \cD(T) \ra \bC$
is continuous, if for all $k \in \bN$ there exist
$l \in \bN$ and $C > 0$, both possibly depending of $k$, such
that $|\lan \varphi, f\ran| \leq C\|f\|_{k,l}$ for all $f \in \cD(T)$. 
We call distributions to these functionals and denote
its space by $\cD'(T)$.

\bigskip
In addition, we define the space of Schwartz functions $\cS(T)$ 
as the space of rapidly decaying smooth functions with the
family of seminorms 
\[
\|f\|_{k} = \sup\{\lan x_{1}\ran^{k}
|D^{\alpha}f(x_{1},x')| : (x_{1},x') \in T, \ 0 \leq |\alpha| \leq k \},
\]

with $k \in \bN$. We say that $f_{j} \ra f$ 
in $\cS(T)$ if $\|f_{j} - f\|_{k} \ra 0$ for all $k$. We say a linear 
functional $\varphi: \cS(T) \ra \bC$ is continuous if 
$\lan \varphi, f_{j}\ran \ra \lan \varphi,f\ran$ whenever
$f_{j} \ra f$ in $\cS(T)$. We call tempered distributions 
to these functionals and denote its space by $\cS'(T)$.
We define that $\varphi_{j} \wra \varphi$ in $\cS'(T)$ if 
$\lan \varphi_{j}, f\ran \ra \lan \varphi, f\ran$ for all $f \in \cS(T)$. 
A well--known result in functional analysis says that if 
$\varphi \in \cS'(T)$, then there exist $k \in \bN$ and $C > 0$ such 
that $|\lan \varphi, f\ran| \leq C\|f\|_{k}$ for all $f \in \cS(T)$.
The space of tempered distributions $\cS'(T)$ is a subspace of
the distributions $\cD'(T)$. The definitions of the spaces $\cS(\bR)$ 
and $\cS'(\bR)$ are analogous. 

\bigskip
For $1 \leq p \leq \infty$, let $L^{p}(T) = L^{p}(T,dx_{1}dx')$ denote 
the standard $L^{p}$ space in $T$. For a nonnegative integer $s$, we 
consider the $L^{p}$ Sobolev spaces $W^{s,p}(T)$  with norm given by
$\|f\|_{W^{s,p}(T)} 
:= \sum_{|\alpha| \leq s}\|D^{\alpha}f\|_{L^{p}(T)}$.
Similarly, we also consider the spaces $W^{s,p}(\bR)$ and 
$W^{s,p}(\bT^{d})$.


\subsubsection{Fourier analysis on smooth functions}


For a function $f \in L^{1}(\bR)$ we define its Fourier 
transform by 
$\what{f}(\xi) := \int_{\bR}e^{-2\pi ix_{1}\xi}f(x_{1})dx$.


\begin{proposition}[\cite{T1}, \cite{Z}]
\label{propn: Fourier_R}
If $f \in \cS(\bR)$, then its Fourier transform $\what{f}$
satisfies the following:
\begin{enumerate}[label=\alph*).]
\item the transform and its derivatives have polynomial decay bounds 
\[
|D_{\xi}^{\alpha}\what{f}(\xi)| 
\cleq \lan \xi\ran^{-2m}\|x_{1}^{\alpha}f\|_{W^{2m,1}(\bR)},
\]
where $\lan \xi\ran := (1 + \xi^{2})^{1/2}$ and the constant 
of the inequality may depend on $m$,
\item $\what{f} \in \cS(\bR)$ and we have the
inversion formula 
$f(x_{1}) = \int_{\bR}e^{2\pi ix_{1}\xi}\what{f}(\xi)d\xi$,
with pointwise absolute uniform convergence, as well as for 
its derivatives,
\item Plancherel's identity holds,
$\|f\|_{L^{2}(\bR)} = \|\what{f}\|_{L^{2}(\bR)}$.
\end{enumerate}
\end{proposition}


For $k \in \bZ^{d}$, let $e_{k}(x') := e^{2\pi ik\cdot x'}$. 
For a function $f \in L^{1}(\bT^{d})$ we define its $k$-th Fourier 
coefficient by $f_{k} := \int_{\bT^{d}}e_{-k}(x')f(x')dx'$.


\begin{proposition}[\cite{T1}]
\label{propn: Fourier_torus}
If $f \in C^{\infty}(\bT^{d})$, then its Fourier coefficients and 
series satisfy the following:
\begin{enumerate}[label=\alph*).]
\item the coefficients have polynomial decay bound 
\[
|f_{k}| \cleq \lan k\ran^{-2m}\|f\|_{W^{2m,1}(\bT^{d})},
\]
where $\lan k\ran := (1 + |k|^{2})^{1/2}$ and the constant of the
inequality may depend on $m$ and $d$,
\item there is pointwise absolute uniform convergence of the 
Fourier series $f(x') = \sum_{k \in \bZ^{d}}f_{k}e_{k}(x')$,
as well as for of its derivatives,
\item Plancherel's identity holds, 
$\|f\|_{L^{2}(\bT^{d})}^{2} = \sum_{k \in \bZ^{d}}|f_{k}|^{2}$.
\end{enumerate}
\end{proposition}


Similarly, for a function $f \in \cS(T)$ we define its $k$-th Fourier coefficient 
by $f_{k}(x_{1}) := \int_{\bT^{d}}e_{-k}(x')f(x_{1},x')dx'$. 
The previous results can be combined as follows.


\begin{proposition}
\label{propn: Fourier_cylinder}
If $f \in \cS(T)$, then its Fourier coefficients $f_{k}$ are in $\cS(\bR)$. 
Moreover, these satisfy the following:
\begin{enumerate}[label=\alph*).]
\item the coefficients have polynomial decay bounds 
\[
\|f_{k}\|_{L^{1}(\bR)} \cleq \lan k\ran^{-2m}\|f\|_{W^{2m,1}(T)},
\]
where the constant of the
inequality may depend on $m$ and $d$,
\item the transform of the coefficients and its derivatives have polynomial 
decay bounds 
\[
|D_{\xi}^{\alpha}\what{f_{k}}(\xi)| 
\cleq \lan \xi, k \ran^{-2m}\|x_{1}^{\alpha}f\|_{W^{2m,1}(T)},
\]
where $\lan \xi, k\ran := (1 + \xi^{2} + |k|^{2})^{1/2}$ and the 
constant of the inequality may depend on $m$ and $d$,
\item the inversion formula holds,
\[
f(x_{1},x') = \sum_{k \in \bZ^{d}}f_{k}(x_{1})e_{k}(x')
= \sum_{k \in \bZ^{d}}\int_{\bR}
e^{2\pi ix_{1}\xi}e_{k}(x')\what{f_{k}}(\xi)d\xi,
\]
with pointwise absolute uniform convergence, as well as for its derivatives,
\item Plancherel's identity holds, 
$\|f\|_{L^{2}(T)}^{2} = \sum_{k \in \bZ^{d}}\|f_{k}\|_{L^{2}(\bR)}^{2} 
= \sum_{k \in \bZ^{d}}\|\what{f_{k}}\|_{L^{2}(\bR)}^{2}$. 
\item for any $k \in \bZ^{d}$, the function $f_{k}(x_{1})e_{k}(x')$ 
is in $\cS(T)$, and we have that 
\[
\|f_{k}e_{k}\|_{l} \cleq \lan k\ran^{-(2m - l)}\|f\|_{l + 2m},
\]
where the constant may depend on $l$, $m$, and $d$.
Moreover, the partial sums of the Fourier series,
$S_{N}f(x_{1},x') := \sum_{|k| \leq N}f_{k}(x_{1})e_{k}(x')$,
converge to $f$ in $\cS(T)$,
\end{enumerate}
\end{proposition}


\begin{proof}
For $\alpha \leq m$ we have that
\[
\lan x_{1}\ran^{m}|D_{x_{1}}^{\alpha}f_{k}|
\leq \|\lan x_{1}\ran^{m}D_{x_{1}}^{\alpha}
f(x_{1},\cdot)\|_{L^{1}(\bT^{d})}
\leq \|f\|_{m},
\]

where $\|f\|_{m}$ is the seminorm defined above 
for functions in $\cS(T)$. This proves 
that the Fourier coefficients $f_{k}$ are in $\cS(\bR)$. 
Moreover, using the identity 
$\lan k\ran^{2m}e_{-k}(x') = \lan D_{x'}\ran^{2m}e_{-k}(x')$ 
and integrating by parts yields that
\begin{align*}
\lan k\ran^{2m}|f_{k}(x_{1})| 
&= \biggl|\int_{\bT^{d}} \lan D_{x'}\ran^{2m}(e_{-k}(x'))f(x_{1},x')dx'\biggr| \\
& = \biggl|\int_{\bT^{d}} e_{-k}(x')\lan D_{x'}\ran^{2m}f(x_{1},x')dx'\biggr|
\cleq \int_{\bT^{d}}|\lan D_{x'}\ran^{2m}f(x_{1},x')|dx'.
\end{align*}

Integrating over $\bR$ gives the first result. Similarly,
using the identities
\[
D_{\xi}^{\alpha}e^{-2\pi ix_{1}\xi} 
= (-x_{1})^{\alpha}e^{-2\pi ix_{1}\xi}, \ \  
\lan \xi, k\ran^{2m}(e^{-2\pi ix_{1}\xi}e_{-k}(x'))
= \lan D\ran^{2m}(e^{-2\pi ix_{1}\xi}e_{-k}(x')),
\]

and integrating by parts we conclude that
\begin{align*}
\lan \xi, k\ran^{2m}|D_{\xi}^{\alpha}\what{f_{k}}(\xi)|
& = \biggl|\int_{T}\lan D \ran^{2m}(e^{-2\pi ix_{1}\xi}e_{-k}(x'))
x_{1}^{\alpha}fdx_{1}dx'\biggr| \\
& = \biggl|\int_{T}e^{-2\pi ix_{1}\xi}e_{-k}(x')
\lan D \ran^{2m}(x_{1}^{\alpha}f)dx_{1}dx'\biggr|
\cleq \|x_{1}^{\alpha}f\|_{W^{2m,1}(T)}. 
\end{align*}

Moreover, this result proves that the Fourier decomposition 
converges absolutely, and so the inversion and Plancherel 
formulas follow from Proposition \ref{propn: Fourier_R} and 
Proposition \ref{propn: Fourier_torus}. Finally, we can bound 
\[
\|f_{k}e_{k}\|_{l} 
= \sup\{\lan x_{1}\ran^{l}|D^{\alpha}(f_{k}e_{k})| 
: |\alpha| \leq l\}
\cleq \lan k\ran^{\l}
\sup\{\lan x_{1}\ran^{l}|D_{x_{1}}^{\alpha_{1}}f_{k}| 
: \alpha_{1} \leq l\} \cleq \lan k\ran^{-(2m - l)}\|f\|_{l + 2m},
\]

where we have used Proposition \ref{propn: Fourier_torus} for the 
last step. The convergence of the partial sums in $\cS(T)$ follows 
from this.
\end{proof}


\subsubsection{Fourier analysis on tempered distributions}


From Fubini's theorem we have that 
$\int_{\bR}f\what{g} = \int_{\bR}\what{f}g$,
for $f,g \in \cS(\bR)$. This suggests to define the Fourier
transform of $\varphi \in \cS'(\bR)$ by 
\[
\lan \what{\varphi}, f\ran := \lan \varphi, \what{f}\ran.
\] 

To see indeed that $\what{\varphi} \in \cS'(\bR)$ we use
Proposition \ref{propn: Fourier_R} to get that $\what{f}_{n} \ra \what{f}$
in $\cS(\bR)$ if $f_{n} \ra f$ in $\cS(\bR)$. It is clear that
this definition extends the above definition of Fourier transform
in $\cS(\bR)$. 

\bigskip
Finally, we proceed to define the Fourier coefficients of a tempered 
distribution. Let us consider the operators
$\pi_{k} : \cS(T) \ra \cS(\bR)$ and $\psi_{k} : \cS(\bR) \ra \cS(T)$
given by $\pi_{k}f := f_{k}$ and
$\psi_{k}g := g(x_{1})e_{k}(x')$.
The Fourier inversion formula on $\cS(T)$ 
can be written formally as $I = \sum_{k \in \bZ^{d}}\psi_{k}\pi_{k}$.
Moreover, proceeding as in the proof of Proposition \ref{propn: Fourier_cylinder} 
we have that
\[
\|\pi_{k}f\|_{l} \cleq \lan k\ran^{-2m}\|f\|_{l + 2m}, \ \ 
\|\psi_{k}g\|_{l} \cleq \lan k\ran^{l}\|g\|_{l},
\]

for all $l, m \geq 0$. 
By duality, this gives rise to the adjoint operators 
$\pi_{k}^{*} : \cS'(\bR) \ra \cS'(T)$ and
$\psi_{k}^{*}: \cS'(\bT) \ra \cS'(\bR)$, 
defined by 
\[
\lan \pi_{k}^{*}\phi, f\ran := \lan \phi, \pi_{k}f\ran, \ \
\lan \psi_{k}^{*}\varphi,g \ran := \lan \varphi, \psi_{k}g\ran.
\]

For a distribution $\varphi \in \cS'(T)$, we define its $k$-th Fourier
coefficient by $\varphi_{k} := \psi_{-k}^{*}\varphi \in \cS'(\bR)$.
This definition extends that of Fourier coefficients for functions in 
$\cS(T)$. Below, we prove the formal dual of the inversion formula 
above, which reads 
$I = \sum_{k \in \bZ^{d}}\pi_{k}^{*}\psi_{k}^{*}
= \sum_{k \in \bZ^{d}}\pi_{-k}^{*}\psi_{-k}^{*}$.


\begin{proposition}
\label{propn: Fourier_distributions}
Let $f \in \cS(T)$ and $\varphi \in \cS'(T)$. The Fourier 
coefficients satisfy the following:
\begin{enumerate}[label=\alph*).]
\item the usual differentiation properties hold, i.e. 
$(D_{x'}^{\alpha}\varphi)_{k} = k^{\alpha}\varphi_{k}$,
\item Parseval's identity holds, $\lan \varphi, \ov{f}\ran 
= \sum_{k \in \bZ^{d}}\lan \varphi_{k}, \ov{f_{k}}\ran$, 
with absolute convergence.
\item the partial sums of the Fourier series, 
$S_{N}\varphi = \sum_{|k| \leq N}\pi_{-k}^{*}\varphi_{k}$,
converge to $\varphi$ in $\cS'(T)$.
\end{enumerate}
\end{proposition}


\begin{proof}
If $g \in \cS(\bR)$, then we have that 
\[
\lan (D_{x'}^{\alpha}\varphi)_{k}, g\ran
= \lan  D_{x'}^{\alpha}\varphi, ge_{-k} \ran
= (-1)^{\alpha} \lan \varphi, D_{x'}^{\alpha}(ge_{-k})\ran
= (-1)^{\alpha} \lan \varphi, (-k)^{\alpha}ge_{-k}\ran
= k^{\alpha} \lan \varphi_{k}, g\ran,
\]

i.e. $(D_{x'}^{\alpha}\varphi)_{k} = k^{\alpha}\varphi$. 
To prove Parseval's identity we first prove that the series 
converges absolutely. We know that there exists $l \in \bN$ 
such that $|\lan \varphi, g\ran| \cleq \|g\|_{l}$ for all 
$g \in \cS(T)$. From a remark above we have that
\[
|\lan \varphi_{k},\ov{f_{k}}\ran|
= |\lan \varphi, \ov{f_{k}e_{k}}\ran|
\cleq \|f_{k}e_{k}\|_{l} 
\cleq \lan k\ran^{-(2m- l)}\|f\|_{l + 2m},
\]

so it follows that the series converges absolutely 
by choosing $m$ large. Recalling from Proposition \ref{propn: Fourier_cylinder} 
that $S_{N}f \ra f$ in $\cS(T)$, we conclude that, 
\[
\lan \varphi, \ov{f}\ran 
= \lim_{N \ra +\infty} \sum_{|k| \leq N}
\lan \varphi, \ov{f_{k}e_{k}}\ran
= \lim_{N \ra +\infty} \sum_{|k| \leq N}
\lan \varphi_{k}, \ov{f_{k}}\ran
= \sum_{k \in \bZ^{d}} \lan \varphi_{k}, \ov{f_{k}}\ran.
\]

The convergence $S_{N}\varphi \wra \varphi$ in $\cS'(T)$ 
follows from Plancherel's identity and the fact that 
$\ov{f_{k}} = (\ov{f})_{-k}$.
\end{proof}


\subsection{Function spaces}


Recall that for a nonnegative integer $s$ we 
considered the Sobolev spaces $W^{s,p}(T)$  
with norm $\|f\|_{W^{s,p}(T)} 
:= \sum_{|\alpha| \leq s}\|D^{\alpha}f\|_{L^{p}(T)}$.
For $p = 2$ we denote $H^{s}(T) := W^{s,2}(T)$. 
The definitions of these spaces over $\bR$, $\bT^{d}$, 
and $M$ are analogous. In the case of $T$ and $\bR$,
these spaces are also the completions of the corresponding
space of Schwartz functions under the respective norm, 
while for $M$ these spaces are the completions of 
restrictions to $M$ of the Schwartz functions $\cS(T)$ 
under the $W^{s,p}(M)$ norm.

\bigskip
Since $T$ has no boundary we can 
define the dual space $H^{-1}(T) := (H^{1}(T))^{*}$; 
we leave the definition of $H^{-1}(M)$ to
the next section. By Plancherel's theorem, from
Proposition \ref{propn: Fourier_cylinder},
we see that if $s$ is a nonnegative integer and 
$f \in \cS(T)$, then
\[
\|f\|_{H^{s}(T)}^{2} 
\simeq \sum_{|\alpha| \leq s}\|D^{\alpha}f\|_{L^{2}(T)}^{2}
\simeq \sum_{k \in \bZ^{d}}\int_{\bR}
\lan \xi,k\ran^{2s}|\what{f}_{k}(\xi)|^{2}d\xi.
\]

This allows to extend the definition of the spaces $H^{s}(T)$
to any $s \in \bR$. Moreover, observe that this extension coincides
also with the previous definition of $H^{-1}(T)$. Analogous 
extensions can also be defined for $\bR$ and $\bT^{d}$.

\bigskip
On the boundary $\pa M$, we consider the usual $L^{2}$ space 
$L^{2}(\pa M, d\sigma)$ and its corresponding Sobolev
spaces $H^{s}(\pa M)$; we elaborate more on the Sobolev
spaces in the next section. We also define the Sobolev subspaces
\[
H^{s}_{loc}(T) := \{f : f \in H^{s}([-R,R]\times \bT^{d}) \ 
\mbox{for any} \ R > 0\},
\] 
\[
H^{s}_{c}(T) := \{f \in H^{s}(T) : f(x_{1},x') = 0 \ 
\mbox{when} \ |x_{1}|\geq R \ \mbox{for some} \ R > 0\},
\]

and its analogs over $\bR$.
For $\delta \in \bR$ we define the $L^{2}$ 
weighted spaces $L^{2}_{\delta}(T) 
:= \{f : \lan x_{1}\ran^{\delta}f \in L^{2}(T)\}$,
with the norm $\|f\|_{L^{2}_{\delta}(T)} 
:= \|\lan x_{1}\ran^{\delta}f\|_{L^{2}(T)}$. Similarly, 
we also define $L^{2}_{\delta}(\bR)$. It follows from
Proposition \ref{propn: Fourier_torus} that we also have 
Plancherel's identity for weighted spaces,
\begin{equation}
\label{eqn: weighted_Plancherel}
\|f\|_{L^{2}_{\delta}(T)}
= \int_{T}\lan x_{1}\ran^{2\delta}|f(x_{1},x')|^{2}dx_{1}dx'
= \int_{\bR}\lan x_{1}\ran^{2\delta}
\sum_{k \in \bZ^{d}}|f_{k}(x_{1})|^{2}dx_{1}
= \sum_{k \in \bZ^{d}}\|f_{k}\|_{L^{2}_{\delta}(\bR)}^{2}.
\end{equation}

For a nonnegative integer $s$ the weighted 
Sobolev spaces have two equivalent definitions,
\[
H^{s}_{\delta}(T) 
:= \{f \in L^{2}_{\delta}(T) : D^{\alpha}f \in L^{2}_{\delta}(T) \ 
\mbox{for} \ |\alpha| \leq s\}
= \{f \in L^{2}_{\delta}(T) : \lan x_{1}\ran^{\delta}f \in H^{s}(T)\}.
\]

We consider the norm $\|f\|_{H^{s}_{\delta}(T)}$ as any 
of the two equivalent norms: 
$\sum_{|\alpha| \leq s}\|D^{\alpha}f\|_{L^{2}_{\delta}(T)}$
or $\|\lan x_{1}\ran^{\delta}f\|_{H^{s}(T)}$. We also consider
the analogs of these spaces over $\bR$. By 
\eqref{eqn: weighted_Plancherel} we get that 
\[
\|f\|_{H^{s}_{\delta}(T)}^{2}
\simeq \sum_{|\alpha| \leq s}
\|D^{\alpha}f\|_{L^{2}_{\delta}(T)}^{2}
\simeq \sum_{m = 0}^{s}\sum_{k \in \bZ^{d}}
\lan k\ran^{2s - 2m}
\|D_{x_{1}}^{m}f_{k}\|_{L^{2}_{\delta}(T)}^{2}.
\]

We can endow the space $H^{s}_{\delta}(T)$ with another 
norm by considering a (small) real parameter $\hb$ and 
definining 
\[
\|f\|_{H^{s}_{\delta,\hb}(T)} 
:= \sum_{|\alpha| \leq s}\|(\hb D)^{\alpha}f\|_{L^{2}_{\delta}(T)}
\simeq \biggl(\sum_{m = 0}^{s}\sum_{k \in \bZ^{d}}
\lan \hb k\ran^{2s - 2m}
\|(\hb D_{x_{1}})^{m}f_{k}\|_{L^{2}_{\delta}(T)}^{2}\biggr)^{1/2}.
\]

We call this the \textit{semiclassical} weighted Sobolev space. 
Analogously, we define the semiclassical Sobolev spaces 
$W^{s,p}_{\hb}(T)$ and their norm. 


\subsection{Dirichlet problem: definitions and basic facts}


In this section we introduce the necessary definitions
give a precise formulation of the Dirichlet problem 
\begin{equation*}
\tag{\ref{defn: Dirichlet}}
\biggl\{
\begin{array}{rll}
H_{V,W}u \hspace{-2mm} & = 0 & \ \textrm{in} \ \ M_{-}, \\
u \hspace{-2mm} & = f & \ \textrm{on} \ \ \pa M,
\end{array}
\end{equation*}

where $H_{V,W} = (D + V)^{2} + W = D^{2} + 2V\cdot D
+ (V^{2} + D\cdot V + W)$.


\subsubsection{Trace operators and Sobolev spaces}


We call the trace operator that which restricts functions in the 
cylinder $T$ to its boundary values in $\pa M$, and we denote 
it by $\tr$. For $s > 1/2$, the operator 
$\tr: H^{s}(T) \ra H^{s - 1/2}(\pa M)$ 
is continuous. For functions defined only either in the interior or 
exterior of $M$, denoted by $M_{-}$ and $M_{+}$ respectively, 
there are also the operators 
$\tr^{\pm}: H^{s}(M_{\pm}) \ra H^{s - 1/2}(\pa M)$ 
for $s > 1/2$. 

\bigskip
On the boundary we define the dual space 
$H^{-1/2}(\pa M) = (H^{1/2}(\pa M))^{*}$.
The continuity of the trace operator $\tr : H^{1}(T) \ra H^{1/2}(\pa M)$
gives the existence
of the adjoint operator $\tr^{*}: H^{-1/2}(\pa M) \ra H^{-1}(T)$. 
If $\varphi \in H^{1}(T)$ is supported away from $\pa M$, then  
$\tr(\varphi) = 0$; this implies that the adjoint $\tr^{*}$ 
actually maps $H^{-1/2}(\pa M)$ into $H^{-1}_{c}(T)$. The 
adjoint is supported on $\pa M$ and, formally, we have 
$\tr^{*}\varphi = \varphi d\sigma$.

\bigskip
We also consider the space
$H^{1}_{0}(M) := \{u \in H^{1}(M) : \tr^{-}(u) = 0\}$
and its dual $H^{-1}(M) := (H^{1}_{0}(M))^{*}$. The 
space $H^{1}_{0}(M)$ is also the closure of 
$C^{\infty}_{c}(M_{-})$ under the $H^{1}(M)$ norm.


\subsubsection{Weak solutions}


The definitions below of weak solution and Dirichlet-to-Neumann
map are natural after we formally integrate by parts, 
\begin{align*}
\int_{M}(H_{V,W}u)v 
& = \int_{M}D\cdot[(D + V)u]v + V\cdot[(D + V)u]v + Wuv \\
& = \int_{M}-[(D + V)u]\cdot Dv + V\cdot [(D + V)u]v
+ Wuv + \frac{1}{2\pi i}\int_{\pa M}\nu \cdot[(D + V)u]v \\
& = \int_{M} -Du\cdot Dv + V\cdot(vDu - uDv) + (V^{2} + W)uv
- \frac{i}{2\pi}\int_{\pa M}\nu \cdot[(D + V)u]v.
\end{align*}

For $f \in H^{1/2}(\pa M)$, we say that $u \in H^{1}(M)$ is a weak solution 
to the Dirichlet problem \eqref{defn: Dirichlet} if $\tr^{-}(u) = f$ and	
\begin{equation*}
\tag{\ref{defn: magnetic_weak}}
\int_{M}-Du\cdot D\varphi + V\cdot (\varphi Du - u D\varphi) 
+ (V^{2} + W)u\varphi = 0,
\end{equation*}

for all test functions $\varphi \in H^{1}_{0}(M)$. 


\subsubsection{Inhomogeneous problem and extension operator}


The first step towards solving the Dirichlet problem \eqref{defn: Dirichlet} is the 
solution to the inhomogeneous boundary value problem
\begin{equation*}
\label{defn: inhomogeneous}
\tag{$\ast\ast$}
D^{2}u = f \in H^{-1}(M), \ \ u \in H^{1}_{0}(M).
\end{equation*}
 
We say that $u$ is a solution to \eqref{defn: inhomogeneous} 
if for any $\varphi \in H^{1}_{0}(M)$ we have
\[
\lan f, \varphi \ran = \int_{M}-Du\cdot D\varphi.
\] 


\begin{proposition}[\cite{T1}]
\label{propn: Laplacian}
For any $f \in H^{-1}(M)$ there exists a unique solution 
$u \in H^{1}_{0}(M)$ to the boundary value problem 
$D^{2}u = f$. If $Tf := u$ denotes the solution operator, then 
$T: H^{s}(M) \ra H^{s + 2}(M)\cap H^{1}_{0}(M)$ is bounded 
for any $s \geq -1$.
\end{proposition}


In \cite{T1} it is shown an explicit construction
of a bounded extension operator $E : H^{s - 1/2}(\pa M) \ra H^{s}(M)$
for all $s \geq 1$, such that $\tr^{-}\circ E = I$. Moreover, for any 
$N \in \bN$ and $s \leq N$ there is an extension 
$H^{s}(M) \ra H^{s}(T)$ (that may depend on $N$), so that 
we have an extension $E : H^{s - 1/2}(\pa M) \ra H^{s}(T)$ with 
$\tr\circ E = I$; in particular, the trace operator 
$\tr: H^{s}(T) \ra H^{s - 1/2}(\pa M)$ is surjective for $s \geq 1$.
Moreover, by cutting off the extension with an appropriate fixed 
smooth function we can assume that $Ef$ is supported on some 
fixed compact set of $T$, containing $M$, for any $f \in H^{s}(\pa M)$. 


\begin{remark}
We will be concerned with values of $s$ in a fixed range, so we will 
avoid to refer constantly to the integer associated to the extension.
\end{remark}


\subsubsection{Solution to the Dirichlet problem}


The existence of the extension allows to turn the Dirichlet problem
\eqref{defn: Dirichlet} into the boundary value problem 
\eqref{defn: inhomogeneous} for which we know the existence, 
uniqueness, and regularity properties.


\begin{proposition}[\cite{T1}]
\label{propn: existence_uniqueness}
Assume that the potentials 
satisfy $V, W \in L^{\infty}(M)$ and $D\cdot V \in L^{\infty}(M)$.
If $0$ is not a Dirichlet eigenvalue of $H_{V,W}$ in $M$, then 
for any $f \in H^{1/2}(\pa M)$ there exists a unique weak 
$u \in H^{1}(M)$ solution to the Dirichlet problem \eqref{defn: Dirichlet}.
If we denote $D_{V,W}f := u$, then 
$D_{V,W}: H^{s}(\pa M) \ra H^{s + 1/2}(M)$ is bounded for 
$1/2 \leq s \leq 3/2$.
\end{proposition}


\begin{proof}
Under the conditions on the potentials we have that the first order
differential operator
$X := H_{V,W} - D^{2} = 2V\cdot D + (V^{2} + D\cdot V + W)$ 
maps $H^{1}(M)$ into $L^{2}(M)$. We first consider the case 
$f \in H^{1/2}(\pa M)$, so that $Ef \in H^{1}(M)$. Then, 
$u \in H^{1}(M)$ solves the Dirichlet problem \eqref{defn: Dirichlet}
if and only if $v := u - Ef \in H^{1}_{0}(M)$ solves the boundary 
value problem $H_{V,W}v = -H_{V,W}Ef$.
From Proposition \ref{propn: Laplacian} we can look
for a solution of the form $v = Tw$, with $w \in H^{-1}(M)$, 
leaving us to solve the equation
$(I + XT)w = -H_{V,W}Ef \in H^{-1}(M)$. From 
Proposition \ref{propn: Laplacian} and the conditions on the 
potentials we know that the operator 
$XT : H^{-1}(M) \ra L^{2}(M)$ is continuous, and by 
Rellich's theorem we have that $XT$ is a compact
operator on $H^{-1}(M)$. If $0$ is not a Dirichlet eigenvalue
of $H_{V,W}$ in $M$, then the Dirichlet problem 
\eqref{defn: Dirichlet} has at most one solution and therefore
$I + XT$ is injective. It follows then from Fredholm's alternative
that $I + XT$ is bijective, and by the Open Mapping theorem that
its inverse is continuous. Then,
\[
\|v\|_{H^{1}_{0}(M)} \cleq \|w\|_{H^{-1}(M)} 
\cleq \|H_{V,W}Ef\|_{H^{-1}(M)} 
\cleq \|Ef\|_{H^{1}(M)} \cleq \|f\|_{H^{1/2}(\pa M)}.
\]

Therefore, $u = v + Ef \in H^{1}(M)$ and 
$\|u\|_{H^{1}(M)} \cleq \|v\|_{H^{1}(M)} + \|Ef\|_{H^{1}(M)}
\cleq \|f\|_{H^{1/2}(\pa M)}$, as desired.
To prove the higher--order regularity of the solutions, all we 
need to modify in the proof is the fact that for $f \in H^{3/2}(\pa M)$
we have $Ef \in H^{2}(M)$, and therefore $H_{V,W}Ef \in L^{2}(M)$.
The higher--order regularity properties of $T$ from 
Proposition \ref{propn: Laplacian} imply that $T: L^{2}(M) \ra H^{1}(M)$
is compact, and so $XT$ is compact on $L^{2}(M)$. After this the 
proof carries out exactly as before.
\end{proof}


\subsubsection{Dirichlet-to-Neumann map and normal derivatives}


We define the Dirichlet-to-Neumann (DN) map $\Lambda_{V,W}$ as follows: if
$f, g \in H^{1/2}(\pa M)$ and $v \in H^{1}(M)$ is any function extending $g$, 
i.e. $\tr^{-}(v) = g$, then
\begin{equation*}
\tag{\ref{defn: DN_magnetic}}
\lan \Lambda_{V,W}f,g \ran 
:= \int_{M}-Du\cdot Dv + V\cdot (vDu  - uDv) + (V^{2} + W)uv,
\end{equation*}

where $u = D_{V,W}f \in H^{1}(M)$ is the weak solution of 
\eqref{defn: Dirichlet}. The definition of the weak solution implies 
that the DN map is well-defined, i.e. it depends only on $g$ an not
on the choice of extension. Formally we have that 
\[
\Lambda_{V,W}f = \frac{i}{2\pi}\nu \cdot (D + V)u\bigl|_{\pa M}.
\]

Before proving the boundedness properties of the DN map 
we record a Green identity that will be useful now and 
in Chapter 5. This is just slightly more general than saying
that the divergence theorem holds for vector fields in 
$W^{1,1}(M)$.


\begin{proposition}
\label{propn: Green}
If $\supp(V) \sse M_{-}$, 
$V \in L^{\infty}(M)$, $D\cdot V \in L^{\infty}(M)$,
$w \in W^{1,1}(M)$, then 
\[
\int_{M}V\cdot Dw + (D\cdot V)w = 0.
\]
\end{proposition}


\begin{proof}
This proof is taken from \cite{Sa1}, Lemma 5.2. Let $N = d + 1$. Since 
$L^{\infty}(M)$ does not have good approximation properties, we start 
proving it for $V \in L^{N}(M)$, $D\cdot V \in L^{N/2}(M)$,
$w = W^{1,N/(N - 1)}(M)$. From the Sobolev embedding we have that 
$w \in W^{N/(N - 2)}(M)$, so that the integral is in fact convergent. 

\bigskip
Given that $\supp(V) \sse M_{-}$ we can find a compact set 
$K \sse M_{-}$
and smooth functions $\{V_{k}\}$ such that $\supp(V_{k}) \sse K$, 
$V_{k} \ra V$ in $L^{N}(M)$ and $D\cdot V_{k} \ra D\cdot V$ in 
$L^{N/2}(M)$. Moreover, $V_{k}w \in W^{1,N/(N - 1)}(M)$
and $\supp(V_{k}w) \sse K$. The divergence theorem holds for
vector fields in $W^{1,1}(M)$, and so we get
\begin{align*}
\int_{M}V\cdot Dw + (D\cdot V)w
& = \lim_{k \ra +\infty}\int_{M}V_{k}\cdot Dw + (D\cdot V_{k})w \\
& = \lim_{k \ra +\infty}\int_{M}D\cdot (V_{k}w)
= \lim_{k \ra +\infty}\frac{1}{2\pi i}\int_{\pa M}\nu\cdot(V_{k}w) 
= 0.
\end{align*}

The conditions $V \in L^{N}(M)$ and $D\cdot V \in L^{N/2}(M)$ are
satisfied if we assume 
$V \in L^{\infty}(M)$ and $D\cdot V \in L^{\infty}(M)$. Finally, the
integral only takes place in $\supp(V) \sse M_{-}$. We know that
there exist smooth functions $\{w_{k}\}$ such that 
$w_{k} \ra w$ in $L^{1}(M)$ and $Dw_{k} \ra Dw$ in
$L^{1}(\supp(V))$, and thus the conclusion follows.
\end{proof}


\begin{remark}
The condition $\supp(V) \sse M_{-}$ is not necessary;
in \cite{Sa1} this is proven under weaker conditions whose 
analogs would be $\supp(V) \sse M$ and 
$D\cdot V \in L^{\infty}(T)$.
\end{remark}


Before we continue, we need to define the interior and exterior
normal derivative of a function. This represents no problem 
if the function $u$ is in $H^{2}(M)$ or $H^{2}_{loc}(M_{+})$,
as the gradient $Du$ is in $H^{1}(M)$ or $H^{1}_{loc}(M_{+})$
and so its trace is in $H^{1/2}(\pa M)$. Moreover, for 
$\varphi \in C^{\infty}_{c}(T)$ it satisfies either 
\[
\int_{\pa M}(\pa_{\nu}^{\pm}u)\varphi 
= \mp 4\pi^{2}\int_{M_{\pm}}(D^{2}u)\varphi 
+ Du\cdot D\varphi.
\]

These identities suggest that we can define the normal derivatives 
for harmonic functions in $H^{1}(M)$ or $H^{1}_{loc}(M_{+})$. 
We say $u$, in $H^{1}(M)$ or $H^{1}_{loc}(M_{+})$, is harmonic 
if for any $\varphi \in C^{\infty}_{c}(M_{\pm})$ we have 
\begin{equation}
\label{defn: harmonic}
\int_{M_{\pm}}Du\cdot D\varphi = 0,
\end{equation}

as it corresponds. By continuity these definitions extend to 
all test functions $\varphi \in H^{1}(M_{\pm})$ with 
$\tr^{\pm}(\varphi) = 0$.
If $f \in H^{1/2}(\pa M)$ and $v \in H^{1}(M_{\pm})$ is any 
function extending $f$, i.e. $\tr^{\pm}(v) = f$, then we define 
the normal derivatives as the functionals
\begin{equation}
\label{defn: normal_derivatives}
\lan \pa_{\nu}^{\pm}u,f \ran
:= \mp 4\pi^{2}\int_{M_{\pm}}Du\cdot Dv.
\end{equation}

The condition \eqref{defn: harmonic} ensures that 
this is well-defined, i.e. it depends only on $f$ and not 
on the choice of the extension.
In particular, taking $v = Ef \in H^{1}_{c}(T)$ and using the 
boundedness and support properties of $Ef$ we can conclude
that $\pa_{\nu}^{\pm}u \in H^{-1/2}(\pa M)$.


\begin{proposition}
\label{propn: DN_map}
Assume that the potentials 
satisfy $V, W \in L^{\infty}(M)$ and $D\cdot V \in L^{\infty}(M)$.
Suppose in addition that $\supp(V) \sse M_{-}$.
If $0$ is not a Dirichlet eigenvalue of $H_{V,W}$ in $M$, then the
Dirichlet-to-Neumann map \\ 
$\Lambda_{V,W}: H^{s}(\pa M) \ra H^{s - 1}(\pa M)$ is bounded
for $1/2 \leq s \leq 3/2$. Moreover, if $f \in H^{3/2}(\pa M)$ and 
$u = D_{V,W}f \in H^{2}(M)$, then we have 
\[
\Lambda_{V,W}f = \frac{1}{4\pi^{2}}\pa_{\nu}^{-}u\bigl|_{\pa M}.
\]
\end{proposition}


\begin{proof}
We first prove that 
$\Lambda_{V,W}: H^{1/2}(\pa M) \ra H^{-1/2}(\pa M)$
is bounded. If $f,g \in H^{1/2}(\pa M)$, then we have to 
show that 
$|\lan \Lambda_{V,W}f,g\ran| 
\cleq \|f\|_{H^{1/2}(\pa M)}\|g\|_{H^{1/2}(\pa M)}$.
For $u,v \in H^{1}(M)$ we have that
\[
\biggl|\int_{M}-Du\cdot Dv 
+ V\cdot (vDu - uDv) + (V^{2} + W)uv\biggr|
\cleq \|u\|_{H^{1}(M)}\|v\|_{H^{1}(M)}.
\]

In particular, taking $u = D_{V,W}f \in H^{1}(M)$ and 
$v = Eg \in H^{1}(M)$, we conclude from 
\eqref{defn: DN_magnetic}
that
\[
|\lan \Lambda_{V,W}f, g\ran|  
\cleq \|u\|_{H^{1}(M)}\|v\|_{H^{1}(M)}
\cleq \|f\|_{H^{1/2}(\pa M)}\|g\|_{H^{1/2}(\pa M)},
\]

where we used in the last inequality the boundedness of 
$D_{V,W}$ and $E$.

\bigskip
Now we prove the result when $f \in H^{3/2}(\pa M)$. 
Let $g, v$ be as before. From 
Proposition \ref{propn: existence_uniqueness}
we have that 
$u = D_{V,W}f \in H^{2}(M)$, and so 
$\pa_{\nu}^{-}u \in H^{1/2}(\pa M)$. Moreover, 
we can integrate by parts to obtain 
\[
\int_{M}(D^{2}u)v = \int_{M}-Du\cdot Dv 
- \frac{1}{4\pi^{2}}\int_{\pa M}(\pa_{\nu}^{-}u)g 
\]

In addition, for $u \in H^{2}(M)$, $v \in H^{1}(M)$ we have that 
$uv \in W^{1,1}(M)$, so that we obtain 
$\int_{M}D\cdot (Vu)v = - \int_{M} V\cdot(uDv)$.
from Proposition \ref{propn: Green}.
From the previous identities and $H_{V,W}u = 0$ we get that
\begin{align*}
0 & = \int_{M}D\cdot [(D + V)u]v + V\cdot [(D + V)u]v + Wuv \\
& = \int_{M}-Du\cdot Dv + V\cdot(vDu - uDv) + (V^{2}+W)uv
- \frac{1}{4\pi^{2}}\int_{\pa M}(\pa_{\nu}^{-}u)g,
\end{align*}

i.e. $\Lambda_{V,W}f = \pa_{\nu}^{-}u/4\pi^{2}$, and 
$\|\Lambda_{V,W}f\|_{H^{1/2}(\pa M)}
\cleq \|Du\|_{H^{1}(M)}
\cleq \|u\|_{H^{2}(M)}
\cleq \|f\|_{H^{3/2}(\pa M)}$,
as we wanted to prove.
\end{proof}


An important application of the previous theorem is the case of the 
Laplacian $H_{0,0} = D^{2}$. We know that $0$ is not a Dirichlet 
eigenvalue of the Laplacian in $M$, and so we have the DN map
$\Lambda_{0,0}$ defined by
\begin{equation}
\label{eqn: DN_free}
\lan \Lambda_{0,0}f,g \ran := \int_{M}-Du\cdot Dv,
\end{equation}

where $u = D_{0,0}f \in H^{1}(M)$ and $v \in H^{1}(M)$ is any 
function extending $g \in H^{1/2}(\pa M)$. We will 
not use the result for $s > 3/2$, but it can be shown that for 
$s \geq 1/2$, the map $\Lambda_{0,0}: H^{s}(\pa M) \ra H^{s - 1}(\pa M)$ 
is bounded. Moreover, the symmetry in \eqref{eqn: DN_free} implies 
the symmetry of the DN map, i.e. we have 
$\lan \Lambda_{0,0}f,g \ran = \lan \Lambda_{0,0}g,f \ran$ for 
$f,g \in H^{1/2}(\pa M)$. 


\section{Semiclassical pseudodifferential operators over $\bR \times \bT^{\lowercase{d}}$}


We denote the points in the cylinder $T = \bR \times \bT^{d}$ by $(x_{1},x')$, 
meaning that $x_{1} \in \bR$ and $x' \in \bT^{d}$. As it has been usual in the 
inverse problem literature, instead of the large parameter $\tau$ (appearing
in the Carleman estimate) we consider a small parameter $\hb = 1/\tau > 0$, 
and use the standard notation and results from semiclassical analysis. We use
the notation $\hb$ instead of $h$ to prevent confusion with the later use of 
$h$ for a harmonic function.

\bigskip
In this section, we define and prove the necessary results for pseudodifferential operators 
on the cylinder $T = \bR \times \bT^{d}$. We will use the definition and basic properties of 
these operators on $\bR$ and $\bT^{d}$, for which we refer to \cite{St}, 
\cite{Z}, \cite{Sa1}, \cite{So}, \cite{RT}.

\bigskip
Some of the results below may be valid in greater generality than that 
we consider here. We will restrict to prove the results that we will need.


\subsection{Definitions and elementary properties}


\subsubsection{Semiclassical Fourier transform}


We will use the ideas from semiclassical analysis  only for the real 
variable $x_{1}$, as the term $\tau x_{1} = x_{1}/\hb$ appears 
in the limiting Carleman weight, and expressions of the form 
\[
e^{2\pi \tau x_{1}}D_{x_{1}}e^{-2\pi \tau x_{1}} 
= D_{x_{1}} + i\tau = \tau(\hb D_{x_{1}} + i)
\]

will continue to appear through the problem. For this reason, 
throughout the present chapter, we define the semiclassical 
Fourier transform, for functions in $L^{1}(\bR)$, by
\[
\what{f^{\hb}}(\xi) := \int_{\bR}e^{-2\pi ix_{1}\xi/\hb}f(x_{1})dx_{1},
\]

i.e. $\what{f^{\hb}}(\hb \xi) = \what{f}(\xi)$. We can rewrite 
the results from Proposition \ref{propn: Fourier_cylinder} as follows.


\begin{proposition}
\label{propn: Fourier_semiclassical}
If $f \in \cS(T)$, then its Fourier coefficients $f_{k}$ are in $\cS(\bR)$. 
Moreover, these satisfy the following:
\begin{enumerate}[label=\alph*).]
\item the transform of the coefficients and its derivatives have polynomial 
decay bounds 
\[
|(\hb D_{\xi})^{\alpha}\what{f_{k}^{\hb}}(\xi)| 
\cleq \frac{\|x_{1}^{\alpha}f\|_{W_{\hb}^{2m,1}(T)}}{\lan \xi, \hb k \ran^{2m}},
\]
where $\lan \xi, \hb k\ran := (1 + \xi^{2} + |\hb k|^{2})^{1/2}$ and the 
constant of the inequality may depend on $m$ and $d$,
\item the inversion formula holds,
\[
f(x_{1},x') = \sum_{k \in \bZ^{d}}f_{k}(x_{1})e_{k}(x')
= \frac{1}{\hb}\sum_{k \in \bZ^{d}}\int_{\bR}
e^{2\pi ix_{1}\xi/\hb}e_{k}(x')\what{f_{k}^{\hb}}(\xi)d\xi,
\]
with pointwise absolute uniform convergence, as well as for its derivatives,
\item Plancherel's identity holds 
$\|f\|_{L^{2}(T)}^{2} = \sum_{k \in \bZ^{d}}\|f_{k}\|_{L^{2}(\bR)}^{2} 
= \hb^{-1}\sum_{k \in \bZ^{d}}\|\what{f_{k}^{\hb}}\|_{L^{2}(\bR)}^{2}$. 
\end{enumerate}
\end{proposition}


\subsubsection{Semiclassical pseudodifferential operators}


For the differential operator 
$a_{\alpha,\beta}(x_{1},x')(\hb D_{x_{1}})^{\alpha}(\hb D_{x'})^{\beta}$ 
on $T$ and $f \in \cS(T)$ we have the Fourier inversion relation								
\[
[a_{\alpha,\beta}(x_{1},x')(\hb D_{x_{1}})^{\alpha}(\hb D_{x'})^{\beta}]f(x_{1},x')
= \frac{1}{\hb}\sum_{k \in \bZ^{d}}\int_{\bR}e^{2\pi ix_{1}\xi/\hb}e_{k}(x')
[a_{\alpha,\beta}(x_{1},x')\xi^{\alpha}(\hb k)^{\beta}]\what{f_{k}^{\hb}}(\xi)d\xi.
\]

We refer to the function 
$a(x_{1},x',\xi,k) 
= a_{\alpha,\beta}(x_{1},x')\xi^{\alpha}(\hb k)^{\beta}$
as the symbol of the differential operator. In what follows we show 
that we can admit symbols more general than polynomials (in the dual 
variables $\xi$ and $k$). Finally, although we only need to define the 
symbol over $\bR \times \bT^{d} \times \bR \times \bZ^{d}$, it may 
be convenient also to allow for symbols over 
$\bR \times \bT^{d} \times \bR \times \bR^{d}$. We denote the points 
in $\bR \times \bT^{d} \times \bR \times \bR^{d}$ by $(x_{1},x',\xi,t)$, 
and we call $\xi$ and $t$ the dual real and toroidal variables, respectively. 


\begin{definition}
\label{defn: symbol_R}
We say that $a = a(x_{1}, \xi; \hb)$ is a (semiclassical) $m$-th order 
symbol over $\bR \times \bR$ if there exists $\hb_{0}$ such that if 
$0 < \hb  \leq \hb_{0}$, then for any
$M \geq 0$ there exists a constant $A_{M}$ such that
\[
|D_{x_{1}}^{\alpha}D_{\xi}^{\beta}a(x_{1},\xi; \hb)| 
\leq A_{M}\lan \xi \ran^{m},
\]

whenever $\alpha + |\beta| \leq M$. The associated pseudodifferential 
operator is defined by
\[
Af(x_{1}) := \Op_{\hb}(a)f(x_{1}) 
= \frac{1}{\hb}\int_{\bR}e^{2\pi ix_{1}\xi/\hb}
a(x_{1},\xi; \hb)\what{g^{\hb}}(\xi)d\xi.
\]
\end{definition}


\begin{definition}
\label{defn: symbol}
We say that $a = a(x_{1}, x', \xi, t; \hb)$ is a (semiclassical) $m$-th order 
symbol over $\bR \times \bT^{d} \times \bR \times \bR^{d}$ if there exists 
$\hb_{0}$ such that if $0 < \hb \leq \hb_{0}$, then for any
$M \geq 0$ there exists a constant $A_{M}$ such that
\[
|D_{x_{1}}^{\alpha}D_{x'}^{\beta}D_{\xi}^{\gamma}a(x_{1},x',\xi,t; \hb)| 
\leq A_{M}\lan \xi, \hb t\ran^{m},
\]

whenever $\alpha + |\beta|+ \gamma \leq M$. The associated pseudodifferential 
operator is defined by
\[
Af(x_{1},x') := \Op_{\hb}(a)f(x_{1},x') 
:= \frac{1}{\hb}\sum_{k \in \bZ^{d}}\int_{\bR}e^{2\pi ix_{1}\xi/\hb}e_{k}(x')
a(x_{1},x', \xi, k; \hb)\what{f_{k}^{\hb}}(\xi)d\xi.
\]
\end{definition}


\begin{remark}
Observe that we do not require the order of the factor $\lan \xi\ran$ or 
$\lan \xi, \hb t\ran$ to decrease whenever we differentiate with respect 
to $\xi$. This would be the case if the symbol were a polynomial or a 
rational function, but we will be considering more general symbols.
In the notation of \cite{St}, these would correspond to symbols
in $S^{m}_{0,0}$.
\end{remark}


\begin{remark}
Note that we do not require any condition on the differences (or derivatives)
with respect to the dual toroidal variables. In a later section, \textit{Composition}, 
we will need these symbols and refer to them as special.
\end{remark}


\begin{remark}
To avoid unnecessary notation, we may occasionally drop the dependance
of the symbol on the semiclassical parameter and just write $a(x_{1},x',\xi,t)$.
\end{remark}


\begin{example}
With this definition, the functions $\xi$ and $\hb t_{j}$ are 
symbols of order $1$. Moreover, we have that 
$\hb D_{x_{1}} = \Op_{\hb}(\xi)$ and 
$\hb D_{x'_{j}} = \Op_{\hb}(\hb t_{j})$ 
as
\[
(\hb D_{x_{1}})f(x_{1},x') = \frac{1}{\hb}\sum_{k \in \bZ^{d}}\int_{\bR}
e^{2\pi ix_{1}\xi/\hb}e_{k}(x')(\xi)\what{f_{k}^{\hb}}(\xi)d\xi,
\]
\[
(\hb D_{x'_{j}})f(x_{1},x') = \frac{1}{\hb}\sum_{k \in \bZ^{d}}\int_{\bR}
e^{2\pi ix_{1}\xi/\hb}e_{k}(x')(\hb k_{j})\what{f_{k}^{\hb}}(\xi)d\xi.
\]
\end{example}


\begin{example}
The function $\lan \xi,\hb t\ran^{-2} := 1/(\xi^{2} + |\hb t|^{2} + 1)$ 
is a symbol of order $-2$. 
\end{example}


\begin{proposition}
\label{propn: operations_symbols}
If $a$, $b$ are symbols of order $m$ and $n$, then 
$D_{x_{1}}^{\alpha}D_{x'}^{\beta}D_{\xi}^{\gamma}a$, $a + b$, and 
$ab$ are symbols of order $m$, $\max\{m,n\}$, and $m + n$, respectively. 
The seminorms of each of these symbols are bounded by those of $a$, 
the maximum of those of $a$ and $b$, and products of those of $a$ and  $b$,
respectively.
\end{proposition}


\begin{proof}
This is a routine argument.
\end{proof}


\begin{proposition}
\label{propn: operator_Schwartz}
If $A = \Op_{\hb}(a)$ is a pseudodifferential operator over $T$, then $A$ maps 
the space of Schwartz functions $\cS(T)$ into itself.
\end{proposition}


\begin{proof}
For this proof we will use the notation $D_{x} = (D_{x_{1}},D_{x'})$.
Let $f \in \cS(T)$. The polynomial control of the symbol $a$ and its derivatives,
together with the rapid decay of $\what{f_{k}^{\hb}}(\xi)$ from 
Proposition \ref{propn: Fourier_semiclassical} give that $Af \in C^{\infty}(T)$, and it
is bounded together with its derivatives. Moreover,
differentiating the expression we see that (the vector) $(\hb D_{x})Af$ equals 
\begin{align*}
\hb D_{x}Af(x_{1},x') & = \frac{1}{\hb}\sum_{k \in \bZ^{d}}\int_{\bR}
\hb D_{x}(e^{2\pi ix_{1}\xi/\hb}e_{k}(x')a)\what{f_{k}^{\hb}}(\xi)d\xi \\
& = \frac{1}{\hb}\sum_{k \in \bZ^{d}}\int_{\bR}e^{2\pi ix_{1}\xi/\hb}e_{k}(x')
[(\xi,\hb k) a + \hb D_{x}a]\what{f_{k}^{\hb}}(\xi)d\xi,
\end{align*}

and so it is a pseudodifferential operator corresponding to the symbol
$(\xi,\hb t)a + \hb D_{x}a$. By induction the same is true for higher order 
derivatives. Therefore, in order to show that $Af \in \cS(T)$, it suffices to 
show that $\lan x_{1}\ran^{2m}|Af| \leq C_{m}$ for all $m \geq 0$. 
Integrating by parts we obtain that
\begin{align*}
\lan x_{1} \ran^{2m}Af(x_{1},x') &= \frac{1}{\hb}\sum_{k \in \bZ^{d}}\int_{\bR}
\lan \hb D_{\xi} \ran^{2m}(e^{2\pi ix_{1}\xi/\hb})e_{k}(x')a\what{f_{k}^{\hb}}(\xi)d\xi \\
& = \frac{1}{\hb}\sum_{k \in \bZ^{d}}\int_{\bR}
e^{2\pi ix_{1}\xi/\hb}e_{k}(x')\lan \hb D_{\xi}\ran^{2m}[a\what{f_{k}^{\hb}}(\xi)]d\xi.
\end{align*}

Again, the polynomial control of the symbol $a$ and its derivatives,
together with the rapid decay of the derivatives of $\what{f_{k}^{\hb}}(\xi)$ 
from Proposition \ref{propn: Fourier_semiclassical} give that this is bounded, from
where the conclusion follows.
\end{proof}


\begin{proposition}
\label{propn: symbol_computations}
Let $A = \Op_{\hb}(a)$ be a pseudodifferential operator over $T$. Then,
it satisfies the following identities,
\[
\hb D_{x_{1}}A = \Op_{\hb}(\xi a + \hb D_{x_{1}}a), \ \
\hb^{2}D_{x_{1}}^{2}A = \Op_{\hb}(\xi^{2}a 
+ 2\hb\xi D_{x_{1}}a + \hb^{2}D_{x_{1}}^{2}a),
\]
\[
\hb D_{x'_{j}}A = \Op_{\hb}(\hb t_{j}a + \hb D_{x'_{j}}a), \ \
\hb^{2}D_{x'}^{2}A = \Op_{\hb}(|\hb t|^{2}a 
+ 2\hb(\hb t\cdot D_{x'}a) + \hb^{2}D_{x'}^{2}a),
\]
\[
A \circ \hb D_{x_{1}} = \Op_{\hb}(\xi a), \ \
A \circ \hb^{2}D_{x_{1}}^{2} = \Op_{\hb}(\xi^{2}a), \ \ 
A \circ \hb^{2}D_{x'}^{2} = \Op_{\hb}(|\hb t|^{2}a).
\]
\end{proposition}


\begin{proof}
These results follow directly from the definition.
\end{proof}


The simplest case when dealing with pseudodifferential operators in $\bR^{d}$, is
when the symbol has spatial compact support, see Chapter 6, Section 2.1 in
\cite{St}. This is always the case for symbols on the torus, so in analogy
to \cite{St}, we decompose the symbol in its Fourier series. With uniform 
convergence (in $x_{1}$, $x'$, $\xi$, and $l$), we have that
$a(x_{1},x',\xi,l; \hb) 
= \sum_{k \in \bZ^{d}} a_{k}(x_{1},\xi,l; \hb)e_{k}(x')$,
so we can rewrite the operator $A = \Op_{\hb}(a)$ as
\begin{align}
\label{eqn: pseudodifferential_decomposition}
Af(x_{1},x') & = \frac{1}{\hb}\sum_{l \in \bZ^{d}}\int_{\bR}e^{2\pi ix_{1}\xi/\hb}e_{l}(x')
a(x_{1},x', \xi, l)\what{f_{l}^{\hb}}(\xi)d\xi \nonumber \\
& = \frac{1}{\hb}\sum_{k,l \in \bZ^{d}}\int_{\bR}e^{2\pi ix_{1}\xi/\hb}e_{k + l}(x')
a_{k}(x_{1}, \xi, l)\what{f_{l}^{\hb}}(\xi)d\xi \nonumber \\
& = \sum_{k,l \in \bZ^{d}}\biggl(\frac{1}{\hb}\int_{\bR}e^{2\pi ix_{1}\xi/\hb}
a_{k - l}(x_{1}, \xi, l)\what{f_{l}^{\hb}}(\xi)d\xi\biggr)e_{k}(x').
\end{align}

If $a$ is a symbol over $\bR \times \bT^{d} \times \bR \times \bR^{d}$, 
then for fixed $k, l \in \bZ^{d}$ we can define a symbol over $\bR \times \bR$ by 
$a_{k,l}(x_{1},\xi; \hb) := a_{k - l}(x_{1},\xi,l; \hb)$. We will elaborate
below on the properties of this symbol. 
Let us define $A^{kl} := \Op_{\hb}(a_{k,l})$ on $\cS(\bR)$. For $f \in \cS(T)$, 
we define $A_{kl}f(x_{1},x') := A^{kl}f_{l}(x_{1})e_{k}(x')$, so that 
\eqref{eqn: pseudodifferential_decomposition} can be expressed as the
decomposition 
\begin{equation}
\label{eqn: pseudodifferential_decomposition_II}
Af(x_{1},x') 
= \sum_{k,l \in \bZ^{d}}A^{kl}f_{l}(x_{1})e_{k}(x')
= \sum_{k,l \in \bZ^{d}} A_{kl}f(x_{1},x').
\end{equation} 


\subsection{Boundedness}


In this section we prove a weighted version of the 
Calder\'on--Vaillancourt theorem for pseudodifferential 
operators over $T$. It is interesting to observe 
that we do not need to control the differences over the dual 
toroidal variables; this had already been noted in \cite{So}, 
\cite{RT}.

\bigskip
Recall from before that for the symbol $a(x_{1},x',\xi,l; \hb)$ over 
$\bR \times \bT^{d} \times \bR \times \bZ^{d}$, we defined the symbol 
$a_{k,l}(x_{1},\xi; \hb) := a_{k - l}(x_{1},\xi,l; \hb)$ over $\bR \times \bR$. 


\begin{proposition}
\label{propn: symbol_decomposition}
If $a(x_{1},x',\xi,l; \hb)$ is a semiclassical zero order symbol over
$\bR \times \bT^{d} \times \bR \times \bZ^{d}$, then
$a_{k,l}(x_{1},\xi; \hb)$ is a semiclassical zero order symbol 
over $\bR \times \bR$ with seminorm bounds
\[
|D_{x_{1}}^{\alpha}D_{\xi}^{\beta}a_{k,l}(x_{1},\xi; \hb)|							
\cleq A_{M + 2N}\lan k - l\ran^{-2N},
\]

whenever $\alpha + \beta \leq M$ and any $N \geq 0$.
\end{proposition}


\begin{proof}
From Proposition \ref{propn: Fourier_torus} we have that
\[
|D_{x_{1}}^{\alpha}D_{\xi}^{\beta}a_{k,l}(x_{1},\xi)|	
= |D_{x_{1}}^{\alpha}D_{\xi}^{\beta}a_{k - l}(x_{1},\xi,l)|	
\cleq \lan k - l\ran^{-2N}\|D_{x_{1}}^{\alpha}D_{\xi}^{\beta}
a(x_{1},\cdot,\xi,l)\|_{W^{2N,1}(\bT^{d})}.
\]

Given that $a$ is a zero order symbol, then for any $N \geq 0$
we can bound 
$\|D_{x_{1}}^{\alpha}D_{\xi}^{\beta}
a(x_{1},\cdot,\xi,l)\|_{W^{2N,1}(\bT^{d})}
\cleq A_{M + 2N}$, whenever $\alpha + \beta \leq M$, as 
we wanted to prove.
\end{proof}


We use this to show that the decomposition from 
\eqref{eqn: pseudodifferential_decomposition_II} actually converges. 
The first step is to recall the standard boundedness properties 
of pseudodifferential operators on weighted spaces over $\bR$. We will 
elaborate a little more on the quantitative aspect of the bound in the 
appendix at the end of the chapter. 


\begin{proposition}[\cite{Sa1}]
\label{propn: zero_boundedness_R}
Let $0 < \hb \leq 1$. Let $a(x_{1},\xi; \hb)$ be a semiclassical zero order 
symbol over $\bR \times \bR$. For any $\delta \in \bR$ the operator 
$\Op_{\hb}(a)$ is bounded in $L^{2}_{\delta}(\bR)$. 
Moreover, if $|\delta| \leq \delta_{0}$, then 
the operator norms 
$\|\Op_{\hb}(a)\|_{L^{2}_{\delta}(\bR) \ra L^{2}_{\delta}(\bR)}$ 
are uniformly bounded (in $\delta$ and $\hb$) by a multiple (depending 
on $\delta_{0}$) of some seminorm of $a$.
\end{proposition} 


\begin{remark}
From Proposition \ref{propn: symbol_decomposition} and 
Proposition \ref{propn: zero_boundedness_R} we obtain 
that if $|\delta| \leq \delta_{0}$, then we can uniformly bound 
$\|A^{kl}\|_{L^{2}_{\delta}(\bR) \ra L^{2}_{\delta}(\bR)} 
\cleq A_{M + 2N}\lan k - l\ran^{-2N}$, for some value of $M$ and
any $N \geq 0$.
\end{remark}


Let us consider the elliptic differential operator 
$\lan \hb D\ran^{2} = (\hb D)^{2} + 1 = \Op_{\hb}(\lan \xi, \hb t\ran^{2})$,
and the multiplier operator $\lan \hb D\ran^{-2} := \Op_{\hb}(\lan \xi, \hb t\ran^{-2})$.
These are pseudodifferential operators, so by 
Proposition \ref{propn: operator_Schwartz} they map $\cS(T)$ to itself.
Moreover, these are inverses to each other on $\cS(T)$. The proof the 
following result is presented in the appendix.


\begin{proposition}
\label{propn: isomorphisms}
Let $|\delta| \leq \delta_{0}$ and let $0 < \hb \leq 1$. The differential 
operator $\lan \hb D\ran^{2} : H^{2}_{\delta,\hb}(T) \ra L^{2}_{\delta}(T)$ 
and the multiplier operator 
$\lan \hb D\ran^{-2} : L^{2}_{\delta}(T) \ra H^{2}_{\delta,\hb}(T)$ are 
uniformly bounded (in $\delta$ and $\hb$) operators, and inverses to each 
other. The bounds of the operators may depend in $\delta_{0}$.
\end{proposition}


\begin{proposition}
\label{propn: conjugation}
Let $0 < \hb \leq 1$. If $A = \Op_{\hb}(a)$ is an $m$-th order 								
pseudodifferential operator, then $\lan \hb D\ran^{2}A\lan \hb D\ran^{-2}$ 
is also an $m$-th order pseudodifferential operator with symbol 							
\[
\wilde{a} := a + \frac{2\hb(\xi, \hb t)\cdot (D_{x_{1}}a,D_{x'}a) 
+ \hb^{2}(D_{x_{1}}^{2}a + D_{x'}^{2}a)}{\lan \xi, \hb t\ran^{2}}.
\]

Moreover, the seminorms of $\wilde{a}$ are bounded by seminorms 
of $a$.
\end{proposition}


\begin{proof}
The first part is a direct computation, 
\begin{align*}
\lan \hb D\ran^{2}A& \lan \hb D\ran^{-2}f(x_{1},x') \\
& = \lan \hb D\ran^{2}\biggl(\frac{1}{\hb}\sum_{k \in \bZ^{d}}\int_{\bR}
e^{2\pi ix_{1}\xi/\hb}e_{k}(x')a
\frac{\what{f_{k}^{\hb}}(\xi)}{\lan \xi, \hb k\ran^{2}}d\xi\biggr) \\
& = \frac{1}{\hb}\sum_{k \in \bZ^{d}}\int_{\bR}\lan \hb D\ran^{2}
(e^{2\pi ix_{1}\xi/\hb}e_{k}(x')a)
\frac{\what{f_{k}^{\hb}}(\xi)}{\lan \xi, \hb k\ran^{2}}d\xi \\
& = \frac{1}{\hb}\sum_{k \in \bZ^{d}}\int_{\bR}
e^{2\pi ix_{1}\xi/\hb}e_{k}(x')\biggl(a 
+ \frac{2\hb(\xi, \hb k)\cdot (D_{x_{1}}a,D_{x'}a) 
+ \hb^{2}(D_{x_{1}}^{2}a + D_{x'}^{2}a)}{\lan \xi, \hb k\ran^{2}}\biggr)
\what{f_{k}^{\hb}}(\xi)d\xi.
\end{align*}

The symbols $(\xi,\hb t)/(\xi^{2} + |\hb t|^{2} + 1)$ and
$1/(\xi^{2} + |\hb t|^{2} + 1)$ have order $-1$ and $-2$, respectively,
with uniformly bounded (in $\hb$) seminorms. The conclusion 
follows then from Proposition \ref{propn: operations_symbols}.
\end{proof}


\begin{theorem}
\label{thm: zero_boundedness_cylinder}
Let $0 < \hb \leq 1$.
Let $a$ be a zero order symbol over 
$\bR \times \bT^{d} \times \bR \times \bR^{d}$. For 
$s = 0,2$ and $|\delta|\leq \delta_{0}$, the operator 
$\Op_{\hb}(a)$ is uniformly bounded (in $\delta$ and $\hb$) on 
$H^{s}_{\delta,\hb}(T)$. The bounds depend on $d$, 
$\delta_{0}$, and (linearly) in some seminorm of the 
symbol, but are independent of $\hb$.
\end{theorem}


\begin{proof}
We prove first the result on the weighted spaces $L^{2}_{\delta}(T)$, 
and then conjugate by $\lan \hb D\ran^{2}$ to show the result in 
$H^{2}_{\delta,\hb}(T)$. Recall the decomposition
$A = \sum_{k,l}A_{kl}$ from
\eqref{eqn: pseudodifferential_decomposition_II}, 
where $A_{kl}f(x_{1},x') := A^{kl}f_{l}(x_{1})e_{k}(x')$. 
Choosing $N = N(d)$ large enough and using the bounds 
for the operators $A^{kl}$, from the remark after 
Proposition \ref{propn: zero_boundedness_R}, we obtain that
\[
\sup_{k \in \bZ^{d}}\sum_{l \in \bZ^{d}}
\|A^{kl}\|_{L^{2}_{\delta}(\bR) \ra L^{2}_{\delta}(\bR)}, \ \
\sup_{l \in \bZ^{d}}\sum_{k \in \bZ^{d}}
\|A^{kl}\|_{L^{2}_{\delta}(\bR) \ra L^{2}_{\delta}(\bR)}
\cleq A_{M + 2N}.
\]

By Plancherel's theorem and Schur's criterion it follows that 
\begin{align*}
\|Af\|_{L^{2}_{\delta}(T)}^{2}
& = \sum_{k \in \bZ^{d}}\biggl\|\sum_{l \in \bZ^{d}}
A^{kl}f_{l}\biggr\|_{L^{2}_{\delta}(\bR)}^{2} \\ 
& \leq \sum_{k \in \bZ^{d}}\biggl(\sum_{l \in \bZ^{d}}
\|A^{kl}\|_{L^{2}_{\delta}(\bR) \ra L^{2}_{\delta}(\bR)} 
\|f_{l}\|_{L^{2}_{\delta}(\bR)}\biggr)^{2} 
\cleq A_{M + 2N}^{2}\sum_{l \in \bZ^{d}}
\|f_{l}\|_{L^{2}_{\delta}(\bR)}^{2}
= A_{M + 2N}^{2}\|f\|_{L^{2}_{\delta}(T)}^{2}.
\end{align*}

This gives that $A$ is a bounded operator on $L^{2}_{\delta}(T)$ for 
$|\delta| \leq \delta_{0}$. From Proposition \ref{propn: isomorphisms} we get 
that the boundedness of $A$ on $H^{2}_{\delta,\hb}(T)$ is equivalent to 
the boundedness of $\lan \hb D\ran^{2}A\lan \hb D\ran^{-2}$ on 
$L^{2}_{\delta}(T)$. We know from Proposition \ref{propn: conjugation} that 
this is a zero order pseudodifferential operator with seminorms bounded 
by those of $a$, and thus the conclusion follows.
\end{proof}


\begin{remark}
The result proven above will be sufficient for our purposes,
 but the result can be extended to $H^{s}_{\delta,\hb}(T)$ 
for any $0 \leq s \leq 2$ by complex interpolation.
\end{remark}


\subsection{Composition}


Let $a,b$ be zero order symbols, and let $A = \Op_{\hb}(a)$ and $B = \Op_{\hb}(b)$. 
We know from Proposition \ref{propn: operations_symbols} that $ab$ is also a zero order 
symbol. The following result provides a relation 
between the operator $\Op_{\hb}(ab)$ and the composition $AB$. This will be used to 
obtain the invertibility of pseudodifferential operators corresponding to 
certain zero order symbols. In contrast to Theorem \ref{thm: zero_boundedness_cylinder}, 
in this case we need our symbols to satisfy bounds for the differences in the 
dual toroidal variables. For $(u,v,w) \in \bR \times \bR^{d} \times \bR^{d}$ 
we denote $\lan u,v,w\ran := (1 + u^{2} + |v|^{2} + |w|^{2})^{1/2}$.


\begin{definition}
\label{defn: special_symbol}
We say a zero order symbol $a = a(x_{1},x',\xi,t; \hb)$ is special
if, in addition to the conditions from Definition \ref{defn: symbol}, for any 
$M \geq 0$ there exists a constant $A'_{M}$ such that
\[
|D_{x_{1}}^{\alpha}D_{x'}^{\beta}D_{\xi}^{\gamma}
(a(\cdot,t_{1}) - a(\cdot,t_{2}))| 
\leq \hb A'_{M}|t_{1} - t_{2}|,
\]

whenever $\alpha + |\beta|+ \gamma \leq M$.
\end{definition} 


\begin{remark}
If the symbol is differentiable with respect to $t$ and satisfies
\[
|D_{x_{1}}^{\alpha}D_{x'}^{\beta}D_{\xi}^{\gamma}D_{t}a| 
\leq \hb A'_{M},
\]

whenever $\alpha + |\beta|+ \gamma \leq M$, then the symbol is
be special.
\end{remark}


\begin{theorem}
\label{thm: composition}
Let $0 < \hb \leq 1$ and $\delta_{0} \geq 0$.
Let $a,b$ be zero order symbols over 
$\bR \times \bT^{d} \times \bR \times \bR^{d}$. 
There exists a zero order symbol $c$ such that
$\Op_{\hb}(a)\Op_{\hb}(b) = \Op_{\hb}(c)$. If the 
symbols are special, then, for $s = 0,2$ 
and $|\delta|\leq \delta_{0}$, we have 
\[
\|\Op_{\hb}(a)\Op_{\hb}(b) - \Op_{\hb}(ab)\|_{H^{s}_{\delta,\hb}(T) \ra H^{s}_{\delta,\hb}(T)}
= \|\Op_{\hb}(c - ab)\|_{H^{s}_{\delta,\hb}(T) \ra H^{s}_{\delta,\hb}(T)}
\cleq \hb,
\]

where the constant of the inequality is a multiple 
(depending on $d$, $\delta_{0}$) of the product 
of some seminorm of the symbols, but is independent 
of $\hb$.
\end{theorem}


\begin{proof}
Let $A = \Op_{\hb}(a)$ and $B = \Op_{\hb}(b)$. Let us decompose 
$A = \sum A_{jk}$ and $B = \sum B_{lm}$ as in \eqref{eqn: pseudodifferential_decomposition_II}. 
We have that
$A_{jk}B_{lm} = 0$ if $k \neq l$, so that 
\[
ABf(x_{1},x') = \sum_{j,k,l \in \bZ^{d}}A_{jk}B_{kl}f(x_{1},x'),
\]
\vspace{-2mm}

and $A_{jk}B_{kl}f(x_{1},x') = A^{jk}B^{kl}f_{l}(x_{1})e_{j}(x')$.
We know that there exists a zero order symbol $c_{j,k,l}$ over 
$\bR \times \bR$ such that $A^{jk}B^{kl} = \Op_{\hb}(c_{j,k,l})$, 
see \cite{St}, \cite{Z}. Thus, 
\[
ABf(x_{1},x') = \sum_{j,k,l \in \bZ^{d}}A_{jk}B_{kl}f(x_{1},x') 
= \sum_{j,k,l \in \bZ^{d}}A^{jk}B^{kl}f_{l}(x_{1})e_{j}(x') 
= \sum_{j,k,l \in \bZ^{d}}\Op_{\hb}(c_{j,k,l})f_{l}(x_{1})e_{j}(x').
\]

From Proposition \ref{propn: symbol_decomposition} and
Proposition \ref{propn: bounds_symbol_composition}, which we prove in
the appendix, there exists some $K$ such that for
any $N \geq 0$ we have that
\begin{equation}
\label{eqn: symbol_composition}
|D_{x_{1}}^{\alpha}D_{\xi}^{\beta}c_{j,k,l}| 
\cleq A_{K + M + 2N}B_{K + M + 2N}\lan j - k\ran^{-2N}\lan k - l\ran^{-2N},
\end{equation}
\begin{equation}
\label{eqn: difference_composition}
|D_{x_{1}}^{\alpha}D_{\xi}^{\beta}(c_{j,k,l} - a_{j,k}b_{k,l})| 
\cleq \hb A_{K + M + 2N}B_{K + M + 2N}\lan j - k\ran^{-2N}\lan k - l\ran^{-2N},
\end{equation}
 
whenever $\alpha + \beta \leq M$. Recall that for a symbol 
$s$ we denote $s_{k,l}(x_{1},\xi) = s_{k - l}(x_{1},\xi,l)$.
Using this notation and the decomposition from 
\eqref{eqn: pseudodifferential_decomposition_II}, we see 
that if $c$ were a symbol such that $AB = \Op_{\hb}(c)$, then 
we must have $c_{j,l}(x_{1},\xi) = \sum_{k \in \bZ^{d}} c_{j,k,l}(x_{1},\xi)$.
Thus we define
\begin{align*}
c(x_{1},x',\xi,l) & = \sum_{j \in \bZ^{d}}c_{j}(x_{1},\xi,l)e_{j}(x') \\
& = \sum_{j \in \bZ^{d}}c_{j+l,l}(x_{1},\xi)e_{j}(x')
= \sum_{j \in \bZ^{d}}c_{j,l}(x_{1},\xi)e_{j - l}(x')
:=  \sum_{j,k \in \bZ^{d}}c_{j,k,l}(x_{1},\xi)e_{j - l}(x').
\end{align*}

It follows from \eqref{eqn: symbol_composition} that $c$ is a zero
order symbol. Moreover,
\begin{align*}
& c(x_{1}, x',\xi,l) - a(x_{1},x',\xi,l)b(x_{1},x',\xi,l) \\
& = \sum_{j,k \in \bZ^{d}}(c_{j,k,l}(x_{1},\xi)
- a_{j - k}(x_{1},\xi,l)b_{k - l}(x_{1},\xi,l))e_{j - l}(x') \\
& = \sum_{j,k\in \bZ^{d}}(c_{j,k,l}(x_{1},\xi) 
- a_{j,k}(x_{1},\xi)b_{k,l}(x_{1},\xi))e_{j - l}(x') 
+ (a_{j - k}(x_{1},\xi,k) - a_{j - k}(x_{1},\xi,l))b_{k,l}(x_{1},\xi)e_{j - l}(x').
\end{align*}

From \eqref{eqn: difference_composition} and the inequality 
$\lan x + y\ran \cleq \lan x\ran\lan y\ran$, we obtain that
the first difference is a zero order symbol with seminorms 
bounded by (appropriate) multiples of $\hb$. For the second 
difference we use that the symbol is special and 
Proposition \ref{propn: Fourier_torus} (as in the proof of 
Proposition \ref{propn: symbol_decomposition}) to obtain
\[
|D_{x_{1}}^{\alpha}D_{\xi}^{\beta}
(a_{j - k}(x_{1},\xi,k) - a_{j - k}(x_{1},\xi,l))|
\cleq \hb|k - l|A'_{M + 2N}\lan j - k\ran^{-2N},
\]

any $N \geq 0$, whenever $\alpha + \beta \leq M$. Therefore, 
\[
|D_{x_{1}}^{\alpha}D_{\xi}^{\beta}
[(a_{j - k}(x_{1},\xi,k) - a_{j - k}(x_{1},\xi,l))b_{k,l}(x_{1},\xi)]|
\cleq \hb A'_{M + 2N}B_{M + 2N}\lan j -k\ran^{-2N}\lan k - l\ran^{-2N + 1},
\]

for any $N \geq 0$, whenever $\alpha + \beta \leq M$. As 
before, we conclude that the second difference is a zero order 
symbol with seminorms bounded by (appropriate) multiples of 
$\hb$, and the result follows from 
Theorem \ref{thm: zero_boundedness_cylinder}.
\end{proof}


\begin{proposition}
\label{propn: operations_special_symbols}
Let $a$ and $b$ be special zero order symbols over 
$\bR \times \bT^{d} \times \bR \times \bR^{d}$. Then,
\begin{enumerate}[label=\textrm{\alph*).}]
\item $D_{x_{1}}^{\alpha}D_{x'}^{\beta}D_{\xi}^{\gamma}a$, 
$a + b$, and $ab$ are also special zero order symbols.
Moreover, their seminorms are controlled by the products of
the seminorms of $a$ and $b$.
\item the function $e^{a}$ is also a special zero order symbol, 
\item for small enough $\hb$, depending on $a$, the function 
$\log(1 + \hb a)$ is also a special zero order symbol.
\end{enumerate}
\end{proposition}


\begin{proof}
This is a routine argument.
\end{proof}


\begin{corollary}
\label{cor: invertibility}
Let $a$ be a special zero order symbol over 
$\bR \times \bT^{d} \times \bR \times \bR^{d}$ and 
let $\delta_{0} > 0$. Then there exists $\hb_{0} > 0$, 
such that if $0 < \hb \leq \hb_{0}$, then 
$\Op_{\hb}(e^{a})$ is an invertible operator in 
$H^{s}_{\delta,\hb}(T)$ for $|\delta| \leq \delta_{0}$ 
and $s = 0,2$. Moreover, the norms in 
$H^{s}_{\delta,\hb}(T)$ of the operator and its inverse 
are uniformly bounded (in $\delta$ and $\hb$).
\end{corollary}


\begin{proof}
We know from Proposition \ref{propn: operations_special_symbols} that 
$e^{\pm a}$ are zero order symbols. From Theorem \ref{thm: composition} 
we have that
\[
\|\Op_{\hb}(e^{a})\Op_{\hb}(e^{-a}) - I\|_{H^{s}_{\delta,\hb}(T) 
\ra H^{s}_{\delta,\hb}(T)},  
\|\Op_{\hb}(e^{-a})\Op_{\hb}(e^{a}) - I\|_{H^{s}_{\delta,\hb}(T) 
\ra H^{s}_{\delta,\hb}(T)} \cleq \hb.
\]

This implies that $\Op_{\hb}(e^{a})$ has left and right inverses and 
the conclusion follows.
\end{proof}


\subsection{Appendices} 


\subsubsection{Some facts about weighted spaces}


Let us recall the multiplier operator 
$\lan \hb D\ran^{-2} := \Op_{\hb}(\lan \xi, \hb t\ran^{-2})$, i.e.
\begin{align*}
\lan \hb D\ran^{-2}f(x_{1},x') 
& = \frac{1}{\hb}\sum_{k \in \bZ^{d}}\int_{\bR}
e^{2\pi ix_{1}\xi/\hb}e_{k}(x')\frac{1}{\xi^{2} + |\hb k|^{2} + 1}
\what{f_{k}^{\hb}}(\xi)d\xi \\
& = \sum_{k \in \bZ^{d}}\biggl(\frac{1}{\hb}\int_{\bR}
\frac{e^{2\pi ix_{1}\xi/\hb}}{\xi^{2} + |\hb k|^{2} + 1}
\what{f_{k}^{\hb}}(\xi)d\xi\biggr)e_{k}(x')
\end{align*}
 
If $\lambda > 0$, then we have the classical Fourier transform
\[
\int_{\bR}\frac{e^{2\pi ix_{1}\xi}}{\xi^{2} + \lambda^{2}}d\xi
= \frac{\pi}{\lambda}e^{-2\pi\lambda|x_{1}|},
\]
 
so that the multiplier operator is also given by a 
convolution with the convergent Fourier series 
\begin{equation}
\label{eqn: Fourier_convolution}
\frac{\pi}{\hb}\sum_{k \in \bZ^{d}}\frac{1}{\lan \hb k\ran}
e^{-2\pi\lan \hb k\ran|x_{1}|/\hb}e_{k}(x').
\end{equation}

In the following results we study the properties of convolutions with
functions of the form $e^{-\lambda|x_{1}|}$ to give a proof to
Proposition \ref{propn: isomorphisms}.


\begin{proposition}
\label{propn: weighted_Young}
Let $|\delta| \leq \delta_{0}$ and $g \in L^{1}_{\delta_{0}}(\bR)$. 
For any $f \in L^{2}_{\delta}(\bR)$ we have that
$\|f \ast g\|_{L^{2}_{\delta}} 
\cleq \|f\|_{L^{2}_{\delta}}\|g\|_{L^{1}_{\delta_{0}}}$,
where the constant of the inequality may depend on $\delta_{0}$.
\end{proposition}


\begin{proof}
Using that $\lan a \ran \lan b\ran^{-1} \cleq \lan a - b\ran$ and 
$\lan a \ran \geq 1$ we obtain that
\begin{align*}
\lan x_{1}\ran^{\delta}|f \ast g|(x_{1}) 
& \leq \lan x_{1}\ran^{\delta}(|f|\ast |g|)(x_{1}) \\
& = \int_{\bR}\lan y_{1}\ran^{\delta}|f(y_{1})|
\lan x_{1} - y_{1}\ran^{\delta_{0}}|g(x_{1} - y_{1})|
[\lan x_{1}\ran^{\delta}\lan y_{1}\ran^{-\delta}
\lan x_{1} - y_{1}\ran^{-\delta_{0}}]dy_{1} \\
& \cleq (\lan \cdot \ran^{\delta}|f|)
\ast (\lan \cdot \ran^{\delta_{0}}|g|)(x_{1}).
\end{align*}
The conclusion then follows from Young's inequality.
\end{proof}


\begin{proposition}
\label{propn: weighted_exp}
Let $\lambda \geq 1$. Then $e^{-\lambda|x_{1}|} \in L^{1}_{\delta_{0}}(\bR)$ 
for all $\delta_{0} \geq 0$, and satisfies 
$\|e^{-\lambda|x_{1}|}\|_{L^{1}_{\delta_{0}}} \cleq \lambda^{-1}$,
with the constant depending on $\delta_{0}$.
\end{proposition}


\begin{proof}
Integrating by parts we obtain that if $n \geq 0$ is an integer, then 
\[
\int_{0}^{\infty}e^{-\lambda x_{1}}x_{1}^{n}dx_{1} 
= \frac{n!}{\lambda^{n + 1}} \cleq \frac{1}{\lambda}.
\]

We have that $\lan x_{1}\ran^{n} \cleq 1 + |x_{1}|^{n}$, so that 
if $\delta_{0} \leq n$, then
\[
\int_{\bR}e^{-\lambda |x_{1}|}\lan x_{1}\ran^{\delta_{0}}dx_{1} 
\leq \int_{\bR}e^{-\lambda |x_{1}|}\lan x_{1}\ran^{n}dx_{1}
\cleq \int_{0}^{\infty}e^{-\lambda x_{1}}(1 + x_{1}^{n})dx_{1}  
\cleq \frac{1}{\lambda},
\]

as we wanted to prove.
\end{proof}


\begin{proposition}
\label{propn: convolution_operator}
Let $0 < \hb \leq 1$, $|\delta| \leq \delta_{0}$, and 
$\lambda \geq 1$. If $f \in L^{2}_{\delta}(\bR)$ and we define
$T_{\lambda}f := e^{-2\pi\lambda|x_{1}|/\hb}\ast f$,
then $T_{\lambda}f \in H^{2}_{\delta,\hb}(\bR)$ and 
$\|(\hb D_{x_{1}})^{m}T_{\lambda}f\|_{L^{2}_{\delta}} 
\cleq \hb\lambda^{m - 1}\|f\|_{L^{2}_{\delta}}$ 
for $0 \leq m \leq 2$.
\end{proposition}


\begin{proof}
Differentiating we observe that 
\[
\hb D_{x_{1}}(e^{-2\pi\lambda|x_{1}|/\hb}) 
= \lambda i\sgn(x_{1})e^{-2\pi\lambda|x_{1}|/\hb}, \ \ 
(\hb D_{x_{1}})^{2}(e^{-2\pi\lambda|x_{1}|/\hb})
= \frac{\hb\lambda}{\pi}\delta_{0} 
- \lambda^{2} e^{-2\pi\lambda|x_{1}|/\hb},
\]

so that we have
\[
\hb D_{x_{1}}T_{\lambda}f 
= \lambda i(\sgn(x_{1})e^{-2\pi\lambda|x_{1}|/\hb})\ast f, \ \ 
(\hb D_{x_{1}})^{2}T_{\lambda}f 
= \frac{\hb\lambda}{\pi}f - \lambda^{2}T_{\lambda}f.
\]

From Proposition \ref{propn: weighted_Young} and 
Proposition \ref{propn: weighted_exp} we obtain that
\[
\|T_{\lambda}f\|_{L^{2}_{\delta}} 
\cleq \frac{\hb}{\lambda}\|f\|_{L^{2}_{\delta}}, \ \ 
\|\hb D_{x_{1}}T_{\lambda}f\|_{L^{2}_{\delta}} 
\cleq \hb\|f\|_{L^{2}_{\delta}}, \ \ 
\|(\hb D_{x_{1}})^{2}T_{\lambda}f\|_{L^{2}_{\delta}}
\cleq \hb\lambda\|f\|_{L^{2}_{\delta}}.
\] 
\end{proof}


\begingroup
\def\theproposition{\ref{propn: isomorphisms}}							
\begin{proposition}															
Let $|\delta| \leq \delta_{0}$ and let $0 < \hb \leq 1$. The differential 
operator $\lan \hb D\ran^{2} : H^{2}_{\delta,\hb}(T) \ra L^{2}_{\delta}(T)$ 
and the multiplier operator 
$\lan \hb D\ran^{-2} : L^{2}_{\delta}(T) \ra H^{2}_{\delta,\hb}(T)$ are 
uniformly bounded (in $\delta$ and $\hb$) operators, and inverses to each 
other. The bounds of the operators may depend in $\delta_{0}$.
\end{proposition}
\addtocounter{proposition}{-1}
\endgroup


\begin{proof}
It is clear that the differential operator 
$\lan \hb D \ran^{2}: H^{2}_{\delta,\hb}(T) \ra L^{2}_{\delta}(T)$
is a bounded operator. Using the notation of 
Proposition \ref{propn: convolution_operator}, 
we get from \eqref{eqn: Fourier_convolution} that 
\[
F(x_{1},x') := \lan \hb D\ran^{-2}f(x_{1},x') 
= \frac{\pi}{\hb}\sum_{k \in \bZ^{d}}\frac{1}{\lan \hb k\ran}
T_{\lan \hb k\ran}f_{k}(x_{1})e_{k}(x').
\] 

We have that $\|F\|_{H^{2}_{\delta,\hb}(T)}^{2} 
\simeq \sum_{k \in \bZ^{d}}
\lan \hb k\ran^{4}\|F_{k}\|_{L^{2}_{\delta}(\bR)}^{2}
+ \lan \hb k\ran^{2}\|\hb D_{x_{1}}F_{k}\|_{L^{2}_{\delta}(\bR)}^{2}
+ \|(\hb D_{x_{1}})^{2}F_{k}\|_{L^{2}_{\delta}(\bR)}^{2}$.
The bound 
$\|F\|_{H^{2}_{\delta,\hb}(T)}^{2} \cleq \|f\|_{L^{2}_{\delta}(T)}^{2}$
then follows from Proposition \ref{propn: convolution_operator}. 
We have proven
that both of these maps are uniformly bounded. The fact
that these maps are inverses to each other on $\cS(T)$, together
with the density of $\cS(T)$ in $L^{2}_{\delta}(T)$ and 
$H^{s}_{\delta, \hb}(T)$, implies the desired result.
\end{proof}


\subsubsection{Some facts about semiclassical pseudodifferential 
calculus on $\bR$}


In the results above we needed some quantitative results 
for the bounds of the symbols and operators over $\bR$. 
They are implicitly hinted in the literature, but, for the sake
of completeness, we state them explicitly. We start with 
the operator bounds for zero order pseudodifferential 
operators. To avoid unnecessary notation, we denote
$x_{1} \in \bR$ simply by $x$.


\begingroup
\def\theproposition{\ref{propn: zero_boundedness_R}}							
\begin{proposition}[\cite{Sa1}]										
Let $0 < \hb \leq 1$. Let $a(x_{1},\xi; \hb)$ be a semiclassical zero order 
symbol over $\bR \times \bR$. For any $\delta \in \bR$ the operator 
$\Op_{\hb}(a)$ is bounded in $L^{2}_{\delta}(\bR)$. 
Moreover, if $|\delta| \leq \delta_{0}$, then 
the operator norms 
$\|\Op_{\hb}(a)\|_{L^{2}_{\delta}(\bR) \ra L^{2}_{\delta}(\bR)}$ 
are uniformly bounded (in $\delta$ and $\hb$) by a multiple (depending 
on $\delta_{0}$) of some seminorm of $a$.
\end{proposition} 
\addtocounter{proposition}{-1}
\endgroup


\begin{proof}
Let us write 
\[
\Op_{\hb}(a)f(x) 
:= \frac{1}{\hb}\int_{\bR}e^{2\pi ix\xi/\hb}
a(x,\xi)\what{f^{\hb}}(\xi)d\xi
= \int_{\bR}e^{2\pi ix\xi}a(x, \hb\xi)\what{f}(\xi)d\xi.
\]

The symbols $a_{\hb}(x,\xi) := a(x,\hb \xi)$ satisfy the same seminorm 
estimates $|D_{x}^{\alpha}D_{\xi}^{\beta}a_{\hb}| \leq A_{M}$,
whenever $\alpha + \beta \leq M$. Therefore, it suffices to prove
this estimate for the case $\hb = 1$. The case $\delta = 0$, i.e. in 
$L^{2}(\bR)$, is the Calder\'on--Vaillancourt theorem, and the 
bound for it in terms of the seminorms is stated in \cite{St} 
at the end of Section 2.4, Chapter 6, or Section 4.5 in \cite{Z}
(in the semiclassical setting for the Weyl quantization, but the method
of proof is the same). We prove first the case $\delta > 0$ for 
$\delta = 2n$, with $n$ a positive integer. Using the identities   
\[
\lan x\ran^{2n}e^{2\pi ix\xi} = \lan D_{\xi}\ran^{2n}e^{2\pi ix\xi}, \ \ 
D_{\xi}^{k}\what{f}(\xi) = (-1)^{k}\what{(x^{k}f)}(\xi),
\]

and integrating by parts we obtain that if $f \in \cS(\bR)$, then
\begin{align*}
\lan x\ran^{2n}\Op(a)f 
& = \int_{\bR}\lan D_{\xi}\ran^{2n}(e^{2\pi ix\xi})a(x,\xi)\what{f}(\xi)d\xi \\
& = \int_{\bR}e^{2\pi ix\xi}\lan D_{\xi}\ran^{2n}(a(x,\xi)\what{f}(\xi))d\xi 
= \sum_{k = 0}^{2n}\int_{\bR}e^{2\pi ix\xi}a_{k}(x,\xi)\what{(x^{k}f)}(\xi)d\xi
= \sum_{k = 0}^{2n}\Op(a_{k})(x^{k}f),
\end{align*}

for some zero order symbols $a_{k}(x,\xi)$, with seminorms controlled
by those of $a$. From this and the Calder\'on--Vaillancourt theorem we get
that 
\[
\|\lan x\ran^{2n}\Op(a)f\|_{L^{2}(\bR)}
\leq \sum_{k = 0}^{2n}\|\Op(a_{k})(x^{k}f)\|_{L^{2}(\bR)}
\cleq \sum_{k = 0}^{2n}\|x^{k}f\|_{L^{2}(\bR)} \cleq \|\lan x\ran^{2n}f\|_{L^{2}(\bR)},
\]

with the constant of the inequality depending on $n$ and some seminorm of $a$.
We have shown that $\Op(a)$ is bounded on $L^{2}_{2n}(\bR)$. The intermediate
values $0 < \delta < 2n$ are obtained by complex interpolation.

\bigskip
Now, let us consider the case $\delta < 0$ for $\delta = -2n$, with
$n$ a positive integer. Integrating by parts we obtain that if $f \in \cS(\bR)$, 
then
\begin{align*}
\lan x\ran^{-2n}\Op(a)\lan x\ran^{2n}f
& = \lan x\ran^{-2n}\int_{\bR} e^{2\pi ix\xi}a(x,\xi)\lan D_{\xi}\ran^{2n}\what{f}(\xi)d\xi \\
& = \lan x\ran^{-2n}\int_{\bR}\lan D_{\xi}\ran^{2n}(e^{2\pi ix\xi}a(x,\xi))\what{f}(\xi)d\xi
= \Op(\wilde{a})f,
\end{align*}

for some zero order symbol $\wilde{a}(x,\xi)$, with seminorms controlled
by those of $a$. This identity can be rewritten as 
$\lan x\ran^{-2n}\Op(a) = \Op(\wilde{a})\lan x\ran^{-2n}$, and so the 
Calder\'on--Vaillancourt theorem gives the boundedness on 
$L^{2}_{-2n}(\bR)$. Again, the intermediate values $-2n < \delta < 0$
are obtained by complex interpolation.
\end{proof}


We also prove the following result for the symbol of 
the composition, which we used in the proof of 
Theorem \ref{thm: composition}.


\begin{proposition}
\label{propn: bounds_symbol_composition}
Let $a(x,\xi; \hb)$ and $b(x,\xi; \hb)$ be symbols 
over $\bR \times \bR$ satisfying 
\[
|D_{x}^{\alpha}D_{\xi}^{\beta}a| 
\leq A_{M}, \ \ 
|D_{x}^{\alpha}D_{\xi}^{\beta}b| \leq B_{M},
\]

whenever $\alpha + \beta \leq M$. If 
$c = c(x_{1},\xi; \hb)$ is the symbol such that 
$\Op_{\hb}(c) = \Op_{\hb}(a)\Op_{\hb}(b)$, then
there exists some $K$ such that for any 
$M \geq 0$ the symbol satisfies
\[
|D_{x}^{\alpha}D_{\xi}^{\beta}c| 
\cleq A_{K + M}B_{K + M}, \ \
|D_{x}^{\alpha}D_{\xi}^{\beta}(c - ab)| 
\cleq \hb A_{K + M}B_{K + M},
\]

whenever $\alpha + \beta \leq M$, where the
constants of the inequalities may depend on $M$ 
but are independent of $\hb$.
\end{proposition}


\begin{proof}
Proceeding as in \cite{St}, see Chapter 6, Section 3, 
it suffices to show the estimates for compactly supported 
symbols and prove that these are independent of the size
of the support. Let us recall the integral kernel representation
of a pseudodifferential operator,
\[
\Op_{\hb}(s)f(x) 
= \frac{1}{\hb}\int_{\bR}e^{2\pi ix\xi/\hb}s(x,\xi)
\what{f^{\hb}}(\xi)d\xi
= \frac{1}{\hb}\int_{\bR}\biggl(\int_{\bR}
e^{2\pi i(x - y)\xi/\hb}s(x,\xi)d\xi\biggr)f(y)dy.
\]

Then, the composition has an integral kernel representation 
given by
\begin{align*}
\Op_{\hb}(a)\Op_{\hb}(b)f(x) 
& = \frac{1}{\hb}\int_{\bR}\biggl(\int_{\bR}
e^{2\pi i(x - y)\xi/\hb}a(x,\xi)d\xi\biggr)\Op_{\hb}(b)f(y)dy \\
& = \frac{1}{\hb}\int_{\bR}\biggl(\frac{1}{\hb}\int_{\bR^{2}}
e^{2\pi i(x - y)\xi/\hb}e^{2\pi i(y - z)\eta/\hb}a(x,\xi)b(y,\eta)
d\xi d\eta dy\biggr)f(z)dz \\
& = \frac{1}{\hb}\int_{\bR}\biggl(\int_{\bR}
e^{2\pi i(x - z)\eta/\hb}c(x,\eta)d\eta\biggr)f(z)dz,
\end{align*}

and therefore the symbol of $\Op_{\hb}(a)\Op_{\hb}(b)$ is
equal to
\[
c(x,\eta) := \frac{1}{\hb}\int_{\bR^{2}}
e^{2\pi i(x - y)(\xi - \eta)/\hb}a(x,\xi)b(y,\eta)d\xi dy
= \frac{1}{\hb}\int_{\bR^{2}}
e^{-2\pi iy\xi/\hb}a(x,\eta + \xi)b(x + y,\eta)d\xi dy.
\]

From the inversion formula we have that
\[
a(x,\eta) 
= \frac{1}{\hb}\int_{\bR^{2}}e^{-2\pi iy\xi/\hb}
a(x,\eta + \xi)d\xi dy,
\]
which implies that
\[
c(x,\eta) - a(x,\eta)b(x,\eta)
= \frac{1}{\hb}\int_{\bR^{2}}
ye^{-2\pi iy\xi/\hb}a(x,\eta + \xi)\cdot \frac{b(x + y,\eta) 
- b(x,\eta)}{y}d\xi dy.
\]

Now we use the identity
\[
\frac{D_{\xi}\lan \hb D_{\xi}\ran^{2}}{\lan y\ran^{2}}e^{-2\pi iy\xi/\hb}
= -\frac{1}{\hb}ye^{-2\pi iy\xi/\hb}
\]

and integrate by parts to obtain that
\begin{align*}
c(x,\eta) - a(x,\eta)b(x,\eta)
& = \int_{\bR^{2}}
e^{-2\pi iy\xi/\hb}(D_{\xi}\lan \hb D_{\xi}\ran^{2}a(x,\eta + \xi))
\biggl(\frac{b(x + y,\eta) - b(x,\eta)}{y\lan y\ran^{2}}\biggr)
d\xi dy \\
& =: \int_{\bR^{2}}e^{-2\pi iy\xi/\hb}A(x,\eta,\xi)
B(x,\eta,y)d\xi dy.
\end{align*}

Let us observe that the integral above is absolutely convergent 
because $A$ has compact support in $\xi$ and $B$ is integrable 
in $y$. Therefore, we can exchange the order of integration
and obtain that
\begin{align*}
\biggl|\int_{\bR^{2}}& e^{-2\pi iy\xi/\hb}A(x,\eta,\xi)
B(x,\eta,y)d\xi dy\biggr| \\
& = \hb\biggl|\int_{\bR^{2}}e^{-2\pi iy\mu}A(x,\eta, \hb\mu)
B(x,\eta,y)dyd\mu\biggr| \\
& = \hb \biggl|\int_{\bR}A(x,\eta, \hb\mu)
B(x,\eta,\what{\mu})d\mu\biggr|
\cleq \hb\|A(x,\eta,\cdot)\|_{L^{\infty}(\bR)}
\|\lan D_{y}\ran^{2}B(x,\eta,\cdot)\|_{L^{1}(\bR)},
\end{align*}

where we used in the last inequality that 
\[
\int_{\bR}|\what{f}(\mu)|d\mu 
= \int_{\bR}\frac{|\lan \mu\ran^{2}\what{f}(\mu)|}
{\lan \mu\ran^{2}}d\mu
= \int_{\bR}\frac{|\what{\lan D\ran^{2} f}(\mu)|}
{\lan \mu\ran^{2}}d\mu
\cleq \|\what{\lan D\ran^{2}f}\|_{L^{\infty}(\bR)}
\cleq \|\lan D\ran^{2}f\|_{L^{1}(\bR)}.
\]

Similarly, for $\alpha + \beta \leq M$ we can bound
\begin{align*}
|D_{x}^{\alpha}D_{\eta}^{\beta}(c - ab)| 
& \cleq  \hb\biggl|\int_{\bR}D_{x}^{\alpha}
D_{\eta}^{\beta}(A(x,\eta,\hb\mu)
B(x,\eta,\what{\mu}))d\mu\biggr| \\
& \cleq \hb\sup_{\alpha_{0} + \beta_{0} \leq M}
\|D_{x}^{\alpha_{0}}D_{\eta}^{\beta_{0}}
A(x,\eta,\cdot)\|_{L^{\infty}(\bR)}
\sup_{\substack{\alpha_{1} + \beta_{1} \leq M \\ \gamma_{1} \leq 2}}
\|D_{x}^{\alpha_{1}}D_{\eta}^{\beta_{1}}
D_{y}^{\gamma_{1}}B(x,\eta,\cdot)\|_{L^{1}(\bR)}.
\end{align*}

To finish the estimate we use that
\[
|D_{x}^{\alpha}D_{\eta}^{\beta}A(x,\eta,\xi)| 
\cleq A_{M + 3}, \ \ 
|D_{x}^{\alpha}D_{\eta}^{\beta}D_{y}^{\gamma}
B(x,\eta,y)| 
\cleq \frac{B_{M + 3}}{\lan y\ran^{2}},
\]

whenever $\alpha + \beta \leq M$ and $\gamma \leq 2$.
The differential inequalities for the difference $c - ab$
imply the results for the symbol $c$, and this completes
the proof.
\end{proof}


\section{Conjugation and Carleman Estimate}


The results of this chapter are an adaptation of those from \cite{Sa1} in the case of 
$\bR^{d}$ to the case of the cylinder $T = \bR \times \bT^{d}$. As briefly mentioned in 
the setting, the proof of the magnetic Carleman estimate Theorem \ref{thm: carleman_full} is 
reduced to the case with no potentials. The proof of this Carleman estimate in \cite{KSaU1} 
for $\bR \times M_{0}$, where $M_{0}$ is a Riemannian manifold with boundary, is realized
by an eigenfunction expansion and the solution of first order linear constant coefficient ODEs. 
This can be carried out in the exact same way for the torus $\bT^{d}$, so we have the 
following result.


\begin{theorem}[\cite{KSaU1}]
\label{thm: carleman_free}
Let $\delta > 1/2$. There exists
$\tau_{0} \geq 1$ such that if $|\tau| \geq \tau_{0}$ and 
$\tau^{2} \notin \spec(-\Delta_{g_{0}})$, then for any 
$f \in L^{2}_{\delta}(T)$ there exists a unique 
$u \in H^{1}_{-\delta}(T)$ which solves 
\[
e^{2\pi \tau x_{1}}D^{2}e^{-2\pi \tau x_{1}}u = f.
\]
 
Moreover, this solution is in $H^{2}_{-\delta}(T)$ and satisfies the 
estimates 
\[
\|u\|_{H^{s}_{-\delta}(T)} \cleq |\tau|^{s - 1}\|f\|_{L^{2}_{\delta}(T)},
\]

for $0 \leq s \leq 2$, with the constant of the inequality independent of $\tau$. 
\end{theorem}


\begin{remark}
It is important to note that the constant in the inequality only
requires the condition that $\tau^{2}$ does not belong to
$\spec(-\Delta_{g_{0}})$; it is not necessary to ensure 
any distance condition to the spectrum. We can easily see 
that the uniqueness would fail if
$\tau^{2} \in \spec(-\Delta_{g_{0}})$ as 
$u = e_{m}(x') \in H^{2}_{-\delta}(T)$, with 
$\tau^{2} = |m|^{2}$, is a solution of the homogeneous
problem.
\end{remark}


The theorem above allows to define the operator
$G_{\tau}: L^{2}_{\delta}(T) \ra H^{2}_{-\delta}(T)$ by $G_{\tau}f := u$, 
so that $\Delta_{\tau}G_{\tau} = I$ on $L^{2}_{\delta}(T)$, where 
$\Delta_{\tau} = e^{2\pi\tau x_{1}}D^{2}e^{-2\pi\tau x_{1}}$.

\bigskip
The reduction from our problem to this one is accomplished through 
a conjugation by two invertible pseudodifferential operators, i.e. 
essentially through the construction of an integrating factor. The 
construction of these operators is the main content of this chapter.
 
\bigskip
Let us consider the relevant terms from the expression 
$e^{2\pi \tau x_{1}}H_{V,W}e^{-2\pi \tau x_{1}}$:
\[
\Delta_{\tau} := e^{2\pi\tau x_{1}}D^{2}e^{-2\pi\tau x_{1}}
= D_{x_{1}}^{2} + 2i\tau D_{x_{1}} - \tau^{2} + D_{x'}^{2} ,
\]
\[
V_{\tau} := e^{2\pi \tau x_{1}}(V\cdot D)e^{-2\pi \tau x_{1}}
= e^{2\pi \tau x_{1}}(FD_{x_{1}} + G\cdot D_{x'})e^{-2\pi \tau x_{1}}
= F(D_{x_{1}} + i\tau) + G\cdot D_{x'}.
\]

\begin{remark}
Observe that we have absorbed the negative sign of the Laplacian into the
definition of $\Delta_{\tau}$.
\end{remark}

If we use semiclassical notation, with $\hb = 1/\tau$ a small
parameter, we can denote
\[
\Delta_{\hb} := \tau^{-2}\Delta_{\tau}
= \hb^{2}D_{x_{1}}^{2} + 2i\hb D_{x_{1}} - 1 + \hb^{2}D_{x'}^{2}, 
\ \ \ \
V_{\hb} := \tau^{-1}V_{\tau}
= F(\hb D_{x_{1}} + i) + G\cdot \hb D_{x'}.
\]

Equivalently, we could have defined 
\[
\Delta_{\hb} := e^{2\pi x_{1}/\hb}(\hb D)^{2}e^{-2\pi x_{1}/\hb}, 
\ \ \ \ 
V_{\hb} := e^{2\pi x_{1}/\hb}[V\cdot (\hb D)]e^{-2\pi x_{1}/\hb}.
\]

Then we have that 
$\hb^{2}(\Delta_{\tau} + 2V_{\tau}) = \Delta_{\hb} + 2\hb V_{\hb}$.
A significant part of this chapter is devoted to prove that we can conjugate 
this operator into the Laplacian plus a suitable error, as we state next. This
construction follows closely the ideas from \cite{Sa1}.


\begin{theorem}
\label{thm: conjugation}
Let $1/2 < \delta < 1$. Let $V$ satisfy \eqref{eqn: conditions}.
There are $\vep > 0$ and $0 < \hb_{0} \leq 1$ such that for $0< |\hb| \leq \hb_{0}$ 
there exist zero order semiclassical pseudodifferential operators $A$, $B$, 
$R$ over the cylinder $T$, so that the following conjugation identity holds, 
\[
(\Delta_{\hb} + 2\hb V_{\hb})A = B\Delta_{\hb} + \hb^{1 + \vep}R.
\]

Moreover, the operators $A$ and $B$ are invertible, uniformly bounded (in $\hb$) 
together with its inverses in $H^{s}_{\pm \delta,\hb}(T)$, for $s = 0,2$, 
and $R: L^{2}_{-\delta}(T) \ra L^{2}_{\delta}(T)$ is uniformly bounded
(in $\hb$).
\end{theorem}


With the semiclassical notation we have that
$\Delta_{\hb} = \Op_{\hb}(\xi^{2} + 2i\xi - 1 + |\hb t|^{2}) 
= \Op_{\hb}((\xi + i)^{2} + |\hb t|^{2})$.
Moreover, from Proposition \ref{propn: symbol_computations}
we have that
\begin{align*}
\Delta_{\hb}A = & \ (\hb^{2}D_{x_{1}}^{2} + 2i\hb D_{x_{1}} 
- 1 + \hb^{2}D_{x'}^{2})A \\
= & \ \Op_{\hb}(\xi^{2}a 
+ 2\hb\xi D_{x_{1}}a + \hb^{2}D_{x_{1}}^{2}a) 
+ 2i\Op_{\hb}(\xi a + \hb D_{x_{1}}a) - \Op_{\hb}(a) \\ 
& + \Op_{\hb}(|\hb t|^{2}a 
+ 2\hb(\hb t\cdot D_{x'}a) + \hb^{2}D_{x'}^{2}a) \\
= & \ \Op_{\hb}([(\xi + i)^{2} + |\hb t|^{2}]a) 
+ 2\hb\Op_{\hb}((\xi + i)D_{x_{1}}a + \hb t\cdot D_{x'}a)
+ \hb^{2}\Op_{\hb}(D^{2}a) \\
= & \ A\Delta_{\hb} + 2\hb\Op_{\hb}((\xi + i)D_{x_{1}}a 
+ \hb t\cdot D_{x'}a) + \hb^{2}\Op_{\hb}(D^{2}a), \\ \\
V_{\hb}A = & \ (F(\hb D_{x_{1}} + i) + G\cdot \hb D_{x'})A \\
= & \ \Op_{\hb}((\xi+ i)Fa + \hb F \cdot D_{x_{1}}a) 
+ \Op_{\hb}(\hb t\cdot Ga + \hb G\cdot D_{x'}a) \\
= & \ \Op_{\hb}([(\xi + i)F + \hb t\cdot G]a) 
+ \hb\Op_{\hb}(V\cdot Da),
\end{align*}

so, we obtain
\begin{align*}
(\Delta_{\hb} & + 2\hb V_{\hb})A \\
& = A\Delta_{\hb} 
+ 2\hb\Op_{\hb}((\xi + i)D_{x_{1}}a + \hb t\cdot D_{x'}a 
+ (\xi + i)Fa + \hb t\cdot Ga)
+ \hb^{2}\Op_{\hb}(D^{2}a + 2V\cdot Da).
\end{align*}

If $a$ has nice properties, then the last operator already 
has the form we look for the remainder term in 
Theorem \ref{thm: conjugation}. 
Then, roughly speaking, we are left to make the operator
\begin{equation}
\label{eqn: symbol_commutator}
2\hb\Op_{\hb}((\xi + i)D_{x_{1}}a 
+ \hb t\cdot D_{x'}a + (\xi + i)Fa + \hb t\cdot Ga)
\end{equation}

suitably small. In order to do that, we split it in two parts:
we make one part of it vanish and the remainder will be supported 
on a set where the operator $\Delta_{\hb}$ is elliptic. The remainder 
will be subsumed by the expression $A\Delta_{\hb}$ becoming into 
$B\Delta_{\hb}$.

\bigskip
In order for the operator $A$ to be
invertible it is usual to look for the symbol to be
of the form $a = e^{-u}$, so that the symbol 
\eqref{eqn: symbol_commutator} becomes							
\[
(\xi + i)D_{x_{1}}a + \hb t\cdot D_{x'}a 
+ (\xi + i)Fa + \hb t\cdot Ga
= a[-(\xi + i)D_{x_{1}}u - \hb t\cdot D_{x'}u + (\xi + i)F + \hb t\cdot G],
\]

leaving us to solve the equation 
\begin{equation}
\label{eqn: delta_bar_equation}
(\xi + i)D_{x_{1}}u + \hb t\cdot D_{x'}u = (\xi + i)F + \hb t\cdot G,
\end{equation}

for $t \in \bZ^{d}$. In the following sections we deal with 
appropriate existence and uniqueness of solutions to these equations, 
as well as with its estimates. 
Recall that the symbol of $\Delta_{\hb}$ is $(\xi + i)^{2} + |\hb t|^{2}$,
and note that this vanishes if and only if $\xi = 0$ and $|\hb t| = 1$.
The symbol is elliptic away from this set. Therefore, for the construction 
of the solution to the equation we will be mostly interested in working in 
a neighborhood of this vanishing set. 

\bigskip
Finally, let us mention some difficulties of our problem which do not seem 
to be present in the Euclidean setting, like in \cite{Sun} or \cite{Sa1}. 
Observe that \eqref{eqn: delta_bar_equation} can be rewritten as 
\[
(\xi + i, \hb t)\cdot Du = (\xi + i, \hb t)\cdot V.
\]

Near the vanishing set of the symbol of $\Delta_{\hb}$, i.e. $\xi = 0$ and 
$|\hb t| = 1$, this equation resembles a higher dimensional version of the 
$\ov{\pa}$--equation. It has been usual to reduce all such equations to 
a $\ov{\pa}$--equation through a rotation, see for instance \cite{Sun} 
or \cite{Sa1}. In our setting 
this is not immediately possible, in part because $\bZ^{d}$ does not admit non--trivial 
rotations. One way to try to remedy this could be as follows. In \cite{W}, it is shown 
that for any $k \in \bZ^{d}\sm\{0\}$ there is a matrix $A \in SL_{n}(\bZ^{d})$, 
i.e. a linear automorphism of $\bT^{d}$, such that $Ak = \gcd(k)e_{1}$. This
allows to make a change of variables so that the directional derivative 
$\hb k\cdot D_{x'}$ becomes an ``exact'' partial  derivative $\hb \gcd(k)D_{y'_{1}}$, 
and so the equation reduces from $\bR \times \bT^{d}$ to 
$(\bR \times \bT^{1}) \times \bT^{d - 1}$, where the last $(d - 1)$ toroidal
variables do not intervene. In this case, the coefficient $\hb \gcd(k)$ plays a role 
in the estimates, and this seems difficult to handle. In addition to the previous 
inconvenient, we will to need to estimate of the differences of the solutions 
when $k$ varies; since it is not clear how the change of variables (i.e. the 
matrix) depends on $k$, we will refrain from using this idea. Instead, we will 
proceed using the decomposition in Fourier series.


\subsection{Equation}


We start recalling the assumptions on the magnetic potential $V = (F,G)$:
\[
\tag{\ref{eqn: conditions}}
V \in C^{\infty}_{c}(T), \ \ \supp(V) \sse [-R,R]\times \bT^{d}, \ \
\int_{\bR}V(x_{1},x')dx_{1} = 0
\ \textrm{for all} \ x' \in \bT^{d}.
\]
Under these conditions, we will see
that there is no difference in working over $\bR \times \bT^{d} \times \bR \times \bZ^{d}$ 
or $\bR \times \bT^{d} \times \bR \times \bR^{d}$, and so, to avoid suggesting 
that there is something special about the former, we will work over the latter. 
In this section we state the properties of the solution the equation 
\eqref{eqn: delta_bar_equation} and motivate the reasons for assuming 
\eqref{eqn: conditions}.

\bigskip
Let us recall that for $(u,v) \in \bR \times \bR^{d}$ 
and $(u,v,w) \in \bR \times \bR^{d} \times \bR^{d}$
we use the notation
$\lan u, v\ran := (1 + u^{2} + |v|^{2})^{1/2}$ and 
$\lan u,v,w\ran := (1 + u^{2} + |v|^{2} + |w|^{2})^{1/2}$.


\begin{theorem}
\label{thm: delta_bar_equation}
Let $\hb > 0$ and let $V = (F,G)$ satisfy \eqref{eqn: conditions}.
For fixed $(\xi,t) \in \bR \times \bR^{d}$, 
the equation
\[
(\xi + i)D_{x_{1}}u + \hb t\cdot D_{x'}u  = (\xi + i)F + \hb t\cdot G, 
\tag{\ref{eqn: delta_bar_equation}}
\]

has a unique solution $u(\cdot,\xi,t; \hb) \in C^{\infty}(T)$ with the decay 
condition $\|u(x_{1},\cdot,\xi,t)\|_{L^{\infty}(\bT^{d})} \ra 0$ as 
$x_{1} \ra \pm \infty$. Moreover, 
we have that $u(\cdot,t) \in C^{\infty}(\bR \times \bT^{d} \times \bR)$
and it satisfies the bounds
\begin{equation}
\label{eqn: solution_bounds}
|D_{x_{1}}^{\alpha}D_{x'}^{\beta}D_{\xi}^{\gamma}u| 
\cleq \lan \xi, \hb t\ran , \ \
|D_{x_{1}}^{\alpha}D_{x'}^{\beta}D_{\xi}^{\gamma}
(u(\cdot,\xi, t_{1}) - u(\cdot,\xi, t_{2}))|
\cleq \hb|t_{1} - t_{2}|\lan \xi, \hb t_{1}, \hb t_{2} \ran.
\end{equation}

For $|x_{1}| \geq 2R$ we have the linear decay bound
\begin{equation}
\label{eqn: solution_decay}
|D_{x_{1}}^{\alpha}D_{x'}^{\beta}D_{\xi}^{\gamma}u| 
\cleq \frac{\lan \xi, \hb t\ran}{|x_{1}|}.
\end{equation}

The constants of the inequalities may depend on $\alpha$, $\beta$, 
$\gamma$, $d$, $R$, $\|V\|_{W^{N,1}(T)}$ 
for some $N = N(\alpha,\beta, d)$, but are independent
of $\hb$, $\xi$, $t$.
\end{theorem}


\begin{remark}
Under some conditions on $t$, like $t \in \bZ^{d}$ or other 
arithmetic properties, it may be possible to show that
$u(\cdot,\xi,t) \in \cS(T)$. We do not need such a strong 
result, so we do not intend to prove it.
\end{remark}


The equation \eqref{eqn: delta_bar_equation} has constant coefficients 
in $(x_{1},x')$, so we can decompose $u, F, G$ in its Fourier series
\[
u = \sum_{m \in \bZ^{d}}u_{m}(x_{1},\xi,t;\hb)e_{m}(x'), \ \ 
F(x_{1},x') = \sum_{m \in \bZ^{d}}F_{m}(x_{1})e_{m}(x'), \ \ 
G(x_{1},x') = \sum_{m \in \bZ^{d}}G_{m}(x_{1})e_{m}(x'),
\]

and look for $u_{m}$ to solve the equation
\begin{equation}
\label{eqn: Fourier_ODE}
(\xi + i)D_{x_{1}}u_{m} + \hb t\cdot mu_{m} 
= (\xi + i)F_{m} + \hb t\cdot G_{m}.
\end{equation}

For instance, in order to prove \eqref{eqn: solution_bounds}, 
it would suffice to show inequalities of the form 
\[
|D_{x_{1}}^{\alpha}D_{\xi}^{\beta}u_{m}| 
\cleq \lan m\ran^{-M}\lan \xi,\hb t\ran, \ \
|D_{x_{1}}^{\alpha}D_{\xi}^{\beta}(u_{m}(\cdot,\cdot,t_{1})
- u_{m}(\cdot,\cdot,t_{2}))|
\cleq \hb \lan m\ran^{-M}|t_{1} - t_{2}|\lan\xi, \hb t_{1}, \hb t_{2}\ran,
\]

for some sufficiently large $M$. We will prove these in a later section.

\bigskip
Before we proceed, let us motivate the conditions that we are requiring
for $V$. Considering the Fourier transform (no longer semiclassical)
in \eqref{eqn: Fourier_ODE} gives that
\[
\what{u_{m}}(\eta) = \frac{(\xi + i)\what{F_{m}}(\eta) 
+ \hb t\cdot \what{G_{m}}(\eta)}{(\xi + i)\eta + \hb t\cdot m}.
\]

The denominator vanishes only if $\eta = 0$ and $\hb t\cdot m = 0$. 
This suggests that the case $\hb t \cdot m \neq 0$ may be less problematic
than the case $\hb t \cdot m = 0$. Indeed, we already see from 
\eqref{eqn: Fourier_ODE} that there is not a unique 
solution, and even when defining such a solution it may not
decay. The simplest way to avoid the problem of the denominator 
vanishing is to require $\what{F_{m}}(0) = \what{G_{m}}(0) = 0$, 
which is the same as the vanishing moments from \eqref{eqn: conditions}.
Uniqueness and decay are not necessarily required, but we will
see that the decay estimates play a role in the construction of 
the conjugation (namely on the properties of $R$) and the
reconstruction procedure. 


\subsection{Lemmas: ODEs and calculus}


In this section we prove some boundedness estimates for the 
solution of an ODE of the form \eqref{eqn: Fourier_ODE}, as 
well as some other necessary calculus facts. To avoid unnecessary 
notation, in this section we will denote the variable $x_{1}$ 						
simply by $x$.

\bigskip
We start with the most elementary estimate for solutions of an
ODE, and then improve it under the hypothesis of a vanishing 
moment. The reason why
we will be dealing only with $L^{1}$ and $L^{\infty}$ estimates
is that these are useful to iterate and they relate through the 
Fundamental Theorem of Calculus (i.e. as a $1$-dimensional version 
of the Gagliardo--Nirenberg--Sobolev inequality).


\begin{lemma}
\label{lemma: ODE_original}
Let $a, \xi \in \bR$, $a \neq 0$, and let $H \in \cS(\bR)$. 
Consider the equation
\[
(\xi + i)D_{x}v + av = H.
\]

Then there exists a unique solution in the sense 
of tempered distributions. Moreover, the solution belongs 
to $\cS(\bR)$ and satisfies the estimates
\[
\|v\|_{L^{1}} \leq \frac{\lan \xi \ran}{|a|}\|H\|_{L^{1}}, \ \ 
\|v\|_{L^{\infty}} \cleq \|D_{x}v\|_{L^{1}} 
\cleq \|H\|_{L^{1}}.
\]

The constants in the inequality are independent of $a$, $\xi$, $H$.
\end{lemma}


\begin{proof}
Taking the Fourier transform yields that
\[
\what{v}(\eta) = \frac{\what{H}(\eta)}{(\xi + i)\eta + a}.
\]

We can bound the norm of the denominator by
\[
|(\xi + i)\eta + a|^{2} = \lan \xi \ran^{2}\eta^{2} + 2\xi \eta a + a^{2}
= \biggl(\lan \xi\ran \eta + \frac{\xi a}{\lan \xi\ran}\biggr)^{2} 
+ \frac{a^{2}}{\lan \xi\ran^{2}} \geq \frac{a^{2}}{\lan \xi \ran^{2}} > 0.
\] 

Therefore, the denominator is a non-vanishing smooth function, and we obtain 
the existence and uniqueness in the sense of distributions. Moreover,
the rapid decay of $\what{H}(\eta)$ and the bound for the denominator
give that $v \in \cS(\bR)$. Let 
$\mu := 2\pi ia/(\xi + i) = 2\pi a(1 + i\xi)/\lan \xi\ran^{2}$, 
so that $\Re(\mu) = 2\pi a/\lan \xi\ran^{2}$. The form
of the solution depends on the sign of $\Re(\mu)$.
The cases $a > 0$ and $a < 0$ are analogous, 
so we only consider one of these.
If we assume that $a > 0$, then the solution is given by 
\[
v(x) = \frac{2\pi i}{\xi + i}\int_{-\infty}^{x}e^{-\mu(x - s)}H(s)ds,
\]

as $\Re(\mu) = 2\pi a/\lan \xi\ran^{2} > 0$. This gives that 
\begin{align}
\label{eqn: ODE_L1_bound}
\|v\|_{L^{1}} & \leq \frac{2\pi}{\lan \xi\ran}
\int_{\bR}\int_{-\infty}^{x}e^{-2\pi a(x - s)/\lan \xi\ran^{2}}|H(s)|dsdx \nonumber \\
& = \frac{2\pi}{\lan \xi\ran}\int_{\bR}
\int_{s}^{\infty}e^{-2\pi a(x - s)/\lan \xi\ran^{2}}|H(s)|dxds
= \frac{\lan \xi \ran}{a}\|H\|_{L^{1}}.
\end{align}

Using this together with the equation gives 
\[
\|D_{x}v\|_{L^{1}} \leq \frac{1}{\lan \xi\ran}(a\|v\|_{L^{1}} + \|H\|_{L^{1}})
\leq \frac{\lan \xi \ran + 1}{\lan \xi \ran}\|H\|_{L^{1}}
\cleq \|H\|_{L^{1}},
\]

as we wanted. The bound $\|v\|_{L^{\infty}} \cleq \|D_{x}v\|_{L^{1}}$ follows
from the Fundamental Theorem of Calculus and the fact that $v \in \cS(\bR)$.
\end{proof}


\begin{remark}
It is also possible to show that 
\[
\|v\|_{L^{\infty}} \cleq \frac{\lan\xi \ran}{|a|}\|H\|_{L^{\infty}}.
\]
\end{remark}


The following result shows that a vanishing moment assumption allows 
to consider the missing case $a = 0$ and also gives an improvement of the 
$L^{1}$-estimates for the solution.


\begin{lemma}
\label{lemma: ODE_vanishing}
Let $a, \xi \in \bR$ and let $H \in \cS(\bR)$. The equation
\[
(\xi + i)D_{x}v + av = D_{x}H.
\]

has a unique solution $v \in \cS(\bR)$ and it satisfies 
\[
\|v\|_{L^{1}} \cleq \|H\|_{L^{1}}, \ \ 
\|v\|_{L^{\infty}} \cleq \|D_{x}v\|_{L^{1}} 
\cleq \|D_{x}H\|_{L^{1}}.
\]

The constants in the inequality are independent of $a$, $\xi$, $H$.
Moreover, if $a \neq 0$, then there exists $w \in \cS(\bR)$ such 
that $v = D_{x}w$. 
\end{lemma}


\begin{proof}
If $a = 0$ then $v = H/(\xi + i) \in \cS(\bR)$, and the result follows 
immediately. For $a \neq 0$ the existence and uniqueness follow from 
Lemma \ref{lemma: ODE_original}.
The cases $a > 0$ and $a < 0$ are analogous, so we only consider one 
of these. Assume that $a > 0$ and let
$\mu := 2\pi a/(\xi + i)$ be as in the previous proof.
Integrating by parts yields that
\begin{equation}
\label{eqn: vanishing_solution}
v(x) = \frac{2\pi i}{\xi + i}\int_{-\infty}^{x}e^{-\mu(x - s)}D_{s}H(s)ds
= \frac{1}{\xi + i}\biggl(H(x) 
- \mu\int_{-\infty}^{x}e^{-\mu(x - s)}H(s)ds\biggr).
\end{equation}

We use the estimate \eqref{eqn: ODE_L1_bound} for the integral term, to
conclude that
\[
\|v\|_{L^{1}} \leq \frac{1}{\lan \xi \ran}\biggl(1
+ |\mu|\frac{\lan \xi\ran^{2}}{a}\biggr)\|H\|_{L^{1}}
\cleq \|H\|_{L^{1}}.
\]

The $L^{\infty}$ bound follows trivially if $a = 0$, and from 
Lemma \ref{lemma: ODE_original} if $a \neq 0$. Finally, if $a \neq 0$,
then 
\[
v = \frac{1}{a}D_{x}(H - (\xi + i)v) = D_{x}w,
\]

with $w = (H - (\xi + i)v)/a \in \cS(\bR)$. 
\end{proof}


\begin{remark}
The solutions in $C^{1}(\bR)$ to the equation 
$(\xi + i)D_{x}v + av = 0$ are multiples of 
$e^{-\mu x}$. Therefore, there is also 
uniqueness under the weaker conditions 
$v \in C^{1}(\bR)$ and $v(x) \ra 0$ as $x \ra \pm \infty$.
\end{remark}


\begin{remark}
From \eqref{eqn: vanishing_solution} it is also possible to show that 
\[
\|v\|_{L^{\infty}} 
\leq \frac{1}{\lan \xi\ran}\biggl( 1
+ \frac{|\mu|\lan \xi\ran^{2}}{a}\biggr)\|H\|_{L^{\infty}}
\cleq \|H\|_{L^{\infty}}.
\]

In the proof of the main result of the next section we will remark
why focusing only in the $L^{\infty}\ra L^{\infty}$ estimates may
not be so convenient.
\end{remark}


In the following proposition we show that the exponential function appearing 
from the integrating factor of the differential equations is bounded, together 
with all its derivatives. In proving the boundedness results from the following 
section we will use this result, as well as the idea of the proof.


\begin{lemma}
\label{lemma: exponential}
Let $\eta, \xi \in \bR$ with $\eta \geq 0$. Then, for any polynomial $p$ 
we have
\[
\biggl|e^{-i\eta/(\xi + i)}p\biggl(\frac{\eta}{\lan \xi\ran^{2}}\biggr)
\biggr| \leq C(p),
\]

for some constant $C(p)$ depending on the polynomial. Moreover,
$|D_{\xi}^{\beta}(e^{-i\eta/(\xi + i)})| \leq C_{\beta}$ for any 
$\beta \geq 0$. The constants $C(p)$ and $C_{\beta}$ are independent 
of $\eta$ and $\xi$.
\end{lemma}


\begin{proof}
Let $\mu = i\eta/(\xi + i) = \eta(1 + i\xi)/\lan \xi\ran^{2}$, so that 
$\Re(\mu) = \eta/\lan \xi\ran^{2} > 0$. The first inequality follows from the
triangle inequality and the bound $e^{-x}x^{n} \leq n!$ for $x \geq 0$. 
In order to differentiate with respect to $\xi$, we first 
observe that $D_{\xi}\mu = -\mu/(2\pi i(\xi + i))$. We show 
by induction that 
\[
D_{\xi}^{\beta}(e^{-\mu}) = \frac{e^{-\mu}P_{\beta}(\mu)}{(\xi + i)^{\beta}},
\]

where $P_{\beta}$ is a polynomial of degree $\beta$, whose coefficients
depend only on $\beta$. For $\beta = 0$ it is clear. Moreover, 
\[
D_{\xi}\biggl(\frac{e^{-\mu}P_{\beta}(\mu)}{(\xi + i)^{\beta}}\biggr)
= e^{-\mu}\biggl(\frac{\mu}{2\pi i(\xi + i)}\frac{P_{\beta}(\mu)}{(\xi + i)^{\beta}}
- \frac{\mu}{2\pi i(\xi + i)}\frac{P'_{\beta}(\mu)}{(\xi + i)^{\beta}}
- \frac{\beta P_{\beta}(\mu)}{2\pi i(\xi + i)^{\beta + 1}}\biggr). 
\]

Thus, by defining the polynomial 
$P_{\beta + 1}(z) := (zP_{\beta}(z) - zP'_{\beta}(z) - \beta P_{\beta})/(2\pi i)$
we complete the induction.
Since $|1/(\xi + i)| \leq 1$ and $P_{\beta}$
has degree $\beta$, we obtain that there exists some 
polynomial $\wilde{P}_{\beta}$ of degree $\beta$ such that
\[
\biggl|\frac{P_{\beta}(\mu)}{(\xi + i)^{\beta}}\biggr|
\leq \wilde{P}_{\beta}\biggl(\frac{\eta}{\lan \xi\ran^{2}}\biggr).
\]

With this we conclude that
\[
|D_{\xi}^{\beta}e^{-\mu}| 
= \biggl|\frac{e^{-\mu}P_{\beta}(\mu)}{(\xi + i)^{\beta}}\biggr|
\leq \biggl|e^{-\mu}\wilde{P}_{\beta}\biggl(\frac{\eta}{\lan \xi\ran^{2}}\biggr)\biggr|
\leq C_{\beta}.
\]
\end{proof}


\subsection{Estimates for the solutions of the equations}


The purpose of this section is to finally prove 
Theorem \ref{thm: delta_bar_equation}. We start by proving the 
estimates for the ODEs \eqref{eqn: Fourier_ODE} 
which result from expanding in Fourier series the equation 
\eqref{eqn: delta_bar_equation}. We start with the following
elementary result.


\begin{proposition}
\label{propn: vanishing_moment}
If $f \in \cS(\bR)$ is such that $\int_{\bR}f(x_{1})dx_{1} = 0$, 
then there exists a unique $g \in \cS(\bR)$ such that 
$D_{x_{1}}g = f$. Moreover, if $f$ is compactly supported, 
then so is $g$. 

\bigskip
Similarly, if $f \in \cS(T)$ is such that 
$\int_{\bR}f(x_{1},x')dx_{1} = 0$ for all $x' \in \bT^{d}$, then 
there exists $g \in \cS(T)$ such that $D_{x_{1}}g = f$. Moreover, 
its Fourier coefficients $f_{k} \in \cS(\bR)$ satisfy that 
$\int_{\bR}f_{k}(x_{1}) = 0$ and $D_{x_{1}}g_{k} = f_{k}$. 
\end{proposition}


\begin{proof}
Let us first consider the case $f \in \cS(\bR)$. The uniqueness
follows from the Schwartz condition. For the existence we 
define
\[
g(x_{1}) := \int_{-\infty}^{x_{1}}f(y)dy
= -\int_{x_{1}}^{\infty}f(y)dy.
\]

To show that $g \in \cS(\bR)$ we use that if $L \geq 1$, then 
for any positive $m > 0$ we have
\[
\int_{-\infty}^{-L}\frac{1}{\lan y\ran^{m + 1}}dy, \
\int_{L}^{\infty}\frac{1}{\lan y\ran^{m + 1}}dy 
\cleq \frac{1}{L^{m}}.
\]

If $f$ were compactly supported, then the definition above
shows that $g$ shares the same support. For the case of $T$,
the first part of the proof follows exactly as before. Moreover,
by Fubini's theorem we have that
\[
\int_{\bR}f_{k}(x_{1})dx_{1}
= \int_{\bT^{d}}e_{-k}(x')
\biggl(\int_{\bR}f(x_{1},x')dx_{1}\biggr)dx'
= 0,
\]

and
\[
D_{x_{1}}g_{k}(x_{1})
= \int_{\bT^{d}}D_{x_{1}}g(x_{1},x')e_{-k}(x')dx'
= \int_{\bT^{d}}f(x_{1},x')e_{-k}(x')dx' = f_{k}(x_{1}).
\]
\end{proof}


If $f$ and $g$ are as in Proposition \ref{propn: vanishing_moment}, then
we define $D_{x_{1}}^{-1}f := g$.


\begin{theorem}
\label{thm: delta_bar_Fourier_ODE}
Let $\hb > 0$ and let $V = (F,G)$ satisfy \eqref{eqn: conditions}. 
For fixed $(m,\xi,t)\in \bZ^{d}\times \bR \times \bR^{d}$ the equation 
\[
(\xi + i)D_{x_{1}}u_{m} + \hb t\cdot mu_{m} 
= (\xi + i)F_{m} + \hb t\cdot G_{m}
\tag{\ref{eqn: Fourier_ODE}}
\]

has a unique solution $u_{m}(\cdot, \xi, t;\hb) \in C^{1}(\bR)$ with 
the decay condition $u_{m}(x_{1},\xi,t) \ra 0$ as $x_{1} \ra \pm \infty$. 
Moreover, we have that $u_{m}(\cdot,\xi, t) \in \cS(\bR)$, 
$u_{m}(\cdot,\cdot,t) \in C^{\infty}(\bR \times \bR)$,  
and for any $\alpha,\beta \geq 0$ it satisfies that
$D_{x_{1}}^{\alpha}D_{\xi}^{\beta}u_{m}(\cdot,\xi,t) \in \cS(\bR)$.
If $\hb t \cdot m = 0$, then $u_{m}$ is supported on $|x_{1}| \leq R$.
If $\hb t \cdot m \neq 0$, then $u_{m}$ vanishes in one of the components of 
$|x_{1}| > R$, and decays exponentially (depending on the product $\hb t \cdot m$) 
on the other component. In addition, it satisfies the bounds 
\begin{equation}
\label{eqn: delta_bar_derivatives}
|D_{x_{1}}^{\alpha}D_{\xi}^{\beta}u_{m}|
\cleq \lan \xi, \hb t\ran \|D_{x_{1}}^{\alpha}V_{m}\|_{L^{1}},
\end{equation}
\begin{equation}
\label{eqn: delta_bar_dual_toroidal}
|D_{x_{1}}^{\alpha}D_{\xi}^{\beta}
(u_{m}(\cdot,t_{1}) - u_{m}(\cdot,t_{2}))|
\cleq \hb|t_{1} - t_{2}|(\|D_{x_{1}}^{\alpha}V_{m}\|_{L^{1}} +
\lan \xi,\hb t_{1},\hb t_{2}\ran |m| \|D_{x_{1}}^{\alpha - 1}V_{m}\|_{L^{1}}).
\end{equation}

Moreover, for $|x_{1}| \geq 2R$ we have the linear decay bound
\begin{equation}
\label{eqn: delta_bar_decay}
|D_{x_{1}}^{\alpha}D_{\xi}^{\beta}u_{m}(x_{1})|
\cleq \frac{\lan \xi, \hb t\ran}{|x_{1}|}\|D_{x_{1}}^{\alpha - 1}V_{m}\|_{L^{1}}.
\end{equation}

The constants in the inequalities may depend on 
$\alpha$, $\beta$, $d$, but
are independent of $\hb$, $m$, $\xi$, $t$, $R$.
\end{theorem}


\begin{proof}
The uniqueness of such a solution follows from the remark
after Lemma \ref{lemma: ODE_vanishing}. From 
\eqref{eqn: conditions} and Lemma \ref{lemma: ODE_vanishing} 
we obtain the existence 
and that it belongs to $\cS(\bR)$. As in the previous proofs,
let $\mu = 2\pi i(\hb t\cdot m)/(\xi + i)$, so that 
$\Re(\mu) = 2\pi (\hb t\cdot m)/\lan \xi\ran^{2}$. We know that 
the form of the solution depends on the sign of $\Re(\mu)$;
more explicitly, proceeding as in the proof of 
Lemma \ref{lemma: ODE_vanishing}, we have that:
\begin{enumerate}
\item if $\hb t \cdot m = 0$, then 
\[
u_{m}(x_{1},\xi,t) = \frac{1}{\xi + i}[(\xi + i)D_{x_{1}}^{-1}F_{m}(x_{1})
+ \hb t\cdot D_{x_{1}}^{-1}G_{m}(x_{1})],
\]
\vspace{-6mm}

\item if $\hb t \cdot m > 0$, then 
\begin{align*}
u_{m}(x_{1},\xi,t) = \frac{1}{\xi + i}\biggl[&(\xi + i)D_{x_{1}}^{-1}F_{m}(x_{1})
+ \hb t\cdot D_{x_{1}}^{-1}G_{m}(x_{1}) \\
& - \mu\int_{-\infty}^{x_{1}}e^{-\mu(x_{1} - y_{1})}[(\xi + i)D_{y_{1}}^{-1}F_{m}(y_{1})
+ \hb t\cdot D_{y_{1}}^{-1}G_{m}(y_{1})]dy_{1}\biggr],
\end{align*}
\vspace{-6mm}

\item if $\hb t \cdot m < 0$, then 
\begin{align*}
u_{m}(x_{1},\xi,t) = \frac{1}{\xi + i}\biggl[&(\xi + i)D_{x_{1}}^{-1}F_{m}(x_{1})
+ \hb t\cdot D_{x_{1}}^{-1}G_{m}(x_{1}) \\
& + \mu\int_{x_{1}}^{\infty}e^{-\mu(x_{1} - y_{1})}[(\xi + i)D_{y_{1}}^{-1}F_{m}(y_{1})
+ \hb t\cdot D_{y_{1}}^{-1}G_{m}(y_{1})]dy_{1}\biggr].
\end{align*}
\vspace{-5mm}
\end{enumerate}

From Proposition \ref{propn: vanishing_moment} we know that 
$(\xi + i)D_{x_{1}}^{-1}F_{m} + \hb t\cdot D_{x_{1}}^{-1}G_{m}$ 
is a smooth compactly supported function. 
In particular, this implies that in the first case
the solution is compactly supported. In the second and third case,
this implies that the solutions vanish if $x_{1} < -R$ and $x_{1} > R$,  
respectively, and are decaying exponentials if $x_{1} > R$
and $x_{1} < -R$, respectively. Moreover, all the terms involved
($\xi + i$, $\mu$, $e^{-\mu(x_{1} - y_{1})}$, and 
$(\xi + i)D_{x_{1}}^{-1}F_{m} 
+ \hb t\cdot D_{x_{1}}^{-1}G_{m}$) are differentiable 
with respect to $x_{1}$ and $\xi$, and the possible different
cases depend only on $m$ and $t$ (namely on the sign of 
$\hb t \cdot m$), and not on $x_{1}$ or $\xi$. This implies that 
$u_{m}(\cdot,\cdot,t) \in C^{\infty}(\bR \times \bR)$ 
and $D_{x_{1}}^{\alpha}D_{\xi}^{\beta}u_{m}(\cdot,\xi,t) \in \cS(\bR)$.

\bigskip
To prove the estimates \eqref{eqn: delta_bar_derivatives} we 
succesively differentiate the equation \eqref{eqn: Fourier_ODE}
to show that $D_{x_{1}}D_{\xi}u_{m}(\cdot,\xi,t)$ (which we know
is in $\cS(\bR)$) solves certain ODE, and then use the estimates 
for the unique solution in $\cS(\bR)$ from Lemma \ref{lemma: ODE_vanishing}. 
Differentiating the equation, we see that if $\alpha \geq 0$ 
then $D_{x_{1}}^{\alpha}u_{m}$ solves the equation 
\begin{equation}
\label{eqn: equation_D_{x_{1}}^{alpha}}
(\xi + i)D_{x_{1}}[D_{x_{1}}^{\alpha}u_{m}] 
+ \hb t\cdot m [D_{x_{1}}^{\alpha}u_{m}] 
= (\xi + i)D_{x_{1}}^{\alpha}F_{m} 
+ \hb t\cdot D_{x_{1}}^{\alpha}G_{m}.
\end{equation}

By \eqref{eqn: conditions} and Lemma \ref{lemma: ODE_vanishing} 
we can bound
\begin{equation}
\label{eqn: bounds_D_{x_{1}}^{alpha}_L1}
\|D_{x_{1}}^{\alpha}u_{m}\|_{L^{1}} 
\cleq \lan \xi\ran \|D_{x_{1}}^{\alpha - 1}F_{m}\|_{L^{1}}
+ |\hb t| \|D_{x_{1}}^{\alpha - 1}G_{m}\|_{L^{1}} 
\cleq \lan \xi, \hb t \ran \|D_{x_{1}}^{\alpha - 1}V_{m}\|_{L^{1}},
\end{equation}
\begin{equation}
\label{eqn: bounds_D_{x_{1}}^{alpha}}
|D_{x_{1}}^{\alpha}u_{m}| 
\cleq \|D_{x_{1}}^{\alpha + 1} u_{m}\|_{L^{1}} 
\cleq \lan \xi, \hb t\ran\|D_{x_{1}}^{\alpha}V_{m}\|_{L^{1}}.
\end{equation}

Differentiating \eqref{eqn: equation_D_{x_{1}}^{alpha}} 
with respect to $\xi$ gives that 
$D_{x_{1}}^{\alpha}D_{\xi}u_{m}$ solves the equation
\begin{equation}
\label{eqn: equation_D_{x_{1}}^{alpha}D_{xi}}
(\xi + i)D_{x_{1}}[D_{x_{1}}^{\alpha}D_{\xi}u_{m}] 
+ \hb t\cdot m [D_{x_{1}}^{\alpha}D_{\xi}u_{m}] 
= \frac{1}{2\pi i}(D_{x_{1}}^{\alpha}F_{m} 
- D_{x_{1}}^{\alpha + 1}u_{m}).
\end{equation}

By \eqref{eqn: conditions}, Lemma \ref{lemma: ODE_vanishing}, and 
\eqref{eqn: bounds_D_{x_{1}}^{alpha}_L1} we can bound
\begin{equation}
\label{eqn: bounds_D_{x_{1}}^{alpha}D_{xi}_L1}
\|D_{x_{1}}^{\alpha}D_{\xi}u_{m}\|_{L^{1}}
\cleq \|D_{x_{1}}^{\alpha - 1}F_{m}\|_{L^{1}} 
+ \|D_{x_{1}}^{\alpha}u_{m}\|_{L^{1}}
\cleq \lan \xi, \hb t\ran \|D_{x_{1}}^{\alpha - 1}V_{m}\|_{L^{1}},
\end{equation}
\begin{equation}
\label{eqn: bounds_D_{x_{1}}^{alpha}D_{xi}}
|D_{x_{1}}^{\alpha}D_{\xi}u_{m}| 
\cleq \|D_{x_{1}}^{\alpha + 1}D_{\xi}u_{m}\|_{L^{1}}
\cleq \lan \xi, \hb t\ran \|D_{x_{1}}^{\alpha}V_{m}\|_{L^{1}}.
\end{equation}

By induction on $\beta \geq 2$, differentiating 
\eqref{eqn: equation_D_{x_{1}}^{alpha}D_{xi}}
with respect to $\xi$ gives that
$D_{x_{1}}^{\alpha}D_{\xi}^{\beta}u_{m}$ solves
the equation
\begin{equation}
\label{eqn: equation_D_{x_{1}}^{alpha}D_{xi}^{beta}}
(\xi + i)D_{x_{1}}[D_{x_{1}}^{\alpha}D_{\xi}^{\beta}u_{m}] 
+ \hb t\cdot m [D_{x_{1}}^{\alpha}D_{\xi}^{\beta}u_{m}] 
= \frac{-\beta}{2\pi i} D_{x_{1}}^{\alpha + 1}D_{\xi}^{\beta - 1}u_{m}.
\end{equation}

From \eqref{eqn: equation_D_{x_{1}}^{alpha}D_{xi}^{beta}},
Lemma \ref{lemma: ODE_vanishing}, and 
\eqref{eqn: bounds_D_{x_{1}}^{alpha}D_{xi}_L1} we obtain
\begin{equation}
\label{eqn: bounds_D_{x_{1}}^{alpha}D_{xi}^{beta}_L1}
\|D_{x_{1}}^{\alpha}D_{\xi}^{\beta}u_{m}\|_{L^{1}}
\cleq \|D_{x_{1}}^{\alpha}D_{\xi}^{\beta - 1}u_{m}\|_{L^{1}}
\cleq \ldots 
\cleq \|D_{x_{1}}^{\alpha}D_{\xi}u_{m}\|_{L^{1}}
\cleq \lan \xi, \hb t\ran \|D_{x_{1}}^{\alpha - 1}V_{m}\|_{L^{1}},
\end{equation}
\begin{equation}
\label{eqn: bounds_D_{x_{1}}^{alpha}D_{xi}^{beta}}
|D_{x_{1}}^{\alpha}D_{\xi}^{\beta}u_{m}| 
\cleq \|D_{x_{1}}^{\alpha + 1}D_{\xi}^{\beta}u_{m}\|_{L^{1}}
\cleq \lan \xi, \hb t\ran \|D_{x_{1}}^{\alpha}V_{m}\|_{L^{1}}.
\end{equation}

We have shown \eqref{eqn: delta_bar_derivatives} through
\eqref{eqn: bounds_D_{x_{1}}^{alpha}}, 
\eqref{eqn: bounds_D_{x_{1}}^{alpha}D_{xi}}, and
\eqref{eqn: bounds_D_{x_{1}}^{alpha}D_{xi}^{beta}}. 						
Now we prove \eqref{eqn: delta_bar_dual_toroidal}.
Let us denote 
$u_{m}^{j}(x_{1},\xi) := u_{m}(x_{1},\xi,t_{j})$ and
recall that
$D_{x_{1}}^{\alpha}D_{\xi}^{\beta}u_{m}^{j}(\cdot,\xi) \in \cS(\bR)$
for any $\alpha,\beta \geq 0$.
There are two cases to consider: when both products 
$\hb t_{j}\cdot m$ vanish, and when at least one of them 
does not vanish. In the first case we have that 
\[
u_{m}^{j}(x_{1},\xi)
= \frac{1}{\xi + i}[(\xi + i)D_{x_{1}}^{-1}F_{m}(x_{1})
+ \hb t_{j}\cdot D_{x_{1}}^{-1}G_{m}(x_{1})],
\]

and so 
\[
u_{m}^{1}(x_{1},\xi) - u_{m}^{2}(x_{1},\xi)
= \frac{\hb(t_{1} - t_{2})}{\xi + i}
\cdot D_{x_{1}}^{-1}G_{m}(x_{1}).
\]

It follows directly from this and the Fundamental Theorem
of Calculus that 
\begin{equation}
\label{eqn: bounds_difference_easy}
|D_{x_{1}}^{\alpha}D_{\xi}^{\beta}
(u_{m}^{1} - u_{m}^{2})|
\cleq \hb|t_{1} - t_{2}|\|D_{x_{1}}^{\alpha}V_{m}\|_{L^{1}}.
\end{equation}

Suppose now that $\hb t_{1} \cdot m \neq 0$. Substracting the 
equations \eqref{eqn: equation_D_{x_{1}}^{alpha}} for 
$D_{x_{1}}^{\alpha}u_{m}^{j}$ we obtain that
\begin{equation}
\label{eqn: equation_difference_D_{x_{1}}^{alpha}}
(\xi + i)D_{x_{1}}[D_{x_{1}}^{\alpha}(u_{m}^{1} - u_{m}^{2})]
+ \hb t_{2}\cdot m[D_{x_{1}}^{\alpha}(u_{m}^{1} - u_{m}^{2})]
= \hb(t_{1} - t_{2})\cdot D_{x_{1}}^{\alpha}G_{m} 
- \hb(t_{1} - t_{2})\cdot mD_{x_{1}}^{\alpha}u_{m}^{1}.
\end{equation}

By \eqref{eqn: conditions} and Lemma \ref{lemma: ODE_vanishing} 
we have that the condition $\hb t_{1}\cdot m \neq 0$ implies 
that $u_{m}^{1} = D_{x_{1}}w$ for some 
$w \in \cS(\bR)$. Therefore, the difference 
$D_{x_{1}}^{\alpha}(u_{m}^{1} - u_{m}^{2})$
(which we know is in $\cS(\bR)$) is the unique solution
in $\cS(\bR)$ to a differential equation, 
\eqref{eqn: equation_difference_D_{x_{1}}^{alpha}},
as in the setting of Lemma \ref{lemma: ODE_vanishing}. 
By \eqref{eqn: conditions}, 
\eqref{eqn: equation_difference_D_{x_{1}}^{alpha}}, 
Lemma \ref{lemma: ODE_vanishing}, and
\eqref{eqn: bounds_D_{x_{1}}^{alpha}_L1}
we obtain
\begin{align}
\label{eqn: bounds_difference_D_{x_{1}}^{alpha}}
|D_{x_{1}}^{\alpha}(u_{m}^{1} - u_{m}^{2})| 
& \cleq \|D_{x_{1}}^{\alpha + 1}(u_{m}^{1} - u_{m}^{2})\|_{L^{1}} \nonumber \\
& \cleq \hb |t_{1} - t_{2}|(\|D_{x_{1}}^{\alpha}G_{m}\|_{L^{1}} 
+ |m|\|D_{x_{1}}^{\alpha}u_{m}^{1}\|_{L^{1}}) \nonumber \\
& \cleq \hb |t_{1} - t_{2}|(\|D_{x_{1}}^{\alpha}V_{m}\|_{L^{1}} 
+ \lan \xi,\hb t_{1}\ran|m|\|D_{x_{1}}^{\alpha-1}V_{m}\|_{L^{1}}).
\end{align}

\begin{remark}
This step shows why it is useful to have at disposal the 
$L^{1} \ra L^{\infty}$ estimates and not the 
$L^{\infty}\ra L^{\infty}$ estimates alone. 
\end{remark}

Similarly, substracting the 
equations \eqref{eqn: equation_D_{x_{1}}^{alpha}D_{xi}} 
for $D_{x_{1}}^{\alpha}D_{\xi}u_{m}^{j}$ we obtain that
\begin{align*}
(\xi + i)D_{x_{1}}[D_{x_{1}}^{\alpha}D_{\xi}(u_{m}^{1} - u_{m}^{2})]
& + \hb t_{2}\cdot m[D_{x_{1}}^{\alpha}D_{\xi}(u_{m}^{1} - u_{m}^{2})] \\
& = -\frac{1}{2\pi i}D_{x_{1}}^{\alpha + 1}(u_{m}^{1} - u_{m}^{2})
- \hb(t_{1} - t_{2})\cdot mD_{x_{1}}^{\alpha}D_{\xi}u_{m}^{1}.
\end{align*}

For $\beta \geq 2$ the equation takes the same form. Indeed,
substracting the equations \eqref{eqn: equation_D_{x_{1}}^{alpha}D_{xi}^{beta}} 
for $D_{x_{1}}^{\alpha}D_{\xi}^{\beta}u_{m}^{j}$ we obtain that
\begin{align}
\label{eqn: equation_difference_D_{x_{1}}^{alpha}D_{xi}^{beta}}
(\xi + i)D_{x_{1}}[D_{x_{1}}^{\alpha}D_{\xi}^{\beta}(u_{m}^{1} - u_{m}^{2})]
& + \hb t_{2}\cdot m[D_{x_{1}}^{\alpha}D_{\xi}^{\beta}(u_{m}^{1} - u_{m}^{2})] \nonumber \\
& = -\frac{\beta}{2\pi i}D_{x_{1}}^{\alpha + 1}
D_{\xi}^{\beta - 1}(u_{m}^{1} - u_{m}^{2})
- \hb (t_{1} - t_{2})\cdot mD_{x_{1}}^{\alpha}D_{\xi}^{\beta}u_{m}^{1}.
\end{align}

Since $u_{m}^{1} = D_{x_{1}}w$, then we are in the setting of 
Lemma \ref{lemma: ODE_vanishing} as before.
By \eqref{eqn: equation_difference_D_{x_{1}}^{alpha}D_{xi}^{beta}},
Lemma \ref{lemma: ODE_vanishing}, 
\eqref{eqn: bounds_D_{x_{1}}^{alpha}D_{xi}_L1}, and 
\eqref{eqn: bounds_D_{x_{1}}^{alpha}D_{xi}^{beta}_L1}
we can bound 
\begin{align*}
\|D_{x_{1}}^{\alpha + 1}D_{\xi}^{\beta}(u_{m}^{1} - u_{m}^{2})\|_{L^{1}}
& \cleq \|D_{x_{1}}^{\alpha + 1}
D_{\xi}^{\beta - 1}(u_{m}^{1} - u_{m}^{2})\|_{L^{1}}
+ \hb |t_{1} - t_{2}||m|\|D_{x_{1}}^{\alpha}D_{\xi}^{\beta}u_{m}^{1}\|_{L^{1}} \\
& \cleq \|D_{x_{1}}^{\alpha + 1}
D_{\xi}^{\beta - 1}(u_{m}^{1} - u_{m}^{2})\|_{L^{1}}
+ \hb |t_{1} - t_{2}|\lan \xi, \hb t_{1}\ran|m| \|D_{x_{1}}^{\alpha-1}V_{m}\|_{L^{1}}.
\end{align*}

Iterating this and using \eqref{eqn: bounds_difference_D_{x_{1}}^{alpha}}
we obtain that 
\begin{align*}
\|D_{x_{1}}^{\alpha + 1}D_{\xi}^{\beta}(u_{m}^{1} - u_{m}^{2})\|_{L^{1}}
\cleq \ldots 
& \cleq \|D_{x_{1}}^{\alpha + 1}(u_{m}^{1} - u_{m}^{2})\|_{L^{1}}
+ \hb|t_{1} - t_{2}|\lan \xi, \hb t_{1}\ran|m|
\|D_{x_{1}}^{\alpha-1}V_{m}\|_{L^{1}} \\
&\cleq  \hb |t_{1} - t_{2}|(\|D_{x_{1}}^{\alpha}V_{m}\|_{L^{1}} 
+ \lan \xi,\hb t_{1}\ran|m|\|D_{x_{1}}^{\alpha-1}V_{m}\|_{L^{1}}).
\end{align*}

From this and the Fundamental Theorem of Calculus we 
conclude that 
\begin{equation}
\label{eqn: bounds_difference_D_{x_{1}}^{alpha}D_{xi}^{beta}}
|D_{x_{1}}^{\alpha}D_{\xi}^{\beta}(u_{m}^{1} - u_{m}^{2})|
\cleq \|D_{x_{1}}^{\alpha + 1}D_{\xi}^{\beta}(u_{m}^{1} - u_{m}^{2})\|_{L^{1}}
\cleq \hb |t_{1} - t_{2}|(\|D_{x_{1}}^{\alpha}V_{m}\|_{L^{1}} 
+ \lan \xi,\hb t_{1}\ran|m|\|D_{x_{1}}^{\alpha-1}V_{m}\|_{L^{1}}).
\end{equation}

We have shown \eqref{eqn: delta_bar_dual_toroidal} through
\eqref{eqn: bounds_difference_easy},
\eqref{eqn: bounds_difference_D_{x_{1}}^{alpha}}, and
\eqref{eqn: bounds_difference_D_{x_{1}}^{alpha}D_{xi}^{beta}}. 				
Finally, we prove the decay estimates \eqref{eqn: delta_bar_decay}
for the solution. In the case
$\hb t\cdot m = 0$ there is nothing to prove, as the solution is 
compactly supported on $|x_{1}| \leq R$. The other two cases
are analogous, so we only consider the case $\hb t\cdot m > 0$,
so that $\Re(\mu) > 0$.
For this one we know the solution vanishes if $x_{1} < -R$,
so we only have to deal with $x_{1} > R$. We can
rewrite the solution for $x_{1} > R$ as 
\[
u_{m}(x_{1},\xi,t) = \frac{-\mu e^{-\mu(x_{1} - R)}}{\xi + i}
\int_{-R}^{R}e^{-\mu(R - y_{1})}[(\xi + i)D_{y_{1}}^{-1}F_{m}(y_{1})
+ \hb t\cdot D_{y_{1}}^{-1}G_{m}(y_{1})]dy_{1}.
\]

Integrating by parts we obtain that
\begin{align*}
D_{x_{1}}^{\alpha}u_{m} 
& = \biggl(\frac{-\mu}{2\pi i}\biggr)^{\alpha}u_{m} \\
& = \frac{-\mu e^{-\mu(x_{1} - R)}}{\xi + i}
\int_{-R}^{R}[(-D_{y_{1}})^{\alpha}
e^{-\mu(R - y_{1})}][(\xi + i)D_{y_{1}}^{-1}F_{m}(y_{1})
+ \hb t\cdot D_{y_{1}}^{-1}G_{m}(y_{1})]dy_{1} \\
& = \frac{-\mu e^{-\mu(x_{1} - R)}}{\xi + i}
\int_{-R}^{R}e^{-\mu(R - y_{1})}[(\xi + i)D_{y_{1}}^{\alpha - 1}F_{m}(y_{1})
+ \hb t\cdot D_{y_{1}}^{\alpha - 1}G_{m}(y_{1})]dy_{1} =: \varphi \psi.
\end{align*}

We will prove that $\varphi$ and its derivatives (with respect to $\xi$)
have the required decay, while $\psi$ and its derivatives remain 
bounded. Let us observe that $\mu(R - y_{1}) = i\eta/(\xi + i)$ with
$\eta \geq 0$ for $y_{1} \in [-R,R]$, so that we are in the setting of 
Lemma \ref{lemma: exponential}. It follows from this that 
$|D_{\xi}^{\beta}\psi| 
\cleq \lan \xi, \hb t\ran \|D_{x_{1}}^{\alpha - 1}V_{m}\|_{L^{1}}$,
where we allow the constant of the inequality to depend on $\beta$.
By a similar induction as in the proof of Lemma \ref{lemma: exponential},
we obtain that 
\[
D_{\xi}^{\beta}\biggl(\frac{\mu e^{-\mu s}}{\xi + i}\biggr) 
= \frac{\mu e^{-\mu s}Q_{\beta}(\mu s)}{(\xi + i)^{\beta + 1}},
\]

where $Q_{\beta}$ is a polynomial of degree $\beta$ with coefficients 
depending only on $\beta$. Multiplying by $s$ we obtain 
\[
sD_{\xi}^{\beta}\biggl(\frac{\mu e^{-\mu s}}{\xi + i}\biggr) 
= \frac{e^{-\mu s}Q_{\beta}^{*}(\mu s)}{(\xi + i)^{\beta + 1}},
\]

where $Q_{\beta}^{*}(z) = zQ_{\beta}(z)$ is a polynomial of degree
$\beta + 1$. Again, as $|\xi + i| \geq 1$, there exists a 
polynomial $\wilde{Q}_{\beta}^{*}$ of degree $\beta + 1$ such that 
\[
\biggl|\frac{Q_{\beta}^{*}(\mu s)}{(\xi + i)^{\beta + 1}}\biggr| 
\leq \wilde{Q}_{\beta}^{*}\biggl(\frac{|\mu s|}{\lan \xi\ran}\biggr).
\]

We have that $\mu (x_{1} - R) = i\eta/(\xi + i)$ with $\eta \geq 0$ 
for $x_{1} \geq R$, so that we are in the setting of 
Lemma \ref{lemma: exponential}. It follows from the previous inequalities
that
\[
|(x_{1} - R)D_{\xi}^{\beta}\varphi|
\leq \biggl|e^{-i\eta/(\xi + i)}\wilde{Q}_{\beta}^{*}
\bigg(\frac{\eta}{\lan \xi\ran^{2}}\biggr)\biggl| 
\cleq 1,
\]

where we allow the constant in the last inequality to depend on $\beta$.
In particular, for $x_{1} \geq 2R$ we have that $x_{1} - R \geq x_{1}/2$,
so that the previous inequality gives $|D_{\xi}^{\beta}\varphi| \cleq 1/|x_{1}|$.
Combining the bounds for $\psi$ and $\varphi$ gives the decay
estimate that we wanted.
\end{proof}


\begin{remark}
It may not be relevant for this particular problem, but the condition
$m \in \bZ^{d}$ does not seem to intervene in the proof of the result.
\end{remark}


\begin{remark}
It does not seem relevant, but the constants in the inequalities are 
independent of $R$. It only plays a role when we need that 
$|x_{1}| \geq 2R$ to have the decay.
\end{remark}


Let us discuss briefly why the vanishing moment conditions 
were important in the previous proof. First, they appear
in proving the differences estimates. 
If $t_{1}, t_{2} \in \bR^{d}$ are such that 
$\hb t_{1} \cdot m > 0 > \hb t_{2} \cdot m$, then for $x_{1} > R$
we would have $u_{m}^{2}(x_{1}) = 0$ and 
\[
u_{m}^{1}(x_{1}) = \frac{2\pi i e^{-\mu_{1}(x_{1} - R)}}{\xi + i}
\int_{-R}^{R}e^{-\mu_{1}(R - y_{1})}[(\xi + i)F_{m}(y_{1}) 
+ \hb t_{1}\cdot G_{m}(y_{1})]dy_{1}.
\]

It seems that there is no way to estimate $u_{m}^{1} - u_{m}^{2}$ 
in terms of the difference $\hb |t_{1} - t_{2}|$. For instance, letting 
$\mu_{1} \ra 0$, we obtain that the difference would be approximately 
\[
\frac{2\pi i}{\xi + i}\int_{-R}^{R}[(\xi + i)F_{m}(y_{1}) 
+ \hb t_{1}\cdot G_{m}(y_{1})]dy_{1},
\]

which suggests the need of the vanishing moment condition.
They also show up with a crucial improvement for the decay 
estimates. In the case $\hb t \cdot m > 0$ and $x_{1} > R$, without the
vanishing moment condition, we would only have the exponential decay
\[
u_{m}(x_{1}) = \frac{2\pi i e^{-\mu(x_{1} - R)}}{\xi + i}
\int_{-R}^{R}e^{-\mu(R - y_{1})}[(\xi + i)F_{m}(y_{1}) 
+ \hb t\cdot G_{m}(y_{1})]dy_{1}.
\]

It may happen that $\mu$ is small, for instance if $t\cdot m = 1$, making 
the exponential decay very slow. In this case, the best estimates we seem
to obtain are
\[
|u_{m}| \cleq \frac{\lan \xi\ran}{\mu|x_{1} - R|},
\]

with $1/\mu$ potentially being as big as $1/\hb$. In our proof, what allows us 
to get better estimates is the presence of the factor $\mu$ in front of the 
integral. This factor comes from integrating by parts using the vanishing
moment condition. 

\bigskip
We rewrite the previous estimates to depend on $V$ and no longer on 
the Fourier coefficients $V_{m}$.


\begin{corollary}
\label{cor: delta_bar_Fourier_ODE_II}
Let $V = (F,G)$ satisfy \eqref{eqn: conditions}, and let $u_{m}$ be 
the solution from Theorem \ref{thm: delta_bar_Fourier_ODE}.
Then, for any $M \geq 0$ we have 
\begin{equation}
\label{eqn: delta_bar_derivatives_II}
|D_{x_{1}}^{\alpha}D_{\xi}^{\beta}u_{m}| 
\cleq \lan \xi,\hb t\ran\lan m\ran^{-2M}\|V\|_{W^{\alpha + 2M, 1}(T)},
\end{equation}
\begin{equation}
\label{eqn: delta_bar_dual_toroidal_II}
|D_{x_{1}}^{\alpha}D_{\xi}^{\beta}
(u_{m}(\cdot,t_{1}) - u_{m}(\cdot,t_{2}))|
\cleq \hb |t_{1} - t_{2}|\lan \xi, \hb t_{1}, \hb t_{2}\ran\lan m\ran^{-2M + 1}
\|V\|_{W^{\alpha + 2M, 1}(T)}.
\end{equation}

Moreover, for $|x_{1}| \geq 2R$ we have the linear decay bound
\begin{equation}
\label{eqn: delta_bar_decay_II}
|D_{x_{1}}^{\alpha}D_{\xi}^{\beta}u_{m}(x_{1})|
\cleq \frac{\lan \xi, \hb t\ran}{|x_{1}|}\lan m\ran^{-2M}
\|V\|_{W^{\alpha + 2M,1}(T)}.
\end{equation}

The constants in the inequalities may depend on 
$\alpha$, $\beta$, $M$, $d$, $R$, but
are independent of $\hb$, $m$, $\xi$, $t$.
\end{corollary}


\begin{proof}
We use that $V$ is compactly supported and the 
Fundamental Theorem of Calculus to get that 
\[
\|D_{x_{1}}^{\alpha - 1}V_{m}\|_{L^{1}}
\cleq \|D_{x_{1}}^{\alpha - 1}V_{m}\|_{L^{\infty}}
\cleq \|D_{x_{1}}^{\alpha}V_{m}\|_{L^{1}},
\]

where we allow the constant of the inequality to depend 
on $R$. For $\alpha \geq 0$ and any $M \geq 0$, we have 
from Proposition \ref{propn: Fourier_cylinder} that
\[
\|D_{x_{1}}^{\alpha}V_{m}\|_{L^{1}(\bR)}
\cleq \lan m\ran^{-2M}\|D_{x_{1}}^{\alpha}V\|_{W^{2M,1}(T)}
\cleq \lan m\ran^{-2M}\|V\|_{W^{\alpha + 2M,1}(T)},
\]

where we allow the constant of the inequality to depend on 
$M$. Then the conclusion follows from 
Theorem \ref{thm: delta_bar_Fourier_ODE}.
\end{proof}


\begin{proof}[Proof of Theorem \ref{thm: delta_bar_equation}]
Let $u(\cdot,\xi,t;\hb) \in C^{\infty}(T)$ solve the equation
\eqref{eqn: delta_bar_equation} and satisfy the decay
condition $\|u(x_{1},\cdot,\xi,t)\|_{L^{\infty}(\bT^{d})} \ra 0$
as $x_{1} \ra \pm \infty$. Then its Fourier coefficients 
must solve the ODEs \eqref{eqn: Fourier_ODE}
and satisfy the decay conditions as 
$x_{1} \ra \pm\infty$. Under \eqref{eqn: conditions}, we know from 
Theorem \ref{thm: delta_bar_Fourier_ODE} the  existence and 
uniqueness of such solutions. Let $\{u_{m}\}$ be such
and define 
\[
u(x_{1},x',\xi,t;\hb) := \sum_{m \in \bZ^{d}}
u_{m}(x_{1},\xi,t;\hb)e_{m}(x').
\]

The control on derivatives of $u_{m}$ from 
Corollary \ref{cor: delta_bar_Fourier_ODE_II} implies that 
$u(\cdot, t) \in C^{\infty}(\bR \times \bT^{d} \times \bR)$ and
it solves \eqref{eqn: delta_bar_equation}. Moreover, we can 
bound
\[
|D_{x_{1}}^{\alpha}D_{x'}^{\beta}D_{\xi}^{\gamma}u|
\leq \sum_{m \in \bZ^{d}}|m|^{\beta}
|D_{x_{1}}^{\alpha}D_{\xi}^{\gamma}u_{m}|
\cleq \sum_{m \in \bZ^{d}}|m|^{\beta}\lan m\ran^{-2M} 
\lan \xi, \hb t\ran
\cleq \lan \xi, \hb t\ran,
\]

by taking $M = M(\beta,d)$ sufficiently large. The constants in the
inequalities may depend on the admissible quantities, but are independent
of $\hb$, $\xi$, $t$. In the same way we prove the difference estimate 
from \eqref{eqn: solution_bounds} and the decay estimate 
\eqref{eqn: solution_decay}; finally, the existence of a solution 
satisfying the decay condition follows from 
\eqref{eqn: solution_decay}.
\end{proof}


\subsection{Explicit definition of the symbol and properties}


The purpose of this section is to prove the conjugation from
Theorem \ref{thm: conjugation}. Recall that we have
\[
(\Delta_{\hb} + 2\hb V_{\hb})A = A\Delta_{\hb} +
2\hb\Op_{\hb}((\xi + i)D_{x_{1}}a + \hb t\cdot D_{x'}a 
+ (\xi + i)Fa + \hb t\cdot Ga)
+ \hb^{2}\Op_{\hb}(D^{2}a + 2V\cdot Da).
\]

Let us define $a = e^{-u\phi}$, where $u(x_{1},x',\xi,t;\hb)$ is
the solution from Theorem \ref{thm: delta_bar_equation} to 
\[
(\xi + i)D_{x_{1}}u + \hb t\cdot D_{x'}u = (\xi + i)F + \hb t\cdot G,
\]

and $\phi(x_{1},\xi,t; \hb)$ is defined as follows. Let $\psi(s)$ 
be a (nonnegative, even, decreasing) smooth function such 
that $\psi(s) \equiv 1$ for $|s| \leq 1$ and $\psi(s) \equiv 0$ 
for $|s| \geq 2$. Define 
\[
\phi_{0}(\xi,t; \hb) := \psi(\xi)\psi(4(|\hb t| - 1)), \ \ 
\phi(x_{1},\xi,t; \hb) := \psi(\hb^{\theta}x_{1})\phi_{0}(\xi,t; \hb),
\]

with $\theta > 0$ to be defined later.


\begin{remark}
Note that $\phi_{0}$ and $\phi$ vanish if $|\hb t| \leq 1/2$, in particular 
they vanish for $t$ near the origin and so these are smooth in all 
their variables. Moreover, it is a special zero order symbol.
\end{remark}


\begin{remark}
The cutoff in $x_{1}$ is unnecessary for the properties of $A$ and $B$,
but it is actually needed for the properties of $R$.
\end{remark}


The estimates from Theorem \ref{thm: delta_bar_equation} give that $u\phi$ 
is a special zero order symbol. From Corollary \ref{cor: invertibility} we obtain that, 
for small $\hb$, the operator $A := \Op_{\hb}(a)$  is invertible and its inverse 
is uniformly bounded (in $\hb$) in $H^{s}_{\pm \delta, \hb}(T)$. We also have
\vspace{-2mm}
\begin{align*}
(\xi + i)D_{x_{1}}a & + \hb t\cdot D_{x'}a + (\xi + i)Fa + \hb t\cdot Ga \\
& = a[-(\xi + i)D_{x_{1}}(u\phi) - \hb t\cdot D_{x'}(u\phi) + (\xi + i)F + \hb t\cdot G] \\
& = a(1 - \phi)[(\xi + i)F + \hb t\cdot G] - (\xi + i)auD_{x_{1}}\phi \\ 
& = a(1 - \phi)[(\xi + i)F + \hb t\cdot G] 
- \frac{1}{2\pi i}\hb^{\theta}(\xi + i)au\psi'(\hb^{\theta}x_{1})\phi_{0}.
\end{align*}

Note that $(1 - \phi)$ vanishes if $|x_{1}| \leq \hb^{-\theta}$, $|\xi| \leq 1$ and 
$||\hb t| - 1| \leq 1/4$. Let us rewrite 
\[
1 - \phi = (1 - \phi_{0}) + \phi_{0}\cdot (1 - \psi(\hb^{\theta}x_{1})),
\]

and observe that $1 - \psi(\hb^{\theta}x_{1})$ vanishes for $|x_{1}| \leq \hb^{-\theta}$. 
Because $F$ and $G$ are supported on $|x_{1}| \leq R$, then, for
small enough $\hb$, say $\hb^{-\theta} \geq R$, we have 
\[
a(1 - \phi)[(\xi + i)F + \hb t\cdot G] = a(1 - \phi_{0})[(\xi + i)F + \hb t\cdot G].
\]

The function $(1 - \phi_{0})$ vanishes if $|\xi| \leq 1$ and $||\hb t| - 1| \leq 1/4$,
and outside of this set the operator $\Delta_{\hb}$ is elliptic, i.e. its symbol 
$(\xi + i)^{2} + |\hb t|^{2}$ does not vanish. 


\begin{proposition}
\label{propn: ellipticity}
Outside of the set $\{|\xi| \leq 1\} \cap \{||\hb t| - 1| \leq 1/4\}$, we can bound 
\[
|(\xi + i)^{2} + |\hb t|^{2}| \cgeq \lan \xi,\hb t\ran^{2}
\]
\end{proposition}


\begin{proof}
Let $s := (\xi + i)^{2} + |\hb t|^{2}$. Let us observe that
\[
|s| = ([\xi^{2} + |\hb t|^{2} - 1]^{2} + 4\xi^{2})^{1/2}
= (\xi^{4} + 2\xi^{2}(|\hb t|^{2} + 1) + (|\hb t|^{2} - 1)^{2})^{1/2} 
\geq \xi^{2} + ||\hb t|^{2} - 1|.
\]

This proves that $|s| \geq \xi^{2}$. If $|\xi| \geq 1$, this
also proves $|s| \geq 1$. If $|\xi| \leq 1$, then we must have
$||\hb t| - 1| \geq 1/4$, and this gives
$|s| \geq ||\hb t| + 1|/4 \geq 1/4$. We have shown
that $|s| \cgeq 1$. Thus, all that
remains to prove is that $|s| \cgeq |\hb t|^{2}$. If $|\hb t| \leq 2$,
then this follows from before. If $|\hb t| \geq 2$, then 
$||\hb t|^{2} - 1| \geq |\hb t|^{2}/2$, and the conclusion follows. 
\end{proof}


Let us consider the function 
\[
r(x_{1},x',\xi,t; \hb) 
:= \frac{1 - \phi_{0}}{(\xi + i)^{2} + |\hb t|^{2}}[(\xi + i)F + \hb t\cdot G]
= \frac{(1 - \phi_{0})\lan \xi, \hb t\ran^{2}}{(\xi + i)^{2} + |\hb t|^{2}}
\cdot \frac{(\xi + i)F + \hb t \cdot G}{\lan \xi, \hb t\ran^{2}}
=: q_{1}\cdot q_{2}.
\]

The functions $\xi/\lan \xi,\hb t\ran^{2}$ and $\hb t/\lan \xi,\hb t\ran^{2}$
are special zero order symbols, and so $q_{2}$ is a special zero order
symbol. We know that 
$q_{1}$ is supported outside of $\{|\xi|\leq 1\}\cap \{||\hb t| - 1| \leq 1/4\}$
and $1 - \phi_{0}$ is a special zero order symbol. By 
induction we can show that 
\[
D_{\xi}^{\alpha}D_{t}^{\beta}\frac{\lan \xi,\hb t\ran^{2}}
{(\xi + i)^{2} + |\hb t|^{2}}
= \frac{\hb^{\beta}P_{\alpha,\beta}(\xi,\hb t)}
{((\xi + i)^{2} + |\hb t|^{2})^{\alpha + \beta + 1}},
\]

for some polynomial $P_{\alpha,\beta}$ of degree at most 
$\alpha + \beta + 2$. Outside of $\{|\xi|\leq 1\}\cap \{||\hb t| - 1| \leq 1/4\}$, 
from Proposition \ref{propn: ellipticity}, we can bound 
\[
\biggl|D_{\xi}^{\alpha}D_{t}^{\beta}\frac{\lan \xi,\hb t\ran^{2}}
{(\xi + i)^{2} + |\hb t|^{2}}\biggr|
= \biggl|\frac{\hb^{\beta}P_{\alpha,\beta}(\xi, \hb t)}
{((\xi + i)^{2} + |\hb t|^{2})^{\alpha + \beta + 1}}\biggr|
\cleq \frac{\hb^{\beta}\lan \xi,\hb t\ran^{\alpha + \beta + 2}}
{\lan \xi,\hb t\ran^{2(\alpha + \beta + 1)}} \leq \hb^{\beta}.
\]

This proves that $q_{1}$ is a special zero order symbol, and by 
Proposition \ref{propn: operations_special_symbols} so is $r$.
Let us define the symbol $b$ by
\[
b := a + 2\hb a\frac{1 - \phi_{0}}{(\xi + i)^{2} + |\hb t|^{2}}
[(\xi + i)F + \hb t\cdot G] 
= a(1 + 2\hb r).
\]

Since $r$ is a special zero order symbol, by 
Proposition \ref{propn: operations_special_symbols} we have that, 
for small enough $\hb$, $v := \log(1 + 2\hb r)$ is also a 
special zero order symbol. Therefore
$b = a(1 + 2\hb r) = e^{-u\phi + v}$ and $-u\phi + v$ a special 
zero order symbol. From Corollary \ref{cor: invertibility}, we conclude that 
for small $\hb$, the operator $B := \Op_{\hb}(b)$ is invertible and its inverse 
is uniformly bounded (in $\hb$) in $H^{s}_{\pm \delta, \hb}(T)$.  

\bigskip
Finally, we are left to consider the expression
\[
\hb^{2}(D^{2}a + 2V\cdot Da) 
- \frac{2}{2\pi i}\hb^{1 + \theta}(\xi + i)au\psi'(\hb^{\theta}x_{1})\phi_{0} 
=: r_{1} + r_{2}.
\]

In order to prove that the pseudodifferential operator $R$,
from Theorem \ref{thm: conjugation}, is bounded from 
$L^{2}_{-\delta}(T)$ to $L^{2}_{\delta}(T)$, we will show 
that $\lan x_{1}\ran^{2\delta}r_{i}$ are zero order 
symbols. Recall that $u\phi$ is supported on 
$|x_{1}| \leq 2\hb^{-\theta}$ and so $a = e^{-u\phi} \equiv 1$
if $|x_{1}| \geq 2\hb^{-\theta}$. Therefore, $r_{1}$ is supported on 
$|x_{1}| \leq 2\hb^{-\theta}$, and we have the bounds 
\[
|D_{x_{1}}^{\alpha}D_{x'}^{\beta}D_{\xi}^{\gamma}r_{1}| 
\cleq \hb^{2}, \ \
|D_{x_{1}}^{\alpha}D_{x'}^{\beta}D_{\xi}^{\gamma}
(\lan x_{1}\ran^{2\delta}r_{1})| 
\cleq \hb^{2 - 2\delta \theta}.
\]

The term $\psi'(\hb^{\theta}x_{1})$ gives that
$r_{2}$ is supported on 
$\hb^{-\theta} \leq |x_{1}| \leq 2\hb^{-\theta}$.
The decay estimate from Theorem \ref{thm: delta_bar_equation} 
gives that
$|D_{x_{1}}^{\alpha}D_{x'}^{\beta}D_{\xi}^{\gamma}u| 
\cleq \hb^{\theta}\lan \xi, \hb t\ran$. Moreover, $a$ is a 
zero order symbol and $(\xi+ i)\psi'(\hb^{\theta}x_{1})\phi_{0}$ 
is a zero order symbol supported on 
$\{|\xi| \leq 2\} \cap \{||\hb t| - 1| \leq 1/2\}$. We obtain that
\[
|D_{x_{1}}^{\alpha}D_{x'}^{\beta}D_{\xi}^{\gamma}r_{2}| 
\cleq \hb^{1 + 2\theta}, \ \
|D_{x_{1}}^{\alpha}D_{x'}^{\beta}D_{\xi}^{\gamma}
(\lan x_{1}\ran^{2\delta}r_{2})| 
\cleq \hb^{1 + 2\theta - 2\delta \theta}.
\]

We can ensure that $2 - 2\delta\theta, 1 + 2\theta - 2\delta\theta > 1$ 
by taking $\theta = 1/2$ and $\delta < 1$. If we let 
$\vep = 1 - \delta > 0$, then we obtain the conjugation 
identity
\[
(\Delta_{\hb} + 2\hb V_{\hb})A = B\Delta_{\hb} + \hb^{1 + \vep}R.
\]


\begin{remark} 
The restriction $\delta < 1$ does not appear in \cite{Sa1}.
This may be due that in our case we look for the operator $R$ to be bounded
from $L^{2}_{-\delta}(T)$ to $L^{2}_{\delta}(T)$, so we need a gain of
a factor $\lan x_{1}\ran^{2\delta}$. In \cite{Sa1}, it is only needed from 
$L^{2}_{\delta}(\bR^{d})$ to $L^{2}_{\delta + 1}(\bR^{d})$, so just a
factor $\lan x\ran$ is required. 
\end{remark}


\subsection{Proof of the Carleman estimate}


Let us rewrite the statements of Theorem \ref{thm: carleman_free} and 
Theorem \ref{thm: carleman_full} in semiclassical notation.


\begin{theorem}
\label{thm: carleman_free_semiclassical}
Let $\delta > 1/2$. There exists
$\hb_{0} \leq 1$ such that if $0 < |\hb| \leq \hb_{0}$ and 
$\hb^{-2} \notin \spec(-\Delta_{g_{0}})$, then for any 
$f \in L^{2}_{\delta}(T)$ there exists a unique 
$u \in H^{1}_{-\delta,\hb}(T)$ which solves 
\[
e^{2\pi x_{1}/\hb}(\hb D)^{2}e^{-2\pi x_{1}/\hb}u = \hb^{2}f.
\]
 
Moreover, this solution is in $H^{2}_{-\delta,\hb}(T)$ and satisfies 
the estimates 
\[
\|u\|_{H^{2}_{-\delta,\hb}(T)} \cleq \hb\|f\|_{L^{2}_{\delta}(T)},
\]

with the constant of the inequality independent of $\hb$. 
\end{theorem}


This theorem allows to define
$G_{\hb}: L^{2}_{\delta}(T) \ra H^{2}_{-\delta,\hb}(T)$ by $G_{\hb}f := u$, 
so that $\Delta_{\hb}G_{\hb} = \hb^{2}I$ on $L^{2}_{\delta}(T)$.


\begin{theorem}
\label{thm: carleman_full_semiclassical}
Let $1/2 < \delta < 1$ and let $V,W$ satisfy \eqref{eqn: conditions}. 
There exists $\hb_{0} \geq 1$ such that if $0 < |\hb| \leq \hb_{0}$ and
$\hb^{-2} \notin \spec(-\Delta_{g_{0}})$, then for any 
$f \in L^{2}_{\delta}(T)$ there exists a unique 
$u \in H^{2}_{-\delta,\hb}(T)$ which solves 
\[
e^{2\pi x_{1}/\hb}(\hb^{2}H_{V,W})e^{-2\pi x_{1}/\hb}u = \hb^{2}f.
\]
 
Moreover, this solution satisfies the estimates 
\[
\|u\|_{H^{2}_{-\delta,\hb}(T)} \cleq \hb\|f\|_{L^{2}_{\delta}(T)},
\]

with the constant of the inequality independent of $\hb$. 
\end{theorem}


\begin{proof}
We prove this only for $\hb > 0$, as the other case is analogous.
Let us write 
\begin{equation}
\label{eqn: carleman_equation}
(\Delta_{\hb} + 2\hb V_{\hb} + \hb^{2}\wilde{W})u 
= e^{2\pi x_{1}/\hb}(\hb^{2}H_{V,W})e^{-2\pi x_{1}/\hb}u = \hb^{2}f,
\end{equation}

where $\Delta_{\hb} = e^{2\pi x_{1}/\hb}(\hb D)^{2}e^{-2\pi x_{1}/\hb}$,  
$V_{\hb} = e^{2\pi x_{1}/\hb}[V\cdot (\hb D)]e^{-2\pi x_{1}/\hb}$,
$\wilde{W} := V^{2} + D\cdot V + W$.
To avoid repetition throughout the proof we recall from
Theorem \ref{thm: conjugation} that $A$ and 
$B$ are uniformly bounded invertible operators on 
$H^{s}_{\pm \delta,\hb}(T)$ with uniformly bounded inverses, 
$R : L^{2}_{-\delta}(T) \ra L^{2}_{\delta}(T)$
is bounded. Also, 
$\wilde{W} : L^{2}_{-\delta}(T) \ra L^{2}_{\delta}(T)$ is bounded,
because it is bounded and compactly supported.

\bigskip
We start showing the existence and the estimates for the solution. 
We look for a solution of the form $u = AG_{\hb}g \in L^{2}_{-\delta}(T)$, 
with $g \in L^{2}_{\delta}(T)$, and use Theorem \ref{thm: conjugation}
to rewrite the expression as
\[
(\Delta_{\hb} + 2\hb V_{\hb} + \hb^{2}\wilde{W})u
= (\Delta_{\hb} + 2\hb V_{\hb} + \hb^{2}\wilde{W})AG_{\hb}g
= \hb^{2}(B + \hb^{-1 + \vep}RG_{\hb} + \wilde{W}AG_{\hb})g.
\]

Let $C := \hb^{-1 + \vep}RG_{\hb} + \wilde{W}AG_{\hb}$.
We claim that $C:L^{2}_{\delta}(T) \ra L^{2}_{\delta}(T)$ is a small 
perturbation of the invertible operator $B$, so that $B + C$ is also 
invertible. Using the boundedness properties of $G_{\hb}$, from 
Theorem \ref{thm: carleman_free_semiclassical} we obtain that 
\[
\|\hb^{-1 + \vep}RG_{\hb}\|_{L^{2}_{\delta}(T) \ra L^{2}_{\delta}(T)} 
\cleq \hb^{\vep}, \ \ 
\|\wilde{W}AG_{\hb}\|_{L^{2}_{\delta}(T) \ra L^{2}_{\delta}(T)} \cleq \hb.
\]

We observe that $B + C = (I + CB^{-1})B$,
from where we conclude that $B + C$ is invertible in $L^{2}_{\delta}(T)$
as claimed, and its inverse has uniformly bounded norms. Thus if 
we define $g := (B + C)^{-1}f \in L^{2}_{\delta}(T)$, then we 
obtain that 
\[
u := AG_{\hb}g = AG_{\hb}(B + C)^{-1}f
\]

solves the equation \eqref{eqn: carleman_equation}. The estimates
for $G_{\hb}$ from Theorem \ref{thm: carleman_free_semiclassical} give
\[
\|u\|_{H^{2}_{-\delta,\hb}(T)}
= \|AG_{\hb}(B + C)^{-1}f\|_{H^{2}_{-\delta,\hb}(T)} 
\cleq \hb \|f\|_{L^{2}_{\delta}(T)},
\]

as we wanted. Now we address the uniqueness. 
Assume that $u \in H^{2}_{-\delta, \hb}(T)$ solves 
\[
(\Delta_{\hb} + 2\hb V_{\hb} + \hb^{2}\wilde{W})u = 0.
\]

Let $v := A^{-1}u \in H^{2}_{-\delta,\hb}(T)$, so that $v$ satisfies 
$(B\Delta_{\hb} + \hb^{1 + \vep}R + \hb^{2}\wilde{W}A)v = 0$,
or equivalently
\[
\Delta_{\hb}v = -\hb^{2}B^{-1}(\hb^{-1 + \vep}R+ \wilde{W}A)v.
\]

The right-hand side is in $L^{2}_{\delta}(T)$ and
\[ 
\|\hb^{-1 + \vep}B^{-1}Rv\|_{L^{2}_{\delta}(T)} 
\cleq \hb^{-1 + \vep}\|v\|_{L^{2}_{-\delta}(T)}, \ \
\|B^{-1}\wilde{W}Av\|_{L^{2}_{\delta}(T)}
\cleq \|v\|_{L^{2}_{-\delta}(T)}.
\]

The uniqueness from Theorem \ref{thm: carleman_free_semiclassical} 
implies that $v = -G_{\hb}B^{-1}(\hb^{-1 + \vep}R+ \wilde{W}A)v$.
Using the estimates for $G_{\hb}$, from 
Theorem \ref{thm: carleman_free_semiclassical}, and the bound from
above we obtain that
\[
\|v\|_{L^{2}_{-\delta}(T)} 
\cleq \hb\cdot \hb^{-1 + \vep}\|v\|_{L^{2}_{-\delta}(T)}
= \hb^{\vep}\|v\|_{L^{2}_{-\delta}(T)}.
\]

Taking $\hb$ small enough yields that $v \equiv 0$, from where
we conclude that $u \equiv 0$.
\end{proof}


\begin{remark}
The uniqueness does not follow directly from the equation
and perturbative arguments: if we rewrite the equation as
\[
\Delta_{\hb}u = -\hb^{2}(2\hb^{-1}V_{\hb} + \wilde{W})u,
\]

then the right-hand side is in $L^{2}_{\delta}(T)$, so that 
$u = - G_{\hb}(2\hb^{-1}V_{\hb} + \wilde{W})u$, but we obtain 
no contradiction as we can only say 
$\|G_{\hb}(2\hb^{-1}V_{\hb}u)\|_{L^{2}_{-\delta}(T)} \cleq \|u\|_{L^{2}_{-\delta}(T)}$.
\end{remark}


\begin{remark}
Let $f \in L^{2}_{\delta}(T)$ and let $u \in H^{2}_{-\delta,\hb}(T)$ 
be the unique solution to the equation 
$(\Delta_{\hb} + 2\hb V_{\hb} + \hb^{2}\wilde{W})u = \hb^{2}f$. 
We can rewrite this as 
\[
\Delta_{\hb}u = \hb^{2}f - (2\hb V_{\hb} + \wilde{W})u,
\]

and observe that $(2\hb V_{\hb} + \hb^{2}\wilde{W})u \in L^{2}_{\delta}(T)$. Therefore, 
$u = G_{\hb}\wilde{g}$, for some $\wilde{g} \in L^{2}_{\delta}(T)$. In the proof of existence 
of the solution, we showed that $u$ takes the form $AG_{\hb}g$ for some $g \in L^{2}_{\delta}(T)$. 
This and the last observation yield that the operator $A$ maps the subspace 
$G_{\hb}L^{2}_{\delta}(T) \sse L^{2}_{-\delta}(T)$ to itself.
\end{remark}


\section{Equivalent formulations and boundary characterization}


As mentioned in the introduction, in order to reconstruct the electromagnetic 
parameters we are interested in constructing many solutions to the equation 
$H_{V,W}u = 0$. The result from Theorem \ref{thm: carleman_full} can be used to
construct a unique solution that ``behaves like'' a harmonic function at infinity. 
Indeed, let $h \in H^{2}_{loc}(T)$ be harmonic and let us look for a solution of 
the form $u = h + e^{-2\pi \tau x_{1}}r$; such $u$ solves $H_{V,W}u = 0$ if and 
only if the correction term $r$ solves the equation
\[
e^{2\pi\tau x_{1}}H_{V,W}e^{-2\pi\tau x_{1}}r = -e^{2\pi\tau x_{1}}Xh,
\]

where $X := H_{V,W} - D^{2} = 2V\cdot D + (V^{2} + D\cdot V + W)$ is a 
first order differential operator supported in $M$. The conditions 
\eqref{eqn: conditions_reconstruction} imply that $Xh \in L^{2}_{c}(T)$, 
and so $e^{2\pi\tau x_{1}}Xh \in L^{2}_{\delta}(T)$. From 
Theorem \ref{thm: carleman_full}
we obtain a unique solution $r \in H^{2}_{-\delta}(T)$, and so there is a unique 
solution to $H_{V,W}u = 0$ which ``behaves like'' the harmonic function. As 
has been usual, we call these functions the complex geometrical optics (CGO) 
solutions. 

\bigskip
The purpose of this section is to show that the boundary values of the CGO
can be characterized as the unique solution to a certain boundary integral
equation. The passage from the uniqueness problem at the boundary to a uniqueness 
problem at infinity was first explicitly noticed by Nachman in \cite{N}, and 
has become standard since then; for instance, see \cite{Sa1} or
\cite{KSaU2}. The uniqueness of this corrected solution is crucial 
for our problem; the lack of such is what prevents the local Carleman estimate 
for the magnetic Schr$\ddot{\textrm{o}}$dinger operator in \cite{DKSaU} from 
being useful in the reconstruction procedure. 
 
\bigskip
In this section we follow closely the presentation from \cite{Sa1}, as the 
operators are translation invariant, and \cite{KSaU2}. 
However, we have to proceed slightly different, as $0 \in \spec(-\Delta_{g_{0}})$ and 
the Laplacian in $T$ does not have a bounded inverse, i.e. for $f \in L^{2}(T)$ 
(or $f \in H^{-1}(T)$) there may not be $u \in H^{2}(T)$ (or $u \in H^{1}(T)$) 
such that $D^{2}u = f$.


\subsection{Green functions, operators, and layer potentials}


\subsubsection{$\tau$-dependent Green function and operator}


The differential operator $\Delta_{\tau} = D_{x_{1}}^{2} + 2i\tau D_{x_{1}} 
- \tau^{2} + D_{x'}^{2}$ has constant coefficients, in particular it is 
translation invariant, and so it is its right inverse $G_{\tau}$ from 
Theorem \ref{thm: carleman_free}. Since 
$G_{\tau} : L^{2}_{\delta}(T) \ra L^{2}_{-\delta}(T)$ is bounded, 
there exists a tempered distribution $g_{\tau} \in \cS'(T)$ such that 
$G_{\tau}f = g_{\tau}\ast f$ for Schwartz functions $f \in \cS(T)$,
where the convolution is considered over the whole cylinder.
The purpose of this section is to understand the properties of 
$g_{\tau}$ and other related distributions.
We have that $\Delta_{\tau}g_{\tau} = \delta_{T}(0)$, 
so the Fourier expansion of this distribution is given by
\begin{equation}
\label{eqn: Green_Fourier}
\what{g_{\tau,k}}(\xi) = \frac{1}{\xi^{2} + 2i\tau \xi - \tau^{2} + |k|^{2}} 
= \frac{1}{(\xi + i\tau)^{2} + |k|^{2}},
\end{equation}

and therefore
\begin{equation}
\label{eqn: Green_Fourier_integral}
g_{\tau,k}(x_{1}) 
= \int_{\bR}\frac{e^{2\pi ix_{1}\xi}}{(\xi + i\tau)^{2} + |k|^{2}}d\xi.
\end{equation}

Let us note that this integral converges absolutely, as the denominator is quadratic 
in $\xi$ and never vanishes because $\tau^{2} \notin \spec(-\Delta_{g_{0}})$. 
These integrals can be computed explicitly as follows.


\begin{proposition}
\label{propn: Green_Fourier_explicit}
The Fourier coefficients $g_{\tau,k}(x_{1})$ 
of the distribution $g_{\tau}$ are given by
\[
g_{\tau,k}(x_{1}) 
= \pi e^{2\pi \tau x_{1}}\left\{
\begin{array}{ll}
-2\pi (|x_{1}| - \sgn(\tau) x_{1})
& \textrm{ if } \ k = 0, \\
(e^{-2\pi |k||x_{1}|} - e^{-2\pi|k|\sgn(\tau)x_{1}})/|k|
& \textrm{ if } \ 0 < |k| < |\tau|, \\
e^{-2\pi |k||x_{1}|}/|k|
& \textrm{ if } \ |k| > |\tau|.
\end{array} 
\right.
\]

The distribution $g_{\tau}$ is actually smooth away from 
$(0,0) \in T$. For any $\vep > 0$ and
$|x_{1}| > \vep$, the function $g_{\tau}(x_{1},x')$ and
all its derivatives are uniformly bounded, i.e. we have
$|D_{x_{1}}^{\alpha}D_{x'}^{\beta}g_{\tau}(x_{1},x')| 
\leq C$, for some constant $C = C(\alpha,\beta,\tau,\vep)$.
\end{proposition}


\begin{proof}
Instead of \eqref{eqn: Green_Fourier_integral}, let us consider the expression
\begin{equation}
\label{eqn: lambda_integral}
g_{\tau}(x_{1},\lambda) := \int_{\bR}\frac{e^{2\pi ix_{1}\xi}}
{(\xi + i\tau)^{2} + \lambda^{2}}d\xi,
\end{equation}

with $\lambda \geq 0$ and $\lambda \neq |\tau|$, so that the denominator
does not vanish. We start with the following two observations: first,
$g_{-\tau}(-x_{1},\lambda) = g_{\tau}(x_{1},\lambda)$, so it 
suffices to consider the case $\tau > 0$; second, for fixed $\tau$ and $x_{1}$, 
the function $g_{\tau}(x_{1},\cdot)$ is continuous, so the case 
$\lambda = 0$ follows from the case $\lambda > 0$. 
We would like to relate \eqref{eqn: lambda_integral} to the classical integral
\[
\int_{\bR}\frac{e^{2\pi ix_{1}\xi}}{\xi^{2} + \lambda^{2}}d\xi
= \frac{\pi}{\lambda}e^{-2\pi\lambda|x_{1}|},
\]

which can be obtained by direct computation and the inversion formula.
Consider the meromorphic function 
$f(z) = e^{2\pi ix_{1}z}/(z^{2} + \lambda^{2})$,
with simple poles at $z = \pm \lambda i$, and the rectangular contour 
bounded by the lines 
$\Re(z) = \pm L$, for some large $L > 0$, and $\Im(z) = 0$, 
$\Im(z) = \tau$. The pole $-\lambda i$ is outside of this domain since
we are assuming $\tau > 0$. The pole $\lambda i$ is inside this 
domain if $0 < \lambda < \tau$ and outside if $\lambda > \tau$. 
Moreover, the residue of $f$ at $z = \lambda i$ is equal to 
$e^{-2\pi\lambda x_{1}}/(2\lambda i)$.
The vertical segments of the contour have
length $\tau$, and over them the numerator of $f$ is bounded 
(uniformly in $L$), 
while the denominator is of order $L^{2}$. From the residue 
theorem, after letting $L \ra +\infty$, we deduce that 
\begin{align*}
\frac{\pi}{\lambda}e^{-2\pi\lambda|x_{1}|}
& - e^{-2\pi\tau x_{1}}g_{\tau}(x_{1},\lambda) \\
& = \int_{\bR}f(z)dz - \int_{\bR}f(z + i\tau)dz = 2\pi i
\left\{
\begin{array}{ll}
e^{-2\pi\lambda x_{1}}/(2\lambda i)
& \textrm{ if } \ 0 < \lambda < \tau, \\
0 & \textrm{ if } \ \lambda > \tau,
\end{array} 
\right.
\end{align*}

which gives that 
\[
g_{\tau}(x_{1},\lambda) = \frac{\pi e^{2\pi\tau x_{1}}}{\lambda}
\left\{
\begin{array}{ll}
e^{-2\pi\lambda|x_{1}|} - e^{-2\pi\lambda x_{1}}
& \textrm{ if } \ 0 < \lambda < \tau, \\
e^{-2\pi\lambda|x_{1}|} & \textrm{ if } \ \lambda > \tau,
\end{array} 
\right.
\]

For the case $\lambda = 0$, 
we let $\lambda \ra 0^{+}$, to conclude 
\[
g_{\tau}(x_{1},0) = \pi e^{2\pi\tau x_{1}}\lim_{\lambda \ra 0^{+}}
\frac{e^{-2\pi\lambda|x_{1}|} - e^{-2\pi\lambda x_{1}}}{\lambda}
= -2\pi^{2}e^{2\pi\tau x_{1}}(|x_{1}| - x_{1}).
\]

We have proven the formulas for the Fourier coefficients.
To show the regularity, let us observe that 
\[
D^{2}(e^{-2\pi \tau x_{1}}g_{\tau}) 
= e^{-2\pi \tau x_{1}}\Delta_{\tau}g_{\tau} 
= \delta_{T}(0).
\]

From Weyl's regularity lemma for distributions, see Chapter 10 in 
\cite{G}, it follows that $e^{-2\pi \tau x_{1}}g_{\tau}$ is 
a smooth function away from $(0,0) \in T$. Therefore, $g_{\tau}$ 
is also smooth away from $(0,0)$. Assuming that $\tau > 0$, for 
$x_{1} > 0$ we have that 
\[
g_{\tau}(x_{1},x') 
= \pi \sum_{|k| > \tau}\frac{e^{-2\pi(|k| - \tau)|x_{1}|}}{|k|}e_{k}(x'),
\] 
\vspace{-3mm}

while for $x_{1} < 0$ we have that 
\[
g_{\tau}(x_{1},x')
= 4\pi^{2}e^{-2\pi\tau |x_{1}|}x_{1}
+ \pi \sum_{|k| \neq 0}\frac{e^{-2\pi(\tau + |k|)|x_{1}|}}{|k|}e_{k}(x')
- \pi \sum_{0 < |k| < \tau}\frac{e^{-2\pi(\tau - |k|)|x_{1}|}}{|k|}e_{k}(x').
\]

The uniform boundedness of $g_{\tau}$ and its derivatives,
for $|x_{1}| > \vep$, follows from the fact that $g_{\tau}(x_{1},x')$ is a 
sum of negative exponentials on each half of $|x_{1}| > \vep$.
\end{proof}


\begin{remark}
The coefficient $g_{\tau,k}(x_{1})$ can also be computed by solving the equation 
\[ 
(D_{x_{1}}^{2} + 2i\tau D_{x_{1}} - |\tau|^{2} + |k|^{2})g_{\tau,k}(x_{1}) 
= \delta_{\bR}(0).
\]

This can be solved in each half $x_{1} > 0$ and $x_{1} < 0$ as sum of exponentials,
and then using decay conditions $\lim_{|x_{1}|\ra \pm \infty}g_{\tau,k}(x_{1}) = 0$ 
and the jump condition at $x_{1} = 0$.
\end{remark}


As in the previous proof, we consider the distribution 
$\Gamma_{\tau} := e^{-2\pi \tau x_{1}}g_{\tau} \in \cD'(T)$, 
which is no longer a tempered distribution, and satisfies 
$D^{2}\Gamma_{\tau} = \delta_{T}(0)$.
In principle, it may not make sense to talk about the Fourier 
transform of $\Gamma_{\tau}$ as it is not a tempered distribution.
However, from Proposition \ref{propn: Green_Fourier_explicit}, we could
formally say that the Fourier coefficients of $\Gamma_{\tau}$ are 
given by
\[
e^{-2\pi\tau x_{1}}g_{\tau,k}(x_{1}) 
= \left\{
\begin{array}{ll}
-2\pi^{2} (|x_{1}| - \sgn(\tau) x_{1})
& \textrm{ if } \ k = 0, \\
\pi (e^{-2\pi |k||x_{1}|} - e^{-2\pi|k|\sgn(\tau)x_{1}})/|k|
& \textrm{ if } \ 0 < |k| < |\tau|, \\
\pi e^{-2\pi |k||x_{1}|}/|k|
& \textrm{ if } \ |k| > |\tau|.
\end{array} 
\right.
\]

Based on the formal Fourier coefficients above, we consider the
harmonic function
\[
H_{\tau}(x_{1},x') := 2\pi^{2}\sgn(\tau)x_{1}
- \pi \sum_{0 < |k| < |\tau|}\frac{e^{-2\pi|k|\sgn(\tau)x_{1}}}{|k|}e_{k}(x').
\]

We have that $H_{\tau} \in \cD'(T)$ because it is a smooth function, and so
$\Gamma_{0} := \Gamma_{\tau} - H_{\tau} \in \cD'(T)$ as well. Let us define the
distributions $\Gamma_{0}^{0} := -2\pi^{2}|x_{1}|$ and 
$\Gamma_{0}^{*} := \Gamma_{0} - \Gamma_{0}^{0}$. Formally, 
the Fourier coefficients of $\Gamma_{0}$ and $\Gamma_{0}^{*}$ are given by 
\[
\Gamma_{0,k}(x_{1}) 
= \left\{
\begin{array}{ll}
-2\pi^{2}|x_{1}|
& \textrm{ if } \ k = 0, \\
\pi e^{-2\pi |k||x_{1}|}/|k|
& \textrm{ if } \ k \neq 0,
\end{array} 
\right.
\]
\begin{equation}
\label{eqn: Gamma_{0}^{*}}
\Gamma_{0,k}^{*}(x_{1}) 
= \left\{
\begin{array}{ll}
0
& \textrm{ if } \ k = 0, \\
\pi e^{-2\pi |k||x_{1}|}/|k|
& \textrm{ if } \ k \neq 0.
\end{array} 
\right.
\end{equation}

As in the proof of Proposition \ref{propn: Green_Fourier_explicit},
we have that if $k \neq 0$, then 
\begin{equation}
\label{eqn: Fourier_Gamma_{0}^{*}}
\what{\Gamma_{0,k}^{*}}(\xi) = \frac{1}{\xi^{2} + |k|^{2}}.
\end{equation}


\begin{proposition}
\label{propn: tempered}
The distributions $\Gamma_{0}^{0}$ and $\Gamma_{0}^{*}$ are tempered
distributions, thus so is $\Gamma_{0}$.
\end{proposition}


\begin{proof}
It is clear that $\Gamma_{0}^{0}$ is a tempered distribution. From
\eqref{eqn: Fourier_Gamma_{0}^{*}} we actually obtain that 
$\Gamma_{0}^{*} \in H^{s}(T)$ for all $s < -(d - 3)/2$, and the
conclusion follows.
\end{proof}


\begin{proposition}
\label{propn: Gamma_{tau}}
Let $\Gamma_{\tau}(x,y) := \Gamma_{\tau}(x - y)$. Then, 
$D^{2}\Gamma_{\tau}(x, \cdot) = \delta_{T}(x)$, 
$\Gamma_{\tau}(\cdot,\cdot)$ is smooth in $T \times T$ away 
from the diagonal.
\end{proposition}


\begin{proof}
As mentioned in the proof of Proposition \ref{propn: Green_Fourier_explicit},
the fact that $D^{2}\Gamma_{\tau} = \delta_{T}(0)$ and Weyl's regularity
lemma imply that $\Gamma_{\tau}$ is smooth away from 
$(0,0) \in T$. This gives that $D^{2}\Gamma_{\tau}(x,\cdot) = \delta_{T}(x)$ 
and the smoothness of $\Gamma_{\tau}(\cdot,\cdot)$ away from the diagonal.
\end{proof}


Despite the reasons not being apparent at this moment, we consider the 
following definition. We will see later how this operator appears naturally
when we try to reformulate the differential equation the solution 
$H_{V,W}u = 0$ as an integral equation. For the moment, in 
Proposition \ref{propn: K_{tau}} below, we show how this operator
relates to the distribution $\Gamma_{\tau}$.  

									
\begin{definition}
Let $|\tau|\geq \tau_{0}$, $\tau^{2} \notin \spec(-\Delta_{g_{0}})$, so that 
$\Delta_{\tau}$ has a right inverse $G_{\tau}$ from 
Theorem \ref{thm: carleman_free}.
For functions in $L^{2}_{c}(T)$, we define 
the operator $K_{\tau}:= e^{-2\pi \tau x_{1}}G_{\tau}e^{2\pi \tau x_{1}}$.
\end{definition}


\begin{proposition}
\label{propn: K_{tau}}
The operator $K_{\tau}$ maps $L^{2}_{c}(T)$ into $H^{2}_{loc}(T)$,
is translation invariant, commutes
with differentiation, and satisfies $D^{2}K_{\tau} = I$ on 
$L^{2}_{c}(T)$ and $K_{\tau}D^{2} = I$ on $H^{2}_{c}(T)$.
Moreover, its distributional kernel is $\Gamma_{\tau}(\cdot,\cdot)$,
i.e. $K_{\tau}f(x) = \Gamma_{\tau}\ast f(x) 
= \lan \Gamma_{\tau}(x,\cdot),f\ran$ for $f \in C^{\infty}_{c}(T)$.
\end{proposition}


\begin{proof}
The first claim follows because $G_{\tau}$ maps $L^{2}_{\delta}(T)$
into $H^{2}_{-\delta}(T)$. 
The translation invariance of $K_{\tau}$ follows from the conjugation 
structure and the translation invariance of $G_{\tau}$. Indeed, if we 
denote the translation operators by $t_{y}f(x) := f(x + y)$ and note 
that $t_{y}(e^{\lambda x}f) = e^{\lambda (x + y)}t_{y}f$, then
\begin{align*}
t_{y}K_{\tau}
& = e^{-2\pi\tau(x_{1} + y_{1})}t_{y}G_{\tau}e^{2\pi \tau x_{1}} \\
& = e^{-2\pi\tau(x_{1} + y_{1})}G_{\tau}t_{y}e^{2\pi \tau x_{1}}
= e^{-2\pi\tau(x_{1} + y_{1})}G_{\tau}e^{2\pi\tau(x_{1} + y_{1})}t_{y} 
= e^{-2\pi \tau x_{1}}G_{\tau}e^{2\pi\tau x_{1}}t_{y}
= K_{\tau}t_{y},
\end{align*}

as we wanted to prove. The commutativity with differentiation
follows from the translation invariance. In addition, if 
$f \in L^{2}_{c}(T)$, then 
$e^{2\pi\tau x_{1}}f \in L^{2}_{\delta}(T)$, and so 
\[
D^{2}K_{\tau}f  
= e^{-2\pi\tau x_{1}}(e^{2\pi\tau x_{1}}D^{2}e^{-2\pi\tau x_{1}})G_{\tau}(e^{2\pi\tau x_{1}}f) 
= e^{-2\pi\tau x_{1}}(\Delta_{\tau}G_{\tau})(e^{2\pi\tau x_{1}}f) 
= f.
\] 

If $f \in H^{2}_{c}(T)$, then $D^{2}f \in L^{2}_{c}(T)$, 
and the commutativity with differentiation yields that 
$K_{\tau}D^{2}f = D^{2}K_{\tau}f = f$. Finally, for 
$f \in C^{\infty}_{c}(T)$ we have that 
$e^{2\pi\tau x_{1}}f \in C^{\infty}_{c}(T) \sse \cS(T)$, and so
\begin{align}
\label{eqn: kernel_K_{tau}}
K_{\tau}f(x) & = e^{-2\pi\tau x_{1}}G_{\tau}(e^{2\pi\tau x_{1}}f)(x) \nonumber \\
& = e^{-2\pi\tau x_{1}}\int_{T}g_{\tau}(x_{1} - y_{1},x' - y')
e^{2\pi\tau y_{1}}f(y_{1}, y')dy_{1}dy' \nonumber \\
& = \int_{T}e^{-2\pi\tau (x_{1} - y_{1})}g_{\tau}(x_{1} - y_{1}, x' - y')
f(y_{1}, y')dy_{1}dy'
= \int_{T}\Gamma_{\tau}(x,y)f(y)dy,
\end{align}

as we wanted to prove.
\end{proof}


The purpose of what follows is to show that the mapping properties 
of $K_{\tau}$ from $L^{2}_{c}(T)$ into $H^{2}_{loc}(T)$ can be 
extended to $H^{-1}_{c}(T)$ into $H^{1}_{loc}(T)$. To show this, 
let us consider the operators
\[
K_{0}^{0}f := \Gamma_{0}^{0} \ast f, \ \ 
K_{0}^{*}f := \Gamma_{0}^{*} \ast f, \ \
R_{\tau}f := H_{\tau} \ast f,
\] 

with the above definitions for $\Gamma_{0}^{0}$, $\Gamma_{0}^{*}$, 
and $H_{\tau}$. Given that 
$\Gamma_{\tau} = \Gamma_{0}^{0} + \Gamma_{0}^{*} + H_{\tau}$
we have that \\ 
$K_{\tau} = K_{0}^{0} + K_{0}^{*} + R_{\tau}$, 
and so it suffices to show that each of these maps $H^{-1}_{c}(T)$ 
into $H^{1}_{loc}(T)$. 


\begin{proposition}
\label{propn: K_{0}^{0}}
The operator $K_{0}^{0}$ maps $H^{-1}_{c}(T)$ into $H^{1}_{loc}(T)$. 
\end{proposition}


\begin{proof}
Let $\varphi \in H^{-1}_{c}(T)$, so that 
$\varphi_{0} \in H^{-1}_{c}(\bR) \sse H^{-1}_{c}(T)$.
Since $|x_{1}|$ does not depend on $x'$, we have that 
$|x_{1}|\ast \varphi = |x_{1}|\ast \varphi_{0}$, and it 
remains to show that $|x_{1}| \ast \varphi_{0} \in H^{1}_{loc}(\bR)$. 
Let $\supp(\varphi_{0}) \sse [-L,L]$ and let 
$\phi \in C^{\infty}_{c}(\bR)$ be such that $\phi \equiv 1$ on 
$[-L,L]$ and $\phi \equiv 0$ outside of $[-2L,2L]$. Let us show 
that for fixed $x_{1} \in \bR$, $\phi |x_{1} - \cdot| \in H^{1}(\bR)$. 
Indeed, by Leibniz' rule we obtain
\[
\|\phi|x_{1} - \cdot|\|_{H^{1}(\bR)}
\cleq \|\phi\|_{C^{1}(\bR)}\||x_{1} - \cdot|\|_{H^{1}(-2L,2L)}
\cleq L + |x_{1}|,
\]

where we allow the constants of the inequality to depend on $L$
and $\phi$. Let $\Phi(x_{1}) := |x_{1}|\ast \varphi_{0}$. Then, 
\begin{align*}
|\Phi(x_{1})| & = |\lan \varphi_{0}, |x_{1} - \cdot|\ran| \\
& = |\lan \phi \varphi_{0}, |x_{1} - \cdot| \ran| \\
& = |\lan \varphi_{0}, \phi|x_{1} - \cdot| \ran| 
\leq \|\varphi_{0}\|_{H^{-1}(\bR)}\|\phi|x_{1} - \cdot|\|_{H^{1}(\bR)}
\cleq \|\varphi_{0}\|_{H^{-1}(\bR)}(L + |x_{1}|).
\end{align*}

This shows that $\Phi \in L^{\infty}_{loc}(\bR) \sse L^{2}_{loc}(\bR)$,
which implies that $\Phi' \in H^{-1}_{loc}(\bR)$. Moreover, because
$|x_{1}|'' = 2\delta_{\bR}(0)$, then we have that $\Phi'' = 2\varphi_{0} \in H^{-1}(\bR)$.
Let $\eta \in C^{\infty}_{c}(\bR)$ and $\rho = \eta \Phi$, so that 
$\rho \in L^{2}_{c}(\bR)$. We have to show that $\rho \in H^{1}(\bR)$,
which is equivalent to showing that $\lan \xi\ran \what{\rho}(\xi) \in L^{2}(\bR)$.
Since $\rho \in L^{2}(\bR)$, we have 
$\lan \xi\ran \what{\rho}(\xi) \in L^{2}(|\xi| \leq 1)$. Moreover, because 
$\Phi \in L^{2}_{loc}(\bR)$ and $\Phi', \Phi'' \in H^{-1}_{loc}(\bR)$,
then $\rho'' = \eta''\Phi + 2\eta'\Phi' + \eta\Phi'' \in H^{-1}(\bR)$,
 Therefore, $\lan \xi\ran^{-1}\xi^{2}\what{\rho}(\xi) \in L^{2}(\bR)$, from
 where we conclude that $\lan \xi\ran \what{\rho}(\xi) \in L^{2}(|\xi| \geq 1)$.
This proves the result.
\end{proof}


\begin{proposition}
\label{propn: K_{0}^{*}}
The operator $K_{0}^{*}: H^{s}(T) \ra H^{s + 2}(T)$ is bounded for 
any $s \in \bR$. 
\end{proposition}


\begin{proof}
For $\varphi \in H^{s}(T)$, let us consider the Fourier series 
$\varphi(x_{1},x') = \sum_{k \in \bZ^{d}}\varphi_{k}(x_{1})e_{k}(x')$.
From \eqref{eqn: Gamma_{0}^{*}} we have that 
$\what{(K_{0}^{*}\varphi)}_{0}(\xi)
= \what{(\Gamma_{0}^{*}\ast \varphi)_{0}}(\xi) = 0$, and
for $k \neq 0$ we have
\[
|\what{(K_{0}^{*}\varphi)}_{k}(\xi)|
= |\what{(\Gamma_{0}^{*}\ast \varphi)_{k}}(\xi)|
= |\what{\Gamma_{0,k}^{*}}(\xi)||\what{\varphi_{k}}(\xi)|
= \frac{1}{\xi^{2} + |k|^{2}}|\what{\varphi_{k}}(\xi)|
\cleq \lan \xi, k\ran^{-2}|\what{\varphi_{k}}(\xi)|,
\]

where we used in the last step that $|k| \geq 1$ for all $k \neq 0$.
Therefore we conclude that 
\[
\|K_{0}^{*}\varphi\|_{H^{s + 2}(T)}^{2}
= \sum_{k \neq 0}\int_{\bR}\lan \xi, k\ran^{2s + 4}
|\what{(K_{0}^{*}\varphi)_{k}}(\xi)|^{2}d\xi
\cleq \sum_{k \neq 0}\int_{\bR}\lan \xi, k\ran^{2s}
|\what{\varphi_{k}}(\xi)|^{2}d\xi
= \|\varphi\|_{H^{s}(T)}^{2}.
\]
\vspace{-8mm}

\end{proof}


\begin{proposition}
\label{propn: K_{0}}
The operator $K_{0} := K_{0}^{0} + K_{0}^{*}$ maps 
$L^{2}_{c}(T)$ into $H^{2}_{loc}(T)$ and
$H^{-1}_{c}(T)$ into $H^{1}_{loc}(T)$, satisfies 
$D^{2}K_{0} = I$ on $L^{2}_{c}(T)$ and 
$K_{0}D^{2} = I$ on $H^{2}_{c}(T)$, and
is symmetric, i.e. $\lan K_{0}f,g\ran = \lan K_{0}g,f\ran$ 
for any $f,g \in H^{-1}_{c}(T)$. 
\end{proposition}


\begin{proof}
The operator $R_{\tau}$ maps $L^{2}_{c}(T)$ into 
$C^{\infty}(T)$, because the kernel $H_{\tau}$ is a smooth
function. Therefore, $K_{0} = K_{\tau} - R_{\tau}$ also maps 
$L^{2}_{c}(T)$ into $H^{2}_{loc}(T)$. Moreover, we have 
that $K_{0}^{0}$ and $K_{0}^{*}$ map $H^{-1}_{c}(T)$
into $H^{1}_{loc}(T)$ from Proposition \ref{propn: K_{0}^{0}} and 
Proposition \ref{propn: K_{0}^{*}}, and therefore so does $K_{0}$. 
The identities with the Laplacian 
follow from those of Proposition \ref{propn: K_{tau}}, since 
the kernel of $R_{\tau}$ is a harmonic function. Finally, the 
symmetry of the operator follows from the symmetry of the kernel 
$\Gamma_{0} := \Gamma_{0}^{0} + \Gamma_{0}^{*}$.
\end{proof}


\begin{proposition}
\label{propn: K_{tau}_mapping}
The operator $K_{\tau}$ maps 
$H^{-1}_{c}(T)$ into $H^{1}_{loc}(T)$. 
\end{proposition}


\begin{proof}
Recall that $K_{\tau} = K_{0} + R_{\tau}$. 
The result follows from Proposition \ref{propn: K_{0}} and the fact that
$R_{\tau}$ maps $H^{-1}_{c}(T)$ into $C^{\infty}(T)$. 
\end{proof}


\subsubsection{$\tau$-dependent single layer potential}


Recall that we have the boundedness of $\tr: H^{s}(T) \ra H^{s - 1/2}(\pa M)$, 
for $s > 1/2$. In particular we have $\tr: H^{1}(T) \ra H^{1/2}(\pa M)$ and its 
adjoint $\tr^{*}: H^{-1/2}(\pa M) \ra H^{-1}_{c}(T)$. The results from 
Proposition \ref{propn: K_{0}} and 
Proposition \ref{propn: K_{tau}_mapping} allow for the following definition.


\begin{definition}
\label{defn: S_{0}_S_{tau}}
Define the single layer operator 
$S_{0} := K_{0}\tr^{*}: H^{-1/2}(\pa M) \ra H^{1}_{loc}(T)$.
Similarly, for $|\tau| \geq \tau_{0}$, $\tau^{2} \notin \spec(-\Delta_{g_{0}})$, 
we define the $\tau$-dependent single layer operator
\[
S_{\tau} := K_{\tau}\tr^{*} : H^{-1/2}(\pa M) \ra H^{1}_{loc}(T).
\]
\end{definition}


\begin{proposition}
\label{propn: S_{0}_S_{tau}}
Let $S$ denote either of the single layer operators $S_{0}$ or $S_{\tau}$ 
from Definition \ref{defn: S_{0}_S_{tau}}, and let $\Gamma$ and $K$ denote 
either of $\Gamma_{0}$ and $K_{0}$ or $\Gamma_{\tau}$ and $K_{\tau}$, 
as it corresponds. For $\varphi \in H^{-1/2}(\pa M)$, the single layer potential 
$S\varphi \in H^{1}_{loc}(T)$ satisfies the following properties:
\begin{enumerate}[label=\textrm{\alph*).}]
\item for $x \notin \pa M$ we have the integral representation
$S\varphi(x) = \lan \varphi, \tr(\Gamma(x,\cdot))\ran$,
\item $S\varphi$ is harmonic in $M_{\pm}$,
\item $S\varphi$ has no jump at the boundary,
$\tr^{+}(S\varphi) = \tr^{-}(S\varphi)$, 
and therefore has a well-defined trace,
\item the normal derivatives of $S\varphi$ satisfy that
$\pa_{\nu}^{-}S\varphi - \pa_{\nu}^{+}S\varphi = 4\pi^{2}\varphi$ 
on $\pa M$,
\item if $\varphi \in H^{1/2}(\pa M)$, then 
$S\varphi|_{M} \in H^{2}(M)$,
$S\varphi|_{M_{+}}$ has an extension in 
$H^{2}_{loc}(T)$, and
$\tr \circ S$ maps $H^{1/2}(\pa M)$ into $H^{3/2}(\pa M)$.
\end{enumerate}
\end{proposition}


\begin{proof}
Let $\varphi \in H^{-1/2}(\pa M)$. For a fixed $x \notin \pa M$ 
there exists an open neighborhood 
$N \sse T$, such that $\pa M \sse N$ and $x \notin N$. From 
Proposition \ref{propn: Gamma_{tau}} and the fact that 
$H_{\tau}$ is smooth we have that 
$\Gamma(x,\cdot)$ is smooth in $N$ and so 
$\tr(\Gamma(x,\cdot)) \in H^{1/2}(\pa M)$. Therefore,
\[
\lan \varphi, \tr(\Gamma(x,\cdot))\ran 
= \lan \tr^{*}\varphi, \Gamma(x,\cdot)\ran
= \Gamma\ast \tr^{*}\varphi(x)
= K\tr^{*}\varphi(x) = S\varphi(x).
\]

The harmonicity of $S\varphi$ in $M_{\pm}$ follows from
the previous result as $\Gamma(\cdot,y)$ is harmonic
in $M_{\pm}$ for any $y \in \pa M$. The existence 
of a well-defined trace follows from the fact that 
$S\varphi \in H^{1}_{loc}(T)$. Given that 
$S\varphi \in H^{1}_{loc}(T)$ is harmonic in $M_{\pm}$, 
there are well-defined normal derivatives as elements of 
$H^{-1/2}(\pa M)$. Moreover, 
$K_{\tau} - K_{0} = R_{\tau}$
maps $H^{-1}_{c}(T)$ into $C^{\infty}(T)$, so it 
suffices to show the jump condition for $S = S_{0}$.
Let $g \in H^{3/2}(\pa M)$ and let $v \in H^{2}_{c}(T)$
be some function extending $g$.
The definition of normal derivatives 
\eqref{defn: normal_derivatives}, integration by parts,
and Proposition \ref{propn: K_{0}} give that
\begin{align*}
\lan(\pa_{\nu}^{-} - \pa_{\nu}^{+})S_{0}\varphi,g\ran
& = 4\pi^{2}\int_{T}-DS_{0}\varphi\cdot Dv \\
& = 4\pi^{2}\int_{T}S_{0}\varphi D^{2}v \\
& = 4\pi^{2}\lan K_{0}\tr^{*}\varphi, D^{2}v\ran
= 4\pi^{2}\lan \tr^{*}\varphi, K_{0}D^{2}v\ran
= 4\pi^{2}\lan \tr^{*}\varphi, v\ran
= 4\pi^{2}\lan \varphi, g\ran.
\end{align*}

The density of $H^{3/2}(\pa M)$ in $H^{1/2}(\pa M)$
implies the jump condition of the normal derivatives. 
Finally, as in \cite{KSaU2}, we invoke the transmission 
property from \cite{Mc} to prove the higher regularity 
properties of the single layer potential. Namely, the harmonicity
of $S\varphi|_{M_{\pm}}$ and the jump conditions at the boundary give
that if $\varphi \in H^{1/2}(\pa M)$, then 
$S\varphi|_{M_{\pm}}$ is in $H^{2}(M_{\pm}\cap N)$, for
some neighborhood $N \sse T$ of $\pa M$. The interior regularity
of harmonic functions gives that $S\varphi|_{M} \in H^{2}(M)$
and $S\varphi|_{M_{+}} \in H^{2}_{loc}(M_{+})$, and the 
boundary regularity allows to construct the extension of
$S\varphi|_{M_{+}}$ to $H^{2}_{loc}(T)$.
\end{proof}


\begin{remark}
We will not need this, but the map 
$\tr \circ S_{\tau} : H^{s}(\pa M) \ra H^{s + 1}(\pa M)$ is 
bounded for $s \geq -1/2$.
\end{remark}


\subsection{Equivalent formulations and boundary characterization}


For the rest of the section we assume that $0$ is not an eigenvalue 
of the magnetic Schr$\ddot{\textrm{o}}$dinger operator $H_{V,W}$ 
on $M$. Let $|\tau| \geq \tau_{0}$, 
$\tau^{2} \notin \spec(-\Delta_{g_{0}})$ as in Theorem \ref{thm: carleman_full},
and let $h \in H^{2}_{loc}(T)$ be a harmonic function. 


\begin{theorem}
\label{thm: equivalent}
All the following problems have a unique solution:
\begin{align*}
\mbox{(DE):} \ & \ u = h + e^{-2\pi \tau x_{1}}r, 
\ \mbox{with} \ r \in H^{2}_{-\delta}(T),
\ \mbox{solves the differential equation} 
\ H_{V,W}u = 0 \  \mbox{in} \ T, \\
\mbox{(IE):} \ & \ u \in H^{2}_{loc}(T) \ \mbox{solves the integral equation} 
\ u + K_{\tau}Xu = h \ \mbox{in} \ T, \\ 
\mbox{(EP):} \ & \ \wilde{u} \in H^{2}_{loc}(M_{+}) \ \mbox{is harmonic, has
an extension in} \ H^{2}_{loc}(T) \ \mbox{of the form} \ h + e^{-2\pi\tau x_{1}}r
\ \mbox{with} \\
& \ r \in H^{2}_{-\delta}(T), \ \mbox{and} \ 
\pa_{\nu}^{+}\wilde{u} = 4\pi^{2}\Lambda_{V,W}(\tr^{+}(\wilde{u})), \\
\mbox{(BE):} \ & \ f \in H^{3/2}(\pa M) \ \mbox{solves the boundary
equation} \ (I + \tr \circ S_{\tau}(\Lambda_{V,W} - \Lambda_{0,0}))f = \tr(h).
\end{align*}

These problems are equivalent in the following sense:
\begin{align*}
\mbox{(DE)} \Leftrightarrow \mbox{(IE):} \ 
& \  u \ \mbox{solves (DE) if and only if $u$ solves (IE),} \\
\mbox{(DE)} \Leftrightarrow \mbox{(EP):} \ 
& \  \mbox{if} \ u \ \mbox{solves (DE), then} \ u|_{M_{+}} \ \mbox{solves (EP), and if} 
\ \wilde{u} \ \mbox{solves (EP), then there exists an} \\
& \ \mbox{extension} \ u \ \mbox{to} \ T
\ \mbox{that solves (DE),} \\
\mbox{(DE)} \Rightarrow \mbox{(BE):} \
& \mbox{if} \ u \ \mbox{solves (DE), then} \ \tr(u) \ \mbox{solves (BE),} \\
\mbox{(BE)} \Rightarrow \mbox{(EP):} \
& \mbox{if} \ f \ \mbox{solves (BE), then there is an extension} \
\wilde{u} \ \mbox{to} \ M_{+} \ \mbox{that solves (EP).}
\end{align*}
\end{theorem}


\begin{proof}
From Theorem \ref{thm: carleman_full} we know that (DE) has a unique solution.
It remains to show the equivalence between the existence of solutions, as 
the equivalence of the uniqueness follows from this.

\bigskip 
We start proving that the problems (DE) and (IE) are equivalent.
Assume that a solution to the equation $H_{V,W}u = 0$, has the form 
$u = h + e^{-2\pi \tau x_{1}}r$, with $r \in H^{2}_{-\delta}(T)$. Then 
$u \in H^{2}_{loc}(T)$ and we see that $r$ solves
\[
\Delta_{\tau}r = e^{2\pi\tau x_{1}}D^{2}(u - h)
= e^{2\pi\tau x_{1}}D^{2}u
= - e^{2\pi \tau x_{1}}Xu,
\] 

and $e^{2\pi\tau x_{1}}Xu \in L^{2}_{c}(M)$.
Since $r \in H^{2}_{-\delta}(T)$, the uniqueness from 
Theorem \ref{thm: carleman_full} implies 
that $r = -G_{\tau}(e^{2\pi \tau x_{1}}Xu)$,
and so $h = u - e^{-2\pi \tau x_{1}}r = u + K_{\tau}Xu$.
Conversely, if $u \in H^{2}_{loc}(T)$ satisfies $u + K_{\tau}Xu = h$, 
then $u = h + e^{-2\pi \tau x_{1}}r$ with 
$r := - G_{\tau}(e^{2\pi \tau x_{1}}Xu) \in H^{2}_{-\delta}(T)$.
This gives that 
\[
e^{2\pi\tau x_{1}}D^{2}u 
= e^{2\pi \tau x_{1}}D^{2}(u - h)
= \Delta_{\tau}r 
= -e^{2\pi \tau x_{1}}Xu,
\]

and thus $H_{V,W}u = 0$. 						

\bigskip
Now we show that the problems (DE) and (EP) are equivalent.
Assume that a solution to the equation $H_{V,W}u = 0$, has the form 
$u = h + e^{-2\pi \tau x_{1}}r$, with $r \in H^{2}_{-\delta}(T)$, so that 
$u \in H^{2}_{loc}(T)$. 
If we let $\wilde{u} := u|_{M_{+}}$, then $\wilde{u} \in H^{2}_{loc}(M_{+})$, and, 
given that $V,W$ are supported in $M$, we have that $\wilde{u}$ 
is harmonic in $M_{+}$. If $g \in H^{1/2}(\pa M)$ and $v \in H^{1}_{c}(T)$
is some function extending $g$, then from the definitions 
\eqref{defn: normal_derivatives} and \eqref{defn: DN_magnetic} 
we have
\[
\lan \pa_{\nu}^{+}\wilde{u}, g\ran
= -4\pi^{2}\int_{M_{+}}- D\wilde{u}\cdot Dv,
\]
\[
\lan \Lambda_{V,W}(\tr^{-}(u)),g\ran 
= \int_{M}-Du\cdot Dv + V\cdot (v Du  
- uDv) + (V^{2} + W)uv.
\]

Since $u$ is a solution to $H_{V,W}u = 0$ in $T$, and 
$V,W$ are supported in $M$, we obtain that
\[
-\int_{M_{+}}-D\wilde{u} \cdot Dv
= -\int_{M_{+}}- Du\cdot Dv
= \int_{M}-Du\cdot Dv + V\cdot (v Du  
- uDv) + (V^{2} + W)uv
\]

which gives that 
$\pa_{\nu}^{+}\wilde{u} = 4\pi^{2}\Lambda_{V,W}(\tr^{-}(u))$.
Thus we conclude that
\[
\pa_{\nu}^{+}\wilde{u} 
= 4\pi^{2}\Lambda_{V,W}(\tr^{-}(u))
= 4\pi^{2}\Lambda_{V,W}(\tr^{+}(u))
= 4\pi^{2}\Lambda_{V,W}(\tr^{+}(\wilde{u})).
\]

Conversely, suppose that $\wilde{u} \in H^{2}_{loc}(M_{+})$ is harmonic 
in $M_{+}$, satisfies 
$\pa_{\nu}^{+}\wilde{u} = 4\pi^{2}\Lambda_{V,W}(\tr^{+}(\wilde{u}))$, 
and is such that $\wilde{u}$ has an extension in $H^{2}_{loc}(T)$
of the form $h + e^{-2\pi \tau x_{1}}r$ on $M_{+}$ with 
$r \in H^{2}_{-\delta}(T)$. We want to extend $\wilde{u}$ to the interior 
of $M$ in order to solve $H_{V,W}u = 0$ in $T$. Let 
$\ov{u} = D_{V,W}(\tr^{+}(\wilde{u})) \in H^{2}(M)$, i.e. 
the solution of the problem 
\[
\biggl\{
\begin{array}{rll}
H_{V,W}\ov{u} \hspace{-2mm} & = 0 & \ \textrm{in} \ \ M_{-}, \\
\ov{u} \hspace{-2mm} & = \tr^{+}(\wilde{u}) & \ \textrm{on} \ \ \pa M.
\end{array}
\]

and define $u|_{M_{+}} = \wilde{u} \in H^{2}_{loc}(M_{+})$ and 
$u|_{M} = \ov{u} \in H^{2}(M)$. Then we have
\[
\tr^{+}(u) = \tr^{+}(\wilde{u}) = \tr^{-}(\ov{u}) = \tr^{-}(u),
\]
\[
\pa_{\nu}^{+}u = \pa_{\nu}^{+}\wilde{u} 
= 4\pi^{2}\Lambda_{V,W}(\tr^{+}(\wilde{u}))
= 4\pi^{2}\Lambda_{V,W}(\tr^{-}(\ov{u}))
= \pa_{\nu}^{-}\ov{u} 
= \pa_{\nu}^{-}u,
\]

where we used the result from Proposition \ref{propn: DN_map}.
This implies that $u$ is in $H^{2}_{loc}(T)$ and solves $H_{V,W}u = 0$.
Moreover, $u = h + e^{-2\pi\tau x_{1}}\ov{r}$ in $T$, with 
$\ov{r} \in H^{2}_{-\delta}(T)$, where $\ov{r}|_{M_{+}} = r|_{M_{+}}$ 
and $\ov{r}|_{M} = e^{2\pi \tau x_{1}}(u - h)|_{M}$.

\bigskip
Now we prove that (DE) implies (BE).
Assume that a solution to the equation $H_{V,W}u = 0$ has the form 
$u = h + e^{-2\pi \tau x_{1}}r$, with $r \in H^{2}_{-\delta}(T)$, so that 
$u \in H^{2}_{loc}(T)$ and $\tr(u) \in H^{3/2}(\pa M)$. The equivalence 
between (DE) and (IE) yields that $u + K_{\tau}Xu = h$ and so 
$\tr(u) + \tr K_{\tau}Xu = \tr(h)$. Taking (exterior) traces in 
Proposition \ref{propn: K_{tau}_S_{tau}} below, gives that 
$\tr K_{\tau}Xu = \tr \circ S_{\tau}(\Lambda_{V,W} - \Lambda_{0,0})\tr(u)$,
which implies that $\tr(u)$ solves (BE).

\bigskip
Finally, we show that (BE) implies (EP).
Suppose $f \in H^{3/2}(\pa M)$ solves the boundary equation
$(I + \tr\circ S_{\tau}(\Lambda_{V,W} - \Lambda_{0,0}))f = \tr(h)$. 
Motivated by Proposition \ref{propn: K_{tau}_S_{tau}} below, we define 
$\wilde{u} := h - S_{\tau}(\Lambda_{V,W} - \Lambda_{0,0})f$.
The boundary equation gives that $\tr(\wilde{u}) = f$. From Proposition \ref{propn: S_{0}_S_{tau}} 
we know that the restrictions $\wilde{u}|_{M_{\pm}}$ are in $H^{2}_{loc}(M_{\pm})$ 
and are harmonic in $M_{\pm}$, respectively. Since $\wilde{u}|_{M}$ 
is harmonic in $M$ and $\tr(\wilde{u}) = f$, then we have 
$\pa_{\nu}^{-}\wilde{u} = 4\pi^{2}\Lambda_{0,0}f$. Given that 
$h \in H^{2}_{loc}(T)$, the definition of $\wilde{u}$, and the 
jump condition of the normal derivatives from Proposition \ref{propn: S_{0}_S_{tau}} 
we obtain that
\[
\pa_{\nu}^{+}\wilde{u} = \pa_{\nu}^{-}\wilde{u} 
+ 4\pi^{2}(\Lambda_{V,W} - \Lambda_{0,0})f
= 4\pi^{2}\Lambda_{V,W}f.
\]

From Proposition \ref{propn: S_{0}_S_{tau}} we know that $\wilde{u}|_{M_{+}}$ 
has an extension in $H^{2}_{loc}(T)$.
All that remains is to show that $\wilde{u}|_{M_{+}}$ has an extension 
in $H^{2}_{loc}(T)$ of the form $h + e^{-2\pi \tau x_{1}}r$, with 
$r \in H^{2}_{-\delta}(T)$. Given that $\wilde{u} := h - S_{\tau}\phi$
with $\phi \in H^{1/2}(\pa M)$, all we have to show is that 
$e^{2\pi\tau x_{1}}S_{\tau}\phi|_{M_{+}}$ has an extension in 
$H^{2}_{-\delta}(T)$. From Proposition \ref{propn: S_{0}_S_{tau}} we know that 
it has an extension in $H^{2}_{loc}(T)$, so it suffices to show that 
$e^{2\pi\tau x_{1}}S_{\tau}\phi$ is in $H^{2}_{-\delta}(|x_{1}| \geq L)$ 
for some large $L$. From the integral representation in
Proposition \ref{propn: S_{0}_S_{tau}} we see that 
\[
e^{2\pi\tau x_{1}}S_{\tau}\phi(x)
= e^{2\pi \tau x_{1}}\lan \phi, \tr(\Gamma_{\tau}(x,\cdot))\ran
= \lan e^{2\pi\tau y_{1}}\phi, \tr(g_{\tau}(x,\cdot))\ran,
\]
where we have used that
$\Gamma_{\tau}(x,y) = e^{-2\pi\tau (x_{1} - y_{1})}g_{\tau}(x,y)$.
From Proposition \ref{propn: Green_Fourier_explicit} we have that the restrictions
$\{\tr(D_{x}^{\alpha}g_{\tau}(x,\cdot))\}$ are 							
uniformly bounded for $|x_{1}| \geq L$ with $L$ large. This implies that 
$D^{\alpha}(e^{2\pi\tau x_{1}}S_{\tau}\phi)$ is uniformly bounded 
for $|x_{1}| \geq L$, and we conclude that
$e^{2\pi\tau x_{1}}S_{\tau}\phi \in H^{2}_{-\delta}(|x_{1}| \geq L)$,
as desired.
\end{proof}


The following identity, which follows by integration by parts, is at the core 
of the results of this section, and we consider it interesting in
its own.


\begin{proposition}
\label{propn: K_{tau}_S_{tau}}
Let $u \in H^{2}(M)$ satisfy $H_{V,W}u = 0$ in $M$. 
Let $J : H^{2}(M) \twoheadrightarrow H^{1}(M)$ 
be the compact embedding, and let 
$E: L^{2}(M) \ra L^{2}(T)$ denote the 
extension by zero, so that $EXJu \in L^{2}_{c}(T)$.
For $x \in M_{+}$ we have the identity 
\begin{equation}
\label{eqn: K_{tau}_S_{tau}}
K_{\tau}(EXJu)(x) 
= S_{\tau}[(\Lambda_{V,W} - \Lambda_{0,0})\tr^{-}(u)](x).
\end{equation}
\end{proposition}


\begin{proof}
Let $x \in M_{+}$ be fixed, so that $\Gamma_{\tau}(x,\cdot)$ 
is smooth and harmonic in a neighborhood $M$. From the integral 
representation 
\eqref{eqn: kernel_K_{tau}} and the fact that $EXJu$ is supported 
in $M$ we get that
\begin{equation}
\label{eqn: K_{tau}Xu}
K_{\tau}(EXJu)(x) = \int_{T}\Gamma_{\tau}(x,\cdot)EXJu 
= \int_{M}\Gamma_{\tau}(x,\cdot)Xu
= \int_{M}\Gamma_{\tau}(x,\cdot)(2V\cdot Du + (V^{2} + D\cdot V + W)u).
\end{equation}

From the integral representation in Proposition \ref{propn: S_{0}_S_{tau}}
and the definition of the DN map 
\eqref{defn: DN_magnetic} we have that
\begin{align*}
S_{\tau} (\Lambda_{V,W}\tr^{-}(u))(x)
& = \lan \Lambda_{V,W}\tr^{-}(u), \tr(\Gamma_{\tau}(x,\cdot))\ran \\
& = \int_{M} -Du \cdot D\Gamma_{\tau}(x,\cdot) 
+ V\cdot (\Gamma_{\tau}(x,\cdot)Du - uD\Gamma_{\tau}(x,\cdot)) 
+ (V^{2} + W)u\Gamma_{\tau}(x,\cdot).
\end{align*}

From the integral representation in Proposition \ref{propn: S_{0}_S_{tau}}, the 
harmonicity of $\Gamma_{\tau}(x,\cdot)$ in $M$, the definition \eqref{eqn: DN_free}  
of the DN map $\Lambda_{0,0}$ and its symmetry we have that
\[
S_{\tau}(\Lambda_{0,0}\tr^{-}(u))(x)
= \lan \Lambda_{0,0}\tr^{-}(u), \tr(\Gamma_{\tau}(x,\cdot))\ran 
= \lan \Lambda_{0,0}\tr(\Gamma_{\tau}(x,\cdot)), \tr^{-}(u) \ran
= \int_{M}-D\Gamma_{\tau}(x,\cdot)\cdot Du.
\]

Therefore, we obtain 
\begin{equation}
\label{eqn: S_{tau}_DN}
S_{\tau}[(\Lambda_{V,W} - \Lambda_{0,0})\tr^{-}(u)](x) 
= \int_{M} V\cdot (\Gamma_{\tau}(x,\cdot)Du 
- uD\Gamma_{\tau}(x,\cdot))
+ (V^{2} + W)u\Gamma_{\tau}(x,\cdot).
\end{equation}

From Proposition \ref{propn: Green} we have that 
\[
\int_{M} \Gamma_{\tau}(x,\cdot)(V\cdot Du 
+ (D\cdot V)u) + V\cdot (uD\Gamma_{\tau}(x,\cdot)) 
= 0,
\]

which implies the equality of \eqref{eqn: K_{tau}Xu} 
and \eqref{eqn: S_{tau}_DN} as we wanted.
\end{proof}


\begin{proposition}
\label{propn: compactness}
The operator 
$\tr \circ S_{\tau}(\Lambda_{V,W} - \Lambda_{0,0})$ in 
$H^{3/2}(\pa M)$ is compact.
\end{proposition}


\begin{proof}
Recall that for $f \in H^{3/2}(\pa M)$ we have  
$D_{V,W}f := u \in H^{2}(M)$ as the solution to the Dirichlet
problem
\[
\biggl\{
\begin{array}{rll}
H_{V,W}u \hspace{-2mm} & = 0 & \ \textrm{in} \ \ M_{-}, \\
u \hspace{-2mm} & = f & \ \textrm{on} \ \ \pa M.
\end{array}
\]

Let $J : H^{2}(M) \twoheadrightarrow H^{1}(M)$ 
be the compact embedding, and let $E: L^{2}(M) \ra L^{2}(T)$ denote 
the extension by zero. Then we have 
$EXJu \in L^{2}_{c}(T)$, and so $K_{\tau}EXJu \in H^{2}_{loc}(T)$.
For $x \in M_{+}$ we can rewrite
the result from Proposition \ref{propn: K_{tau}_S_{tau}} as 
\begin{equation}
\label{eqn: S_{tau}DN_K_{tau}X}
S_{\tau}(\Lambda_{V,W} - \Lambda_{0,0})f(x) = K_{\tau}EXJu(x).
\end{equation}


The trace of a single layer potential is well-defined, so we can take 
traces on both sides of \eqref{eqn: S_{tau}DN_K_{tau}X} to obtain
\begin{equation}
\label{eqn: factorization}
\tr \circ S_{\tau}(\Lambda_{V,W} - \Lambda_{0,0})
= \tr K_{\tau}EXJD_{V,W}.
\end{equation}

To prove the result it suffices to
express the right-hand side of \eqref{eqn: factorization}
as a composition of bounded operators,
together with the compact operator $J$.
Recall that $K_{\tau} = e^{-2\pi\tau x_{1}}G_{\tau}e^{2\pi\tau x_{1}}$.
All of the following are continuous operators,
\[
D_{V,W}: H^{3/2}(\pa M) \ra H^{2}(M), \ \ 
J : H^{2}(M) \twoheadrightarrow H^{1}(M), \ \
X : H^{1}(M) \ra L^{2}(M), 
\]
\[
e^{2\pi \tau x_{1}}E : L^{2}(M) \ra L^{2}_{\delta}(T), \ \ 
G_{\tau}: L^{2}_{\delta}(T) \ra H^{2}_{-\delta}(T), \ \  
\tr \circ e^{-2\pi \tau x_{1}} : H^{2}_{-\delta}(T) \ra H^{3/2}(\pa M).
\]

and this completes the proof.
\end{proof}


\begin{corollary}
\label{cor: Fredholm}
The operator $I + \tr\circ S_{\tau}(\Lambda_{V,W} - \Lambda_{0,0})$ in 
$H^{3/2}(\pa M)$ is continuous and invertible. In particular, the boundary
values of the CGO, constructed as $u = h + e^{-2\pi\tau x_{1}}r$, can be
determined by boundary measurements as 
\[
\tr(u) = (I + \tr\circ S_{\tau}(\Lambda_{V,W} - \Lambda_{0,0}))^{-1}\tr(h).
\]
\end{corollary}


\begin{proof}
The uniqueness of the solution to (BE) in Theorem \ref{thm: equivalent} 
implies that the operator 
$I + \tr\circ S_{\tau}(\Lambda_{V,W} - \Lambda_{0,0})$ is injective.
From Proposition \ref{propn: compactness} and Fredholm's alternative it follows
that it is bijective, and therefore invertible by the Open Mapping Theorem.
The fact that the boundary values of the CGO are given by the 
expression above follows from Theorem \ref{thm: equivalent}.
\end{proof}


\section{Reconstruction of the magnetic field}


As mentioned in the introduction and the previous chapter, the purpose of 
proving the Carleman estimate Theorem \ref{thm: carleman_full} is using it to 
construct many special solutions to the equation $H_{V,W}u = 0$, in order
to recover the magnetic field $\curl V$. In contrast to the previous chapter,
we restrict our attention to a particular kind of harmonic functions and show 
that we can find an amplitude, i.e. a correction factor, that gives appropriate
estimates for the remainder term. The choice of the special harmonic functions 
$h = e^{\pm 2\pi |m|x_{1}}e_{m}(x')$, and not any arbitrary harmonic
function, comes from the fact that $(Dh)^{2} = 0$. We elaborate more
on this in a remark after the construction in Proposition \ref{propn: harmonic_modify}.
These ideas follow the so-called WKB method, and are presented systematically 
for more general settings in Sections 2 and 5 in \cite{DKSaU}; see also Section 
4 in \cite{KSjU} or Sections 2 and 3 in \cite{DKSjU}. We proceed analogously
to the proof Lemma 6.1. in \cite{Sa1} in the Euclidean setting.

\bigskip
After the amplitude has been constructed, we define an analog of the 
scattering transform from \cite{N} and \cite{Sa1}, and show that
the estimates for the remainder term allow to disregard them, so that from
the boundary measurements we are able to recover integrals involving the
magnetic potential. After some work, we will show that this allows for the
reconstruction of the magnetic field.

\bigskip
The exposition here follows closely the method from \cite{Sa1}, until the
part involving the analog of the scattering transform. The difference of the 
methods at this point is due to the fact that the integrals contain terms that 
are real exponentials (like in the Laplace transform), rather than complex 
exponentials (like in the Fourier transform). This difference seems difficult, 
if not impossible, to reconcile.

\bigskip
As in the previous chapter, we denote by 
$X := 2V\cdot D + (V^{2} + D\cdot V + W)$ the compactly supported 
first order differential operator, so that 
$H_{V,W} = D^{2} + X$. For the rest of the chapter, the potentials $V,W$ 
and the constants $R, \delta$ are fixed. Any quantities involving 
them, for instance the constants in the inequalities from the previous 
chapters, will be regarded as constants.


\subsection{Construction of CGOs}


A special family of harmonic solutions in $T$ is given by the products 
$e^{\pm 2\pi|m|x_{1}}e_{m}(x')$ for any $m \in \bZ^{d}$. These
solutions are analogous to the Calder\'on complex exponential solutions 
$e^{2\pi i\zeta \cdot x}$, where $\zeta \in \bC^{d}$ and $\zeta \cdot \zeta = 0$. 
In our case, $\zeta \in \bC^{d}$ is replaced by $(\pm i|m|,m) \in i\bR \times \bZ^{d}$. 
We construct the correction terms for these harmonic functions in order
to solve the equation $H_{V,W}u = 0$, and make more explicit the 
corresponding estimates for the correction terms. 


\begin{proposition}
\label{propn: CGO_first}
Let $1/2 < \delta < 1$ and assume that $0$ is not an eigenvalue of 
$H_{V,W}$ in $M$. Let $m \in \bZ^{d}$ and let $\tau > 0$ be such that 
$\tau^{2}\notin \spec(-\Delta_{g_{0}})$.
Then there exists a unique $r_{m,\tau} \in H^{2}_{-\delta}(T)$ such that 
\[
u_{m,\tau} := e^{-2\pi|m|x_{1}}e_{m}(x') + e^{-2\pi \tau x_{1}}r_{m,\tau}
\]

satisfies $H_{V,W}u_{m,\tau} = 0$. Moreover, the correction term satisfies  the estimates
\[
\|r_{m,\tau}\|_{L^{2}_{-\delta}(T)} \cleq \frac{e^{2\pi|\tau - |m||R}\lan m \ran}{\tau}, \ \ 
\|r_{m,\tau}\|_{H^{1}_{-\delta}(T)} \cleq e^{2\pi|\tau - |m||R}\lan m \ran.
\]

In particular, we obtain $\|r_{m,\tau}\|_{L^{2}_{-\delta}(T)} \cleq 1$ if $|\tau - |m|| \cleq 1$.
\end{proposition}


\begin{proof}
Given that $D^{2}(e^{-2\pi|m|x_{1}}e_{m}(x')) = 0$, we have that 
$u_{m,\tau}$ solves $H_{V,W}u_{m,\tau} = 0$ if and only if $r_{m,\tau}$
solves the equation
\[
e^{2\pi \tau x_{1}}H_{V,W}e^{-2\pi \tau x_{1}}r_{m,\tau}
= -e^{2\pi \tau x_{1}}H_{V,W}(e^{-2\pi|m|x_{1}}e_{m}(x'))
= -e^{2\pi \tau x_{1}}X(e^{-2\pi|m|x_{1}}e_{m}(x')).
\]

The right-hand side is compactly supported and thus in $L^{2}_{\delta}(T)$.
Therefore, Theorem \ref{thm: carleman_full} gives the existence and uniqueness 
of a solution in $H^{2}_{-\delta}(T)$. Finally, we observe that the right-hand 
side equals
\[
f := -e^{2\pi \tau x_{1}}X(e^{-2\pi|m|x_{1}}e_{m}(x'))
= -[2V\cdot (i|m|,m) + (V^{2} + D\cdot V + W)]e^{2\pi(\tau - |m|)x_{1}}e_{m}(x'),
\]

so we can bound it by
$\|f\|_{L^{2}_{\delta}(T)} \cleq e^{2\pi|\tau - |m|||R|}\lan m\ran$.
The estimate for the correction term $r_{m,\tau}$ follows from 
Theorem \ref{thm: carleman_full}.
\end{proof}


The estimates of Proposition \ref{propn: CGO_first} for the correction term are not sharp 
enough to allow us to neglect them in a later ``asymptotic expansion''. In order to 
improve the estimates for the correction term, we need to modify the harmonic 
function $e^{-2\pi|m|x_{1}}e_{m}(x')$ appropriately as we show next.


\begin{proposition}
\label{propn: harmonic_modify}
Let $1/2 < \delta < 1$. There exist $\vep, \sigma > 0$ such that for
any $m \in \bZ^{d}$, with $|m|$ sufficiently large, 
there is a smooth function $a_{m}(x_{1},x')$, such that 
$a_{m} - 1$ is supported on $|x_{1}|\leq 2|m|^{\sigma}$, and 
\[
b_{m}(x_{1},x') := e^{2\pi|m|x_{1}}H_{V,W}
e^{-2\pi |m|x_{1}}e_{m}(x')a_{m},
\]

is supported on $|x_{1}|\leq 2|m|^{\sigma}$ with 
$\|b_{m}\|_{L^{2}_{\delta}(T)} \cleq |m|^{1 - \vep}$.
\end{proposition}


\begin{remark}
For the rest of the chapter, the notation $a_{m}$, $b_{m}$ does
not represent the Fourier coefficients of some functions as in previous
chapters.
\end{remark}


\begin{proof}
We compute the conjugated operators
\[
e^{2\pi|m|x_{1}}De^{-2\pi |m|x_{1}}e_{m}(x') 
= e_{m}(x')[(i|m|,m) + D],
\]
\begin{align*}
e^{2\pi|m|x_{1}}D^{2}e^{-2\pi|m|x_{1}}e_{m}(x')
= e_{m}(x')[(i|m|,m) + D]^{2}
= e_{m}(x')[2(i|m|,m)\cdot D + D^{2}].
\end{align*}

Therefore, we have the conjugation identity for operators
\begin{equation}
\label{eqn: conjugation_harmonic}
e^{2\pi|m|x_{1}}H_{V,W}
e^{-2\pi|m|x_{1}}e_{m}(x')
= e_{m}(x')[2(i|m|,m)\cdot (D + V) + H_{V,W}].
\end{equation}

We could define $a_{m} := \exp(v_{m})$, where $v_{m}$ is the 
solution of the equation $(i|m|,m)\cdot (Dv_{m} + V) = 0$. This 
equation can be rewritten as 
\[
iD_{x_{1}}v_{m} + \frac{m}{|m|}D_{x'}v_{m} 
= -\biggl(iF + \frac{m}{|m|}G\biggr).
\]

From Theorem \ref{thm: delta_bar_equation} we know that this equation
has a unique solution which decays, and is bounded with 
bounded derivatives of all orders. The only inconvenient with this
is that the term $D^{2}v_{m}$ may not be in $L^{2}_{\delta}(T)$. 
Therefore, we are left to redefine $a_{m} := \exp(w_{m})$, where 
$w_{m} := v_{m}\psi(x_{1}/|m|^{\sigma})$, with $\sigma > 0$ to be 
determined and $\psi$ a cutoff function such that $\psi(t) \equiv 1$
if $|t| \leq 1$ and $\psi(t) \equiv 0$ if $|t| \geq 2$.
With this we have $a_{m} - 1$ is supported on $|x_{1}| \leq 2|m|^{\sigma}$
and
\[
(D + V)a_{m} 
= a_{m}\biggl[(Dv_{m} + V)\psi\biggl(\frac{x_{1}}{|m|^{\sigma}}\biggr)
+ \biggl(1 - \psi\biggl(\frac{x_{1}}{|m|^{\sigma}}\biggr)\biggr)V
+ \frac{v_{m}}{2\pi i|m|^{\sigma}}\psi'
\biggl(\frac{x_{1}}{|m|^{\sigma}}\biggr)(1,0,\ldots, 0)\biggr].
\]

Because $V$ is compactly supported, we see 
that the second term vanishes if $|m|$ is sufficiently large. 
Moreover, the dot product of $(i|m|,m)$ with first term vanishes 
(by construction). From this and \eqref{eqn: conjugation_harmonic} 
we are left with
\[
b_{m} = e_{m}(x')[2(i|m|,m)\cdot (D + V)a_{m} + H_{V,W}a_{m}]
= e_{m}(x')\biggr[a_{m}\frac{2i|m|v_{m}}{2\pi i|m|^{\sigma}}
\psi'\biggl(\frac{x_{1}}{|m|^{\sigma}}\biggr) + H_{V,W}a_{m}\biggr].
\]

The first term is supported on $|m|^{\sigma} \leq |x_{1}| \leq 2|m|^{\sigma}$. 
Using the boundedness of $a_{m}$ and the decay estimates for $v_{m}$
from Theorem \ref{thm: delta_bar_equation}, we can bound the $L^{2}_{\delta}(T)$ 
norm of the first term by
\[
\frac{|m|}{|m|^{\sigma}}
\biggl(\int_{|m|^{\sigma}}^{2|m|^{\sigma}}
\frac{1}{|x_{1}|^{2}}\lan x_{1}\ran^{2\delta}dx_{1}\biggr)^{1/2}
\cleq \frac{|m|}{|m|^{\sigma}}
\cdot |m|^{\sigma(2\delta - 1)/2}
= |m|^{1 + \sigma(2\delta - 3)/2}.
\]

For the second term, we use that $H_{V,W}a_{m} = (D^{2} + X)a_{m}$. 
The term $Xa_{m}$ represents no problem, as $X$ is compactly supported
(in $|x_{1}| \leq R$) and $a_{m}$ has bounded derivatives of all orders
by Theorem \ref{thm: delta_bar_equation}. 
We are left with $D^{2}a_{m} = a_{m}(D^{2}w_{m} + (Dw_{m})^{2})$, which 
is supported on $|x_{1}| \leq 2|m|^{\sigma}$.
In addition to the boundedness of $a_{m}$, from Theorem 
\ref{thm: delta_bar_equation} we also know that 
$|Dw_{m}|, |D^{2}w_{m}| \cleq \lan x_{1}\ran^{-1}$.
Therefore we can bound the $L^{2}_{\delta}(T)$ norms of these 
terms by						
\begin{align*}
\|D^{2}a_{m}\|_{L^{2}_{\delta}(T)} 
& \leq \|a_{m}D^{2}w_{m}\|_{L^{2}_{\delta}(T)} 
+ \|a_{m}(Dw_{m})^{2}\|_{L^{2}_{\delta}(T)} \\
& \cleq \biggl(\int_{0}^{2|m|^{\sigma}}
\frac{1}{\lan x_{1}\ran^{2}}\lan x_{1}\ran^{2\delta}dx_{1}\biggr)^{1/2}
+ \biggl(\int_{0}^{2|m|^{\sigma}}
\frac{1}{\lan x_{1}\ran^{4}}\lan x_{1}\ran^{2\delta}dx_{1}\biggr)^{1/2} \\
& \cleq \biggl(\int_{0}^{2|m|^{\sigma}}
\frac{1}{\lan x_{1}\ran^{2}}\lan x_{1}\ran^{2\delta}dx_{1}\biggr)^{1/2}
\cleq |m|^{\sigma(2\delta - 1)/2}. 
\end{align*}

Taking any $\sigma > 1$ we obtain that 
$\sigma(2\delta - 1)/2 > 1 + \sigma(2\delta - 3)/2$. For instance, if 
$\sigma = 2$, then we ensure that all these exponents are less than $1$,
as we wanted to prove.
\end{proof}


\begin{remark}
In the setting of Proposition \ref{propn: CGO_first}, the choice $a_{m} \equiv 1$ gives 
compact support for $b_{m}$, but we only obtain 
$\|b_{m}\|_{L^{2}_{\delta}(T)} \cleq |m|$.
\end{remark}


\begin{remark}
Observe that if $h = e^{-2\pi|m|x_{1}}e_{m}(x')$, then the condition 
$(Dh)^{2} = 0$ makes the higher order terms in 
\eqref{eqn: conjugation_harmonic} disappear, leaving 
only to appropriately disregard the next order terms (in this case 
of order $|m|$). 
\end{remark}


We use Proposition \ref{propn: harmonic_modify} to construct another solution to the equation 
$H_{V,W}u = 0$, whose correction term has small norm. We observe that the 
``main terms'' ($e^{-2\pi|m|x_{1}}e_{m}(x')$ and $e^{-2\pi|m|x_{1}}e_{m}(x')a_{m}$)
of the two solutions coincide for $|x_{1}| \geq 2|m|^{\sigma}$, and we later prove
that the corrected solutions must coincide.


\begin{proposition}
\label{propn: CGO_second}
Let $1/2 < \delta < 1$, and let $\vep, \sigma > 0$ be as in Proposition \ref{propn: harmonic_modify}. 
Let $m \in \bZ^{d}$, with $|m|$ sufficiently large, and let 
$\tau > 0$ such that $\tau^{2} \notin \spec(-\Delta_{g_{0}})$.
There exists a unique function $\wilde{r}_{m,\tau} \in H^{2}_{-\delta}(T)$,
such that 
\[
\wilde{u}_{m,\tau} := e^{-2\pi|m|x_{1}}e_{m}(x')a_{m} 
+ e^{-2\pi\tau x_{1}}\wilde{r}_{m,\tau},
\]

satisfies $H_{V,W}\wilde{u}_{m,\tau} = 0$. Moreover, the correction term satisfies
the estimates
\[
\|\wilde{r}_{m,\tau}\|_{L^{2}_{-\delta}(T)} 
\cleq \frac{e^{4\pi|\tau - |m|||m|^{\sigma}}|m|^{1 - \vep}}{\tau}, \ \ 
\|\wilde{r}_{m,\tau}\|_{H^{1}_{-\delta}(T)} 
\cleq e^{4\pi|\tau - |m|||m|^{\sigma}}|m|^{1 - \vep}.
\]

In particular, if $|\tau - |m|||m|^{\sigma} \cleq 1$, then 
\[
\|\wilde{r}_{m,\tau}\|_{L^{2}_{-\delta}(T)} 
\cleq |m|^{-\vep}, \ \ 
\|\wilde{r}_{m,\tau}\|_{H^{1}_{-\delta}(T)} \cleq |m|^{1 - \vep}.
\]
 
In addition, if $K \sse T$ is a compact set, then
\[
\|e^{2\pi(|m| - \tau)x_{1}}\wilde{r}_{m,\tau}\|_{L^{2}(K)} \cleq |m|^{-\vep}, \ \
\|e^{2\pi|m| x_{1}}D(e^{-2\pi \tau x_{1}}\wilde{r}_{m,\tau})\|_{L^{2}(K)} 
\cleq |m|^{1 -\vep},
\]

where the constant of the inequality may depend on $K$.
\end{proposition}


\begin{proof}
We have that $H_{V,W}\wilde{u}_{m,\tau} = 0$ if and only if there exists 
$\wilde{r}_{m,\tau}$ which solves 
\[
e^{2\pi \tau x_{1}}H_{V,W}e^{-2\pi \tau x_{1}}\wilde{r}_{m,\tau} 
= -e^{2\pi(\tau - |m|)x_{1}}b_{m}.
\]

The conclusion follows from Theorem \ref{thm: carleman_full} and 
Proposition \ref{propn: harmonic_modify}.
\end{proof}


\begin{proposition}
\label{propn: CGO_same}
The solutions to the equation $H_{V,W}u = 0$ constructed in 
Proposition \ref{propn: CGO_first} and 
Proposition \ref{propn: CGO_second} are equal.
\end{proposition}


\begin{proof}
We write 
\begin{align*}
\wilde{u}_{m,\tau} 
& = e^{-2\pi|m|x_{1}}e_{m}(x')a_{m} + e^{-2\pi \tau x_{1}}\wilde{r}_{m,\tau} \\
& = e^{-2\pi|m|x_{1}}e_{m}(x') + e^{-2\pi \tau x_{1}}(\wilde{r}_{m,\tau} 
+ e^{2\pi(\tau - |m|)x_{1}}e_{m}(x')(a_{m} - 1)).
\end{align*}
We have that $a_{m} - 1$ is a smooth bounded function 
supported on $|x_{1}| \leq 2|m|^{\sigma}$; in particular,
$e^{2\pi(\tau - |m|)x_{1}}e_{m}(x')(a_{m} - 1) \in H^{2}_{-\delta}(T)$.
The fact that $H_{V,W}\wilde{u}_{m,\tau} = 0$ and the uniqueness
from Proposition \ref{propn: CGO_first} give that we must have 
$\wilde{u}_{m,\tau} = u_{m,\tau}$. 
\end{proof}


\subsection{Transforms and integrals}


Recall that for the Laplacian $H_{0,0} := D^{2}$ in $M$ 
there is a well-defined Dirichlet-to-Neumann map $\Lambda_{0,0}$. Moreover,
this map is symmetric. If $u$ and $\phi$ are solutions to $H_{V,W}u = 0$ and 
$H_{0,0}\phi = 0$, respectively, then we have the integral identities
\[
\lan \Lambda_{V,W} \tr(u),\tr(\phi) \ran
= \int_{M}-D u\cdot D \phi + V\cdot (\phi Du - u D \phi)
+ (V^{2} + W)u\phi,
\]
\[
\lan \Lambda_{0,0}\tr(u),\tr(\phi)\ran
= \lan \Lambda_{0,0}\tr(\phi),\tr(u)\ran
= \int_{M}-Du\cdot D\phi,
\]

and so we obtain
\begin{equation}
\label{eqn: difference_DN}
\lan (\Lambda_{V,W} - \Lambda_{0,0})\tr(u),\tr(\phi) \ran
= \int_{M} V\cdot (\phi Du - u D\phi)
+ (V^{2} + W)u\phi.
\end{equation}

Let $m, n \in \bZ^{d}$, $m, n \neq 0$, be fixed. Let $m_{N} := Nm \in \bZ^{d}$, 
where $N > 0$ is a large integer parameter. Observe first, that 
$m_{N}/|m_{N}| = m/|m|$. With the notation from the last section,
we see from Theorem \ref{thm: delta_bar_equation} that $v_{m_{N}} = v_{m}$, as 
$m_{N}/|m_{N}| = m/|m|$ and both functions are the decaying solutions to the 
equation
\[
iD_{x_{1}}v + \frac{m}{|m|}\cdot D_{x'}v = -\biggl(iF + \frac{m}{|m|}\cdot G\biggr).
\]

According to the construction in Proposition \ref{propn: harmonic_modify}, 
if $N$ is large enough (depending only on $R$ and $\sigma$) and 
$|x_{1}| \leq R$, then 
\begin{equation}
\label{eqn: harmonic_modify_2}
a_{m_{N}}(x_{1},x') 
:= \exp\biggl(v_{m_{N}}\psi\biggl(\frac{x_{1}}{|m_{N}|^{\sigma}}\biggr)\biggr)
= \exp\biggl(v_{m}\psi\biggl(\frac{x_{1}}{|m_{N}|^{\sigma}}\biggr)\biggr)
= \exp(v_{m}) =: \wilde{a}_{m}(x_{1},x').
\end{equation}

Let $u_{m_{N},\tau}$ be the solution to $H_{V,W}u= 0$ constructed in 
the previous section as the correction of the harmonic function 
$e^{-2\pi|m_{N}|x_{1}}e_{m_{N}}(x')$. We choose 
$\tau = \tau(m,N,\sigma)$ to satisfy 
$|\tau - |m_{N}|||m_{N}|^{\sigma} \cleq 1$, so we have the 
last estimates in Proposition \ref{propn: CGO_second} for the correction
$\wilde{r}_{m_{N},\tau}$ on the compact set $M$. For the choice of 
test function we consider the harmonic function 
$\phi_{m_{N},n} = e^{2\pi|m_{N} + n|x_{1}}e_{-(m_{N} + n)x'}$.
Using \eqref{eqn: difference_DN} we define the transform	 
\begin{align*}
T(m,n,N) 
& := \lan (\Lambda_{V,W} - \Lambda_{0,0})\tr(u_{m_{N},\tau}),
\tr(\phi_{m_{N},n})\ran \\
& = \int_{M} V\cdot (\phi_{m_{N},n} Du_{m_{N},\tau} 
- u_{m_{N},\tau} D\phi_{m_{N},n}) 
+ (V^{2} + W)u_{m_{N},\tau}\phi_{m_{N},n},
\end{align*}

From Corollary \ref{cor: Fredholm} we obtain that the transform $T(m,n,N)$
is determined by the knowledge of $M$ and $\Lambda_{V,W}$.
In \cite{N} and \cite{Sa1} this is referred as the \textit{scattering 
transform}; that name does not seem appropriate in our setting.
Let us look at each term of the previous expression on 
$M \sse [-R,R]\times \bT^{d}$. From Proposition \ref{propn: CGO_second} 
and \eqref{eqn: harmonic_modify_2} we have that
\[
u_{m_{N},\tau} = e^{-2\pi|m_{N}|x_{1}}e_{m_{N}}(x')\wilde{a}_{m}
+ e^{-2\pi\tau x_{1}}\wilde{r}_{m_{N},\tau},
\]
\[
Du_{m_{N},\tau} = e^{-2\pi|m_{N}|x_{1}}e_{m_{N}}(x')
\wilde{a}_{m}[(i|m_{N}|,m_{N}) + Dv_{m}] 
+ D(e^{-2\pi\tau x_{1}}\wilde{r}_{m_{N},\tau}),
\]

where we have used that $\wilde{a}_{m} := \exp(v_{m})$ for the second 
expression. From Theorem \ref{thm: delta_bar_equation} and 
Proposition \ref{propn: CGO_second} we obtain that 
\[
u_{m_{N},\tau} = e^{-2\pi|m_{N}|x_{1}}e_{m_{N}}(x')\wilde{a}_{m}
+ e^{-2\pi|m_{N}|x_{1}}R_{1},
\]
\[
Du_{m_{N},\tau} 
= e^{-2\pi|m_{N}|x_{1}}e_{m_{N}}(x')\wilde{a}_{m}(i|m_{N}|,m_{N})
+ e^{-2\pi|m_{N}|x_{1}}R_{2},
\]

with $\|R_{i}\|_{L^{2}(M)} = o(N)$. We also have 
$D\phi_{m_{N},n} = (-i|m_{N} + n|, -(m_{N} + n))\phi_{m_{N},n}$.
Therefore, 
\[
\phi_{m_{N},n} Du_{m_{N},\tau}
= e^{2\pi (|m_{N} + n| - |m_{N}|)x_{1}}e_{-n}(x') 
\wilde{a}_{m}(i|m_{N}|,m_{N})
+ e^{2\pi (|m_{N} + n| - |m_{N}|)x_{1}}\wilde{R}_{1},
\]
\[
u_{m_{N},\tau} D\phi_{m_{N},n}
= e^{2\pi (|m_{N} + n| - |m_{N}|)x_{1}}e_{-n}(x')
\wilde{a}_{m}(-i|m_{N} + n|, -(m_{N} + n))
+ e^{2\pi(|m_{N}+ n| - |m_{N}|) x_{1}}\wilde{R}_{2},
\]
\[
u_{m_{N},\tau}\phi_{m_{N},n}
= e^{2\pi (|m_{N} + n| - |m_{N}|)x_{1}}e_{-n}(x')\wilde{a}_{m}
+ e^{2\pi(|m_{N}+ n| - |m_{N}|) x_{1}}\wilde{R}_{3},
\]

with $\|\wilde{R}_{i}\|_{L^{2}(M)} = o(N)$.
Finally, observe that 
\[
|m_{N} + n| - |m_{N}| 
= \frac{(|m_{N}|^{2} + 2m_{N}\cdot n + |n|^{2}) - |m_{N}|^{2}}
{|m_{N} + n| + |m_{N}|}
= \frac{N}{N}\cdot \frac{2m\cdot n + \frac{|n|^{2}}{N}}
{|m + \frac{n}{N}| + |m|} 
\ra \frac{m\cdot n}{|m|} =: \mu_{m,n},
\]

as $N \ra +\infty$.
These computations and the estimates from Theorem \ref{thm: delta_bar_equation} 
and Proposition \ref{propn: CGO_second} give that
\[
\|\phi_{m_{N},n} Du_{m_{N},\tau} 
- e^{2\pi \mu_{m,n}x_{1}}e_{-n}(x')\wilde{a}_{m}(i|m_{N}|,m_{N})\|_{L^{2}(M)}
= o(N),
\]
\[
\|u_{m_{N},\tau} D\phi_{m_{N},n} 
+ e^{2\pi \mu_{m,n}x_{1}}e_{-n}(x')\wilde{a}_{m}(i|m_{N}|,m_{N})\|_{L^{2}(M)}
= o(N),
\]
\[
\|(V^{2} + W)u_{m_{N},\tau}\phi_{m,n}\|_{L^{2}(M)} = o(N).
\]

Thus, from the knowledge of the transform we are able to obtain the integrals
\[
I(m,n) := \lim_{N \ra +\infty}\frac{T(m,n,N)}{2N} 
= \int_{M}e^{2\pi \mu_{m,n} x_{1}}e_{-n}(x')(i|m|,m)\cdot V\wilde{a}_{m}.
\]

We regard these integrals as a ``mixed non-linear transform'', in the sense 
that we have Laplace and Fourier transforms in the real and toroidal variables, 
respectively, and an additional term $\wilde{a}_{m}(x_{1},x')$.


\subsection{Determination of the Fourier coefficients of the magnetic field}


In order to reconstruct the curl of $V$, we could try to remove the
``non-linear'' term $\wilde{a}_{m}$ from the mixed transform, i.e. 
to determine the integrals 
\[
J(m,n) := \int_{M}e^{2\pi \mu_{m,n} x_{1}}e_{-n}(x')(i|m|,m)\cdot V,
\]

and relate them to the integrals $I(m,n)$.
These integrals contain real exponentials, instead of only complex 
exponentials as in \cite{Sa1}. This will turn out in a significantly different 
result. In the appendix, we introduce the necessary notation and prove the 
following result. 


\begin{theorem}
\label{thm: relation_J_I}
We have the following cases depending on the sign of the dot product 
$m \cdot n$:
\begin{enumerate}
\item if $m \cdot n = 0$, then $J(m,n) = 0$,

\item if $m \cdot n > 0$, then
\[
J(m,n) = \sum_{j = 1}^{\infty}\frac{1}{j}\biggl(\frac{-2\pi}{|m|}\biggr)^{j - 1}
I_{j}^{-}(m,n)
\]

\item if $m \cdot n < 0$, then
\[
J(m,n) = \sum_{j = 1}^{\infty}\frac{1}{j}\biggl(\frac{2\pi}{|m|}\biggr)^{j - 1}
I_{j}^{+}(m,n)
\]
\end{enumerate}

Moreover, if $m \cdot n = \pm 1$, then $J(m,n) = I(m,n)$.
\end{theorem}


\subsubsection{Relation between the families $\{I(m,n)\}$ and $\{J(m,n)\}$}


In this subsection will not be concerned with the explicit relations between 
these two families of integrals, but rather on the existence of such relation. 
Let $[p,q] \sse \bR$ be any interval containing $[-R,R]$ so that 
$M \sse [p,q] \times \bT^{d}$. The condition $\supp(V) \sse M$ implies that  
\[
I(m,n) := \int_{M}e^{2\pi \mu_{m,n} x_{1}}e_{-n}(x')(i|m|,m)
\cdot V\wilde{a}_{m}
= \int_{[p,q] \times \bT^{d}}e^{2\pi \mu_{m,n} x_{1}}e_{-n}(x')(i|m|,m)
\cdot V\wilde{a}_{m}.
\]

Recall from \eqref{eqn: harmonic_modify_2} 
that $\wilde{a}_{m} := \exp(v_{m})$ and $(i|m|,m)\cdot(Dv_{m} + V) = 0$, so
that $(i|m|,m)\cdot (D\wilde{a}_{m} + V\wilde{a}_{m}) = 0$. This and the fact
that $(i|m|,m)\cdot D(e^{2\pi \mu_{m,n} x_{1}}e_{-n}(x')) = 0$ allow us to 
rewrite 
\begin{align}
\label{eqn: integrals_I}
I(m,n) & = -\int_{[p,q] \times \bT^{d}}e^{2\pi \mu_{m,n} x_{1}}e_{-n}(x')
(i|m|,m)\cdot D\wilde{a}_{m} \nonumber \\
& = -\int_{[p,q] \times \bT^{d}}(i|m|,m)
\cdot D(e^{2\pi \mu_{m,n} x_{1}}e_{-n}(x') \wilde{a}_{m}) \nonumber \\
& = \frac{-|m|}{2\pi}\biggl(e^{2\pi \mu_{m,n} x_{1}}
\int_{\bT^{d}}e_{-n}(x')\wilde{a}_{m}(x_{1},x')dx'\biggr)
\biggl|_{x_{1} = p}^{x_{1} = q},
\end{align}

where the last equality follows from the Fundamental Theorem of Calculus
and the fact that the torus $\bT^{d}$ has no boundary. 
Recall that we are interested in determining the integrals 
\[
J(m,n) := \int_{M}e^{2\pi \mu_{m,n} x_{1}}e_{-n}(x')(i|m|,m)\cdot V
= \int_{[p,q] \times \bT^{d}}e^{2\pi \mu_{m,n} x_{1}}e_{-n}(x')(i|m|,m)\cdot V.
\]

Using that $(i|m|,m)\cdot (Dv_{m} + V) = 0$, we can proceed as before 
to obtain 
\begin{align}
\label{eqn: integrals_J}
J(m,n) & = -\int_{[p,q] \times \bT^{d}}e^{2\pi \mu_{m,n} x_{1}}e_{-n}(x')
(i|m|,m)\cdot Dv_{m} \nonumber \\
& = -\int_{[p,q] \times \bT^{d}}(i|m|,m)
\cdot D(e^{2\pi \mu_{m,n} x_{1}}e_{-n}(x')v_{m}) \nonumber \\
& = \frac{-|m|}{2\pi}\biggl(e^{2\pi \mu_{m,n} x_{1}}
\int_{\bT^{d}}e_{-n}(x')v_{m}(x_{1},x')dx'\biggr)
\biggl|_{x_{1} = p}^{x_{1} = q}.
\end{align}

Now we show that we can determine the integrals in \eqref{eqn: integrals_J}
from the knowledge of the integrals in \eqref{eqn: integrals_I}. First, let us observe
that these equalities hold for any $p,q$ such that $M \sse [p,q] \times \bT^{d}$.
Therefore, if necessary we may only consider the case when $p,q$ are large.
In addition, observe that for determining the integrals in \eqref{eqn: integrals_J} it 
suffices to determine $v_{m}(x_{1},x')$ for $|x_{1}|$ large. Moreover, by 
Theorem \ref{thm: delta_bar_equation} we have $|v_{m}(x_{1},x')| \ra 0$ as 
$|x_{1}| \ra +\infty$ (uniformly in $x'$), thus the knowledge of 
$\wilde{a}_{m}(x_{1},x') = \exp(v_{m}(x_{1},x')) \ra 1$ for $|x_{1}|$ large 
and the invertibility of $\exp(z)$ near $z = 0$ are sufficent to determine 
$v_{m}(x_{1},x')$. More concretely, we can recover $v_{m}(x_{1},x')$ 
by the power series 
\[
v_{m}(x_{1},x') = \log(\wilde{a}_{m}(x_{1},x'))
= \sum_{j = 1}^{\infty}\frac{(-1)^{j - 1}}{j}(\wilde{a}_{m}(x_{1},x')-1)^{j}.
\]

Then, the problem reduces to recover $\wilde{a}_{m}(x_{1},x')$ for $|x_{1}|$ 
large from the knowledge of the integrals in \eqref{eqn: integrals_I}. Let us consider 
the Fourier series
\[
v_{m}(x_{1},x') = \sum_{k \in \bZ^{d}} v_{m,k}(x_{1})e_{k}(x'),
\]

so that the Fourier coefficient $v_{m,k}(x_{1})$ solves the equation 
\[
iD_{x_{1}}v_{m,k} + \frac{m \cdot k}{|m|}v_{m,k} = -\biggl(iF_{k} 
+ \frac{m}{|m|}\cdot G_{k}\biggr).
\] 

By Theorem \ref{thm: delta_bar_Fourier_ODE}, the solution $v_{m,k}(x_{1})$ 
vanishes in $(-\infty,-R]$ or $[R,+\infty)$ depending whether 
$m\cdot k  \geq 0$ or $m \cdot k \leq 0$, respectively. Thus, for 
$|x_{1}|\geq R$ we have 
\begin{equation}
\label{eqn: solution_v_{m}}
v_{m}(x_{1},x') = \left\{
\begin{array}{ll}
v_{m}^{+}(x_{1},x') := \sum_{m\cdot k > 0}v_{m,k}(x_{1})e_{k}(x') 
& \mbox{if} \ \ x_{1}\geq R, \\
v_{m}^{-}(x_{1},x') := \sum_{m\cdot k < 0}v_{m,k}(x_{1})e_{k}(x') 
& \mbox{if} \ \ x_{1}\leq -R. \\
\end{array} 
\right.
\end{equation}

From this and \eqref{eqn: integrals_J} we obtain
\begin{equation}
\label{eqn: integrals_J_v_{m}}
J(m,n) = \frac{-|m|}{2\pi} \cdot \left\{
\begin{array}{ll}
e^{2\pi \mu_{m,n}q}v_{m,n}(q) 
& \mbox{if} \ \ m\cdot n > 0, \\
-e^{2\pi \mu_{m,n}p}v_{m,n}(p) 
& \mbox{if} \ \ m\cdot n < 0, \\
0 & \mbox{if} \ \ m\cdot n = 0.
\end{array} 
\right.
\end{equation}

Moreover, we also have
\[
\wilde{a}_{m}(x_{1},x') = \exp(v_{m}(x_{1},x')) 
= \left\{
\begin{array}{ll}
\wilde{a}_{m}^{+}(x_{1},x') := \exp(v_{m}^{+}(x_{1},x'))
& \mbox{if} \ \ x_{1}\geq R, \\
\wilde{a}_{m}^{-}(x_{1},x') := \exp(v_{m}^{-}(x_{1},x'))
& \mbox{if} \ \ x_{1}\leq -R. \\
\end{array} 
\right.
\]

Let us consider the Fourier series
\[
\wilde{a}_{m}(x_{1},x') = \sum_{k \in \bZ^{d}}\wilde{a}_{m,k}(x_{1})e_{k}(x').
\]

Given the form of $v_{m}^{\pm}$ from \eqref{eqn: solution_v_{m}}
and the fact that the exponential is a power series, we 
conclude that
\begin{equation}
\label{eqn: exponential_wilde{a}_{m}}
\wilde{a}_{m}(x_{1},x') = \left\{
\begin{array}{ll}
\wilde{a}_{m}^{+}(x_{1},x') = 1 + \sum_{m \cdot k > 0}
\wilde{a}_{m,k}(x_{1})e_{k}(x')
& \mbox{if} \ \ x_{1}\geq R, \\
\wilde{a}_{m}^{-}(x_{1},x') = 1 + \sum_{m \cdot k < 0}
\wilde{a}_{m,k}(x_{1})e_{k}(x')
& \mbox{if} \ \ x_{1}\leq -R, \\
\end{array} 
\right.
\end{equation}

This and \eqref{eqn: integrals_I} give that
\begin{align*}
I(m,n) & = \frac{-|m|}{2\pi}\biggl(e^{2\pi \mu_{m,n} x_{1}}
\int_{\bT^{d}}e_{-n}(x')\wilde{a}_{m}(x_{1},x')dx'\biggr)
\biggl|_{x_{1} = p}^{x_{1} = q} \\
& = \frac{-|m|}{2\pi} \cdot \left\{
\begin{array}{ll}
e^{2\pi \mu_{m,n}q}\wilde{a}_{m,n}(q)
& \mbox{if} \ \ m\cdot n >  0, \\
-e^{2\pi \mu_{m,n}p}\wilde{a}_{m,n}(p)
& \mbox{if} \ \ m\cdot n <  0, \\
0 & \mbox{if} \ \ m\cdot n = 0.
\end{array} 
\right.
\end{align*}

Recall that this holds for any $p, q$ such that $[-R,R] \sse [p,q]$. 
From this and \eqref{eqn: exponential_wilde{a}_{m}} we conclude that
\begin{equation}
\label{eqn: exponential_wilde{a}_{m}_explicit}
\wilde{a}_{m}(x_{1},x') = 1 + \frac{2\pi}{|m|}\cdot\left\{
\begin{array}{rl}
-\sum_{m \cdot n > 0}I(m,n)e^{-2\pi \mu_{m,n}x_{1}}e_{n}(x')
& \mbox{if} \ \ x_{1}\geq R, \\
\sum_{m \cdot n < 0}I(m,n)e^{-2\pi \mu_{m,n}x_{1}}e_{n}(x') 
& \mbox{if} \ \ x_{1}\leq -R, \\
\end{array} 
\right.
\end{equation}

Therefore, we have shown that from the integrals $I(m,n)$ we are able
to determine $\wilde{a}_{m}(x_{1},x')$ for $|x_{1}| \geq R$, which in
turn determines $v_{m}(x_{1},x')$ for $|x_{1}| \geq R$, and so the
integrals $J(m,n)$. The explicit dependence of $J(m,n)$ on the family of 
integrals $\{I(m,k)\}$ is shown in the appendix.										


\subsubsection{Curl vectors and Laplace transform}


Let us show how we can use the integrals $J(m,n)$ to recover the 
Fourier coefficients of $\curl V$. Using that $\supp(V) \sse M$, we integrate by 
parts to compute the mixed transform of the terms involved in the magnetic field 
$\curl V$,
\begin{align*}
\int_{[p,q] \times \bT^{d}}& e^{2\pi \mu_{m,n} x_{1}}e_{-n}(x')D_{x_{1}}G_{j} \\
& = \int_{[p,q] \times \bT^{d}}e^{2\pi \mu_{m,n} x_{1}}e_{-n}(x')i\mu_{m,n} G_{j}
= \int_{[p,q] \times \bT^{d}}e^{2\pi \mu_{m,n} x_{1}}e_{-n}(x')
\biggl(0, \frac{i m\cdot n}{|m|}\delta_{j}\biggr)\cdot V,
\end{align*}
\begin{align*}
\int_{[p,q] \times \bT^{d}} & e^{2\pi \mu_{m,n} x_{1}}e_{-n}(x')D_{x'_{j}}F \\
& = \int_{[p,q] \times \bT^{d}}e^{2\pi \mu_{m,n} x_{1}}e_{-n}(x')n_{j}F 
= \int_{[p,q] \times \bT^{d}}e^{2\pi \mu_{m,n} x_{1}}e_{-n}(x')
(n_{j},0)\cdot V,
\end{align*}
\begin{align*}
\int_{[p,q] \times \bT^{d}}& e^{2\pi \mu_{m,n} x_{1}}e_{-n}(x')D_{x'_{j}}G_{k} \\
& = \int_{[p,q] \times \bT^{d}}e^{2\pi \mu_{m,n} x_{1}}e_{-n}(x')n_{j}G_{k} 
= \int_{[p,q] \times \bT^{d}}e^{2\pi \mu_{m,n} x_{1}}e_{-n}(x')
(0,n_{j}\delta_{k})\cdot V,
\end{align*}

where $\delta_{1},\ldots, \delta_{d}$ are the standard basis vectors in $\bR^{d}$.
This means that we are interested in determining the integrals 
\[
\int_{[p,q] \times \bT^{d}}e^{2\pi \mu_{m,n} x_{1}}e_{-n}(x')(D_{x_{1}}G_{j} - D_{x'_{j}}F)
= \int_{[p,q]\times \bT^{d}}e^{2\pi \mu_{m,n} x_{1}}e_{-n}(x')
\biggl(-n_{j},\frac{i m\cdot n}{|m|}\delta_{j}\biggr)\cdot V,
\]
\[
\int_{[p,q] \times \bT^{d}}e^{2\pi \mu_{m,n} x_{1}}e_{-n}(x')
(D_{x'_{j}}G_{k} - D_{x'_{k}}G_{j})
= \int_{[p,q] \times \bT^{d}}e^{2\pi \mu_{m,n} x_{1}}e_{-n}(x')(0, n_{j}
\delta_{k} - n_{k}\delta_{j})\cdot V.
\]

Therefore, the problem reduces to obtain the ``curl vectors''
\[
\biggl\{\biggl(-n_{j},\frac{i m\cdot n}{|m|}\delta_{j}\biggr), \ \
(0, n_{j}\delta_{k} - n_{k}\delta_{j})\biggr\}
\]

as linear combinations of vectors $\{(i|m|,m)\}$ while keeping $n$ and $\mu_{m,n}$ 
fixed. Moreover, we would like to have this result for many values of $\mu$. We show
in Lemma \ref{lemma: linear_algebra_full}, from the following section, that this is				
indeed the case, so that for fixed $n \neq 0$, the knowledge of the integrals $J(m,n)$ 
allows to determine the integrals
\[
\int_{[p,q] \times \bT^{d}}e^{2\pi \mu_{m,n} x_{1}}e_{-n}(x')(D_{x_{1}}G_{j} - D_{x'_{j}}F)
= \int_{p}^{q}e^{2\pi \mu_{m,n} x_{1}}
\biggl(\int_{\bT^{d}}e_{-n}(x')(D_{x_{1}}G_{j} - D_{x'_{j}}F)dx'\biggr)dx_{1},
\]
\[
\int_{[p,q] \times \bT^{d}}e^{2\pi \mu_{m,n} x_{1}}e_{-n}(x')(D_{x'_{j}}G_{k} - D_{x'_{k}}G_{j})
= \int_{p}^{q}e^{2\pi \mu_{m,n} x_{1}}
\biggl(\int_{\bT^{d}}e_{-n}(x')(D_{x'_{j}}G_{k} - D_{x'_{k}}G_{j})dx'\biggr)dx_{1}.
\]

for a sequence of values of $\mu_{m,n} = \gcd(n)/K$ converging to $0$. For 
$f \in C^{\infty}_{c}([p,q])$, 
its Laplace transform  
\[
F(\mu) := \int_{p}^{q}e^{2\pi \mu x_{1}}f(x_{1})dx_{1}
\]

is an entire function, and therefore its knowledge along a convergent sequence
is enough to recover the entire function $F$ over all $\bC$. We describe 
this reconstruction in Theorem \ref{thm: reconstruction_entire} in the following section. 				
The values of $F$ over the imaginary axis correspond to the Fourier transform of $f$,
and therefore it is possible to reconstruct $f$ from the knowledge of $F$ along
a convergent sequence. This completes the reconstruction of the Fourier coefficients 
of $\curl V$.


\subsection{Appendices} 

 
\subsubsection{Explicit relation between the families $\{I(m,n)\}$ and $\{J(m,n)\}$}
 

Let us prove Theorem \ref{thm: relation_J_I}.
We mentioned before that we were interested in computing 
$v_{m} = \log \wilde{a}_{m}$ by a power series; in particular, we are concerned with
expressions of the form $(\wilde{a}_{m} - 1)^{k}$. In 
\eqref{eqn: exponential_wilde{a}_{m}_explicit} we were able to express the Fourier
series of $\wilde{a}_{m}$ in terms of the integrals $I(m,n)$. 
In particular, for $x_{1} \geq R$ we have 
\[
\wilde{a}_{m}(x_{1},x') - 1 
= \frac{-2\pi}{|m|}\sum_{m\cdot k > 0}I(m,k)e^{-2\pi \mu_{m,k}x_{1}}e_{k}(x').
\] 

Consider the set
$T_{j}^{+}(m,k) = \{\kappa 
= (\kappa_{1},\ldots, \kappa_{j}) \in (\bZ^{d})^{j} : m \cdot \kappa_{i} > 0, 
\ \ \kappa_{1} + \ldots + \kappa_{j} = k\}$. 
Observe that if $\kappa \in T_{j}^{+}(m,k)$, then 
\[
m \cdot k = m \cdot (\kappa_{1} + \ldots + \kappa_{j})
\geq \gcd(m) + \ldots + \gcd(m) = j\gcd(m).
\]

This implies that $T_{j}^{+}(m,k)$ is empty when $j > m\cdot k/\gcd(m)$;
in particular it is empty when $m \cdot k < 0$. Let us define 
\[
I_{j}^{+}(m,k) = \sum_{\kappa \in T_{j}^{+}(m,k)}I(m,\kappa_{1})\cdot I(m,\kappa_{2})
\cdot \ldots \cdot I(m,\kappa_{j})
\]

By our previous observation, we also see that $I_{j}^{+}(m,k) = 0$ if $j > m\cdot k/\gcd(m)$.
Finally, using that $\mu_{m,k}$ is a linear function of $k$ we obtain 
\[
(\wilde{a}_{m}(x_{1},x') - 1)^{j} = \biggl(\frac{-2\pi}{|m|}\biggr)^{j}
\sum_{m\cdot k > 0}I_{j}^{+}(m,k)e^{-2\pi \mu_{m,k}x_{1}}e_{k}(x')
\]

This implies that if $x_{1} \geq R$, then
\[
v_{m}(x_{1},x') = \sum_{j = 1}^{\infty}\frac{(-1)^{j - 1}}{j}
(\wilde{a}_{m}(x_{1},x') - 1)^{j}
= -\sum_{m \cdot k > 0}\biggl(\sum_{j = 1}^{\infty}\frac{1}{j}\biggl(\frac{2\pi}{|m|}\biggr)^{j}
I_{j}^{+}(m,k)\biggr)e^{-2\pi\mu_{m,k}x_{1}}e_{k}(x').
\]

From this and \eqref{eqn: integrals_J_v_{m}} we conclude that if 
$m \cdot n > 0$, then
\begin{equation}
\label{eqn: relation_J_I_positive}
J(m,n) = \frac{-|m|}{2\pi}e^{2\pi \mu_{m,n}q}v_{m,n}(q)
= \sum_{j = 1}^{\infty}\frac{1}{j}\biggl(\frac{2\pi}{|m|}\biggr)^{j - 1}
I_{j}^{+}(m,n)
\end{equation}

Similarly, if $x_{1} \leq -R$ we have 
\[
\wilde{a}_{m}(x_{1},x') - 1 
= \frac{2\pi}{|m|}\sum_{m\cdot k < 0}I(m,k)e^{-2\pi \mu_{m,k}x_{1}}e_{k}(x').
\]

We consider $T_{j}^{-}(m,k) = \{\kappa 
= (\kappa_{1},\ldots, \kappa_{j}) \in (\bZ^{d})^{j} : m \cdot \kappa_{i} < 0, 
\ \ \kappa_{1} + \ldots + \kappa_{j} = k\}$. 
As before, we have that
that $T_{j}^{-}(m,k)$ is empty when $j > -m\cdot k/\gcd(m)$;
in particular it is empty when $m \cdot k > 0$. Let us define 
\[
I_{j}^{-}(m,k) = \sum_{\kappa \in T_{j}^{-}(m,k)}I(m,\kappa_{1})\cdot I(m,\kappa_{2})
\cdot \ldots \cdot I(m,\kappa_{j})
\]

By our previous observation, we also see that $I_{j}^{-}(m,k) = 0$ if 
$j > -m\cdot k/\gcd(m)$. As before, we obtain
\[
(\wilde{a}_{m}(x_{1},x') - 1)^{j} = \biggl(\frac{2\pi}{|m|}\biggr)^{j}
\sum_{m\cdot k < 0}I_{j}^{-}(m,k)e^{-2\pi \mu_{m,k}x_{1}}e_{k}(x')
\]

This implies that if $x_{1} \leq -R$, then
\[
v_{m}(x_{1},x') = \sum_{j = 1}^{\infty}\frac{(-1)^{j - 1}}{j}
(\wilde{a}_{m}(x_{1},x') - 1)^{j}
= \sum_{m \cdot k < 0}\biggl(\sum_{j = 1}^{\infty}
\frac{(-1)^{j - 1}}{j}\biggl(\frac{2\pi}{|m|}\biggr)^{j}
I_{j}^{-}(m,k)\biggr)e^{-2\pi\mu_{m,k}x_{1}}e_{k}(x').
\]

From this and \eqref{eqn: integrals_J_v_{m}} we conclude that if 
$m \cdot n < 0$, then
\begin{equation}
\label{eqn: relation_J_I_negative}
J(m,n) = \frac{|m|}{2\pi}e^{2\pi \mu_{m,n}p}v_{m,n}(p)
= \sum_{j = 1}^{\infty}\frac{1}{j}\biggl(\frac{-2\pi}{|m|}\biggr)^{j - 1}
I_{j}^{-}(m,n)
\end{equation}

It was shown above that $I_{j}^{\pm}(m,n)$ vanish when $j > |m\cdot n|/\gcd(m)$,
so the sum above is actually a finite sum. In particular, if $m,n$ are such that 
$m \cdot n = \pm 1$, which implies $\gcd(m) = 1$, then we obtain that 
$I_{j}^{\pm} = 0$ for $j \geq 2$. Therefore, $J(m,n) = I(m,n)$ if 
$m \cdot n = \pm 1$. 

\begin{remark}
The relation between these two families of integrals in other problems had 
already been noted by Eskin--Ralston in \cite{ER}, and was also used in 
\cite{Sa1}. In their setting the two families ended up being entirely equal, 
not as in our problem where this only seems to be true in certain cases.
\end{remark}


\subsubsection{A linear algebra lemma}


In the previous subsection we were concerned with determining the 
curl vectors 
\[
\biggl\{\biggl(-n_{j},\frac{i m\cdot n}{|m|}\delta_{j}\biggr), \ \ 
(0, n_{j}\delta_{k} - n_{k}\delta_{j})\biggr\}
\]

as linear combinations of vectors $\{(i|m|,m)\}$ while keeping $n$ and $\mu_{m,n}$ fixed. 
We observe that 
\[
\biggl(-n_{j},\frac{i m\cdot n}{|m|}\delta_{j}\biggr) 
= \frac{i}{|m|}(i|m|n_{j},(m\cdot n)\delta_{j}),
\]

so we can regard the family of ``curl vectors'' as 
\[
\{(i|m|n_{j}, (m\cdot n)\delta_{j}), \ \ 
(0, n_{j}\delta_{k} - n_{k}\delta_{j})\}.
\]

Let $n \in \bZ^{d} \sm \{0\}$. Consider the set of points 
\[
U(K) := \{(i|m|,m) : m \in \bZ^{d}, \ \ m\cdot n = \gcd(n), \ \ |m| = K\},
\]

where $\gcd(n)$ denotes the greatest common divisor of all the entries of $n$. 
We will show that if $d \geq 3$, then we can construct infinitely many $K$ such that linear 
combinations of elements in $U(K)$ generate all the curl vectors
\[
\{(iKn_{j},\gcd(n)\delta_{j}),
\ \ (0, n_{i}\delta_{j} - n_{j}\delta_{i})\}.
\]


\begin{remark} 
It may suffice to generate each curl vector for infinitely many $K$, but we will show
that we can do all of them simultaneously.
\end{remark}


The curl vectors and the conditions defining $U(K)$ are homogeneous functions of the
entries of $n$, so we can assume without loss of generality that $\gcd(n) = 1$. Moreover, 
it suffices to generate the first family of curl vectors, as we can express
\[
(0,n_{i}\delta_{j} - n_{j}\delta_{i})
= n_{i}(iKn_{j},\delta_{j}) - n_{j}(iKn_{i},\delta_{i}).
\]

In addition, note that if $\delta_{j} = \alpha_{1}k_{1} + \ldots + \alpha_{N}k_{N}$, with 
$(i|k_{i}|,k_{i}) \in U(K)$, then 
\[
n_{j} = \delta_{j}\cdot n= (\alpha_{1}k_{1} + \ldots + \alpha_{N}k_{N})\cdot n
= \alpha_{1} + \ldots + \alpha_{N},
\] 

and so 
\[
(iKn_{j},\delta_{j}) = \alpha_{1}(iK,k_{1}) + \ldots + \alpha_{N}(iK,k_{N}).
\]

Therefore it suffices to construct infinitely many $K$ such that the set
\[
V(K) := \{ k \in \bZ^{d} : k \cdot n = 1, \ \ |k| = K\}
\]

has $d$ linearly independent vectors. In what follows, we refer to the rank of a 
finite set of vectors as the dimension of the subspace generated by them. We prove 
this in several steps. 


\begin{proposition}
\label{propn: linear_algebra_two}
Let $d \geq 3$ and let $n \in \bZ^{d}$ be such that $\gcd(n) = 1$. 
Then there exist $m_{1}, m_{2} \in \bZ^{d}$ such that 
$\{m_{1},m_{2},n\}$ are linearly independent and 
\[
m_{1}\cdot n = m_{2}\cdot n = 1, \ \ |m_{1}| = |m_{2}|.
\]
\end{proposition}


\begin{proof}
Let $p \in \bZ^{d}$ be such that $p \cdot n = 1$ and is linearly independent with $n$. 
Since $d \geq 3$, there exists $q \in \bZ^{d}\sm \{0\}$ orthogonal to both $n$ and 
$p$. We can define $m_{1} := p - q$ and $m_{2} := p + q$. With this we have 
\[
m_{i}\cdot n = (p \pm q)\cdot n = 1 \pm  0 = 1, \ \ 
|m_{i}|^{2} = |p \pm q|^{2} = |p|^{2} \pm 2p\cdot q + |q|^{2} = |p|^{2} + |q|^{2}.
\]

Moreover, the span of $\{m_{1},m_{2},n\}$ is the same as the span of $\{p,q,n\}$,
from where we conclude that these vectors are linearly independent.
\end{proof}


\begin{remark}
A curious observation is that the only vectors $n \in \bZ^{2}$ for which 
there exist $m_{1},m_{2} \in \bZ^{2}$ such that
\[
m_{1}\cdot n = m_{2}\cdot n = 1, \ \ |m_{1}| = |m_{2}|,
\] 

are the eight vectors $\pm\{(1, 0), (0,1), (1, 1), (1, -1)\}$. This problem appeared at 
the Olimpiada Iberoamericana de Matem\'atica Universitaria 2018.
\end{remark}


\begin{proposition}
\label{propn: linear_algebra_four}
Let $d \geq 3$ and let $\{m_{1},m_{2},n\}$ be as in 
Proposition \ref{propn: linear_algebra_two}. 
Consider the integers $M := |m_{1}|^{2} = |m_{2}|^{2}$,
$N := |n|^{2}$, $P = m_{1}\cdot m_{2}$. Then the vectors
\[
p_{1} = (NP - 1)m_{1} + (1 - MN)m_{2} + (M - P)n, \ \
p_{2} = (1 - MN)m_{1} + (NP - 1)m_{2} + (M - P)n,
\]

satisfy $p_{i}\cdot m_{i} = p_{i}\cdot n = 0$ and $|p_{1}| = |p_{2}|$. Moreover, 
the rank of the set $\{m_{1},m_{2},p_{1},p_{2}\}$ is $3$.
\end{proposition}


\begin{proof}
The computations are direct:
\begin{align*}
p_{i}\cdot m_{i}
= & \ [(NP - 1)m_{i} + (1 - MN)m_{j} + (M - P)n] \cdot m_{i} \\
= & \ (NP - 1)M + (1 - MN)P + (M - P) = 0, \\
p_{i}\cdot n
= & \ [(NP - 1)m_{i} + (1 - MN)m_{j} + (M - P)n]\cdot n \\
= & \ (NP - 1) + (1 - MN) + (M - P)N = 0, \\
|p_{i}|^{2} = & \ |(NP - 1)m_{i} + (1 - MN)m_{j} + (M - P)n|^{2} \\
= & \ (NP - 1)^{2}M + (1 - MN)^{2}M + (M - P)^{2}N \\
& + 2[(NP - 1)(1 - MN)P + (NP - 1)(M - P) + (1 - MN)(M - P)].
\end{align*}

To see that the rank is $3$, we first observe that
$(M - P) > 0$, as $m_{1}$ and $m_{2}$ are linearly independent.
This implies that $n$ is contained in the span of 
$\{m_{1},m_{2},p_{1},p_{2}\}$. As $\{m_{1}, m_{2}, n\}$  
are linearly independent the conclusion follows.
\end{proof}


\begin{lemma}
\label{lemma: linear_algebra_full}
Let $d \geq 3$ and let $n \in \bZ^{d} \sm \{0\}$. 
We can construct infinitely many $K$ for which there are $d$ linearly independent 
vectors in the set $V(K) := \{k \in \bZ^{d}: k\cdot n = \gcd(n), \ |k| = K\}$.
\end{lemma}


\begin{proof}
We can assume without loss of generality that $\gcd(n) = 1$.
Let $\{m_{1},m_{2},p_{1},p_{2}\}$ be as in Proposition \ref{propn: linear_algebra_four}, 
and let $\{q_{4},\ldots, q_{d}\}$ be an orthogonal set of vectors in $\bZ^{d}$ 
perpendicular in addition to $\{m_{1},m_{2},n\}$, so that it is also perpendicular
to the set $\{m_{1},m_{2},p_{1},p_{2}\}$. For fixed nonzero $\alpha$, 
$\alpha_{4}, \ldots, \alpha_{d}$ consider the vectors 
\[
m_{1} \pm \alpha p_{1} \pm \alpha_{4}q_{4} \pm \ldots \pm \alpha_{d}q_{d}, \ \ 
m_{2} \pm \alpha p_{2} \pm \alpha_{4}q_{4} \pm \ldots \pm \alpha_{d}q_{d}.
\]

Note that the dot product of any of these vectors with $n$ equals $1$, as 
$m_{i} \cdot n = 1$ and $p_{i} \cdot n = q_{j} \cdot n = 0$. Moreover, they
have the same norm for any choice of signs, since all the terms are orthogonal to each other
and $|m_{1}| = |m_{2}|$ and $|p_{1}| = |p_{2}|$. Therefore, all these vectors belong to 
the same $V(K)$ for any choice of signs. Linear combinations of these vectors allow to 
obtain the set $\{m_{1},m_{2},p_{1},p_{2},q_{4},\ldots, q_{d}\}$.
The sets $\{m_{1},m_{2},p_{1},p_{2}\}$ and $\{q_{4},\ldots, q_{d}\}$ are perpendicular
and its combined rank is $3 + (d - 3) = d$. Different choices of the integers 
$\alpha$, $\alpha_{4}, \ldots, \alpha_{d}$ give the infinitely many values of $K$.
\end{proof}
 

\subsubsection{Reconstruction of an entire function from values along a 
convergent sequence}


Suppose $F : \bC \ra \bC$ is an entire function and $\{z_{n}\} \sse \bC$ is 
a known sequence such that $z_{n} \ra 0$ and the sequence $\{F(z_{n})\}$ is 
also known. The Taylor coefficients of $F$ can be recovered recursively as follows,
\[
F^{(0)}(0) = \lim_{n \ra +\infty}F(z_{n}), \ \ 
\frac{F^{(m)}(0)}{m!} = \lim_{n \ra +\infty}\frac{1}{z_{n}^{m}}
\biggl(F(z_{n}) - \sum_{k = 0}^{m - 1}\frac{F^{(k)}(0)}{k!}z_{n}^{k}\biggr),
\]
and so $F$ can be reconstructed like this. However, we would like to 
propose a different approach based on Newton's method of 
divided differences for interpolation polynomials.


\begin{definition}
\label{defn: differences}
Let $F^{(0)}(z) := F(z)$ and define the divided differences
\[
F^{(n)}(z_{1},\ldots,z_{n + 1}) := \frac{1}{z_{n} - z_{n + 1}}
(F^{(n - 1)}(z_{1},\ldots, z_{n - 1},z_{n}) - F^{(n - 1)}(z_{1},\ldots, z_{n - 1},z_{n + 1})).
\]
\end{definition}


\begin{remark}
It is clear that divided differences are symmetric with respect to the last two 
elements, i.e.
\[
F^{(n)}(z_{1},\ldots,z_{n}, z_{n + 1}) = F^{(n)}(z_{1},\ldots,z_{n + 1}, z_{n}).
\]

We will not use this, but it is possible to show by the induction that the divided 
differences are indeed symmetric with respect to all its entries. We can see this 
in the particular following result. 
\end{remark}


\begin{proposition}
\label{propn: differences_powers}
Consider the power function $p_{k}(x) := x^{k}$. Then, 
\[
p_{k}^{(n)}(z_{1},\ldots, z_{n + 1}) 
= \sum_{|\alpha| = k - n}z_{1}^{\alpha_{1}}\ldots z_{n + 1}^{\alpha_{n + 1}}.
\] 

In particular, $p_{k}^{(k)} = 1$ and $p_{k}^{(n)} = 0$ if $n > k$. The 
number of monomials in the expression equals the binomial coefficient 
$\binom{k}{n}$.
\end{proposition}


\begin{proof}
We prove this by induction. For $n = 0$ this is true. Then 
\begin{align*}
p_{k}^{(n + 1)}(z_{1},\ldots,z_{n + 1},z_{n + 2})
& = \sum_{|\alpha| = k - n}z_{1}^{\alpha_{1}} \ldots z_{n}^{\alpha_{n}}
\biggl(\frac{z_{n + 1}^{\alpha_{n + 1}} - z_{n + 2}^{\alpha_{n + 1}}}{z_{n + 1} - z_{n + 2}}\biggr) \\
& = \sum_{\substack{|\alpha| = k - n \\ |\beta| = \alpha_{n + 1} - 1}}z_{1}^{\alpha_{1}}
\ldots z_{n}^{\alpha_{n}}z_{n + 1}^{\beta_{1}}z_{n + 2}^{\beta_{2}}
= \sum_{|\gamma| = k - n - 1}z_{1}^{\gamma_{1}}\ldots z_{n + 2}^{\gamma_{n + 2}}.
\end{align*}
\end{proof}


\begin{definition}
\label{defn: absolute}
For an entire function $F$, we define the $n$-th derivative majorant by the 
convergent series
\[
|F|^{(n)}(R) := \frac{1}{n!}\sum_{k = 0}^{\infty}\frac{|F^{(n + k)}(0)|}{k!}R^{k} .
\]
\end{definition}


\begin{proposition}
\label{propn: majorization}
Let $F$ be an entire function $F$ and $z_{i} \in \bC$, $|z_{i}| \leq R$. Then the 
$n$-th derivative majorant dominates the divided differences, i.e.
\[
|F^{(n)}(z_{1},\ldots,z_{n + 1})| \leq |F|^{(n)}(R).
\]  
\end{proposition}


\begin{proof}
Let $F(z) = \sum_{k = 0}^{\infty}a_{k}z^{k}$. The divided differences are linear 
operators, so we obtain
\[
|F^{(n)}(z_{1},\ldots,z_{n + 1})|
\leq \sum_{k = 0}^{\infty}|a_{k}||p_{k}^{(n)}(z_{1},\ldots,z_{n + 1})|
= \sum_{k = 0}^{\infty}|a_{n + k}||p_{n + k}^{(n)}(z_{1},\ldots,z_{n + 1})|,
\]

where we used in the last equality that $p_{k}^{(n)} = 0$ if $k < n$ from 
Proposition \ref{propn: differences_powers}. Also from Proposition \ref{propn: differences_powers}
we obtain the bound 
\[
|p_{n + k}^{(n)}(z_{1},\ldots,z_{n + 1})| \leq \binom{n + k}{n}R^{k}.
\] 

Therefore we conclude that 
\[
|F^{(n)}(z_{1},\ldots,z_{n + 1})| \leq \sum_{k = 0}^{\infty}
\biggl|\frac{F^{(n + k)}(0)}{(n + k)!}\biggr|
\binom{n + k}{n}R^{k}
= |F|^{(n)}(R).
\]
\end{proof}


\begin{theorem}
\label{thm: differences_limit}
Let $F$ be an entire function and let $z_{i} \ra 0$. Then, the Taylor coefficients of 
$F$ can be recovered by the divided differences:
\[
\lim_{m \ra +\infty}F^{(n)}(z_{m + 1},\ldots,z_{m + n + 1}) = \frac{F^{(n)}(0)}{n!}.
\] 
\end{theorem}


\begin{proof}
Proceeding as in the previous proof, we can bound 
\begin{align*}
|F^{(n)}(z_{m + 1},\ldots,z_{m + n + 1}) - a_{n}|
& \leq \sum_{k = 1}^{\infty}|a_{n + k}||p_{n + k}^{(n)}(z_{m + 1},\ldots,z_{m + n + 1})| \\
& \leq |F|^{(n)}(\max\{|z_{m + 1}|,\ldots,|z_{m + n + 1}|\}) - |F|^{(n)}(0).
\end{align*}

Taking limits as $m \ra +\infty$ gives the result.
\end{proof}


The previous result already allows for the reconstruction of the entire function $F$
from the values along the convergent sequence. However, we provide a slightly 
more explicit reconstruction for $F$ using Newton's divided differences interpolation
polynomials.


\begin{definition}
\label{defn: interpolating_polynomials}
Given a function $f$ and $z_{1},\ldots, z_{N}$, we define the $N$-th interpolation
polynomial by
\[
f_{N}(z; z_{1},\ldots,z_{N}) = \sum_{n = 0}^{N - 1}
f^{(n)}(z_{1},\ldots,z_{n + 1})\prod_{m = 1}^{n}(z - z_{m}).
\]
\end{definition}


\begin{proposition}
\label{propn: interpolation_polynomials}
The interpolation polynomials satisfy $f_{N}(z_{k}; z_{1},\ldots, z_{N}) = f(z_{k})$ 
for $k = 1, \ldots, N$.
\end{proposition}


\begin{proof}
We prove the result by induction. For the base case we have 
$f_{1}(z) = f(z_{1})$. Assume the result is true for $N$.
We have that 
\[
f_{N + 1}(z; w_{1},\ldots, w_{N}, w_{N + 1})
= f_{N}(z; w_{1},\ldots, w_{N}) 
+ f^{(N)}(w_{1},\ldots, w_{N + 1})\prod_{m = 1}^{N}(z - w_{m}).
\]

This and the inductive hypothesis give that for $k = 1, \ldots, N$ we have 
\[
f_{N + 1}(w_{k}; w_{1},\ldots, w_{N}, w_{N + 1}) 
= f_{N}(w_{k}; w_{1},\ldots, w_{N})
= f(w_{k}).
\]

In addition, directly from the definitions it follows that
\[
f_{N + 1}(z; z_{1},\ldots, z_{N}, z_{N + 1})
= f_{N + 1}(z; z_{1}, \ldots, z_{N + 1}, z_{N}).
\]

These two observations imply the result.
\end{proof}


\begin{theorem}
\label{thm: reconstruction_entire}
Let $F$ be an entire function and suppose that $z_{i} \in \bC$, $z_{i} \ra 0$. 
Then, $F$ can be recovered as a limit of the interpolation polynomials: 
\[
F(z) = \lim_{N \ra +\infty}F_{N}(z; z_{1},\ldots, z_{N})
= \sum_{n = 0}^{\infty}f^{(n)}(z_{1},\ldots,z_{n + 1})
\prod_{m = 1}^{n}(z - z_{m}).
\] 

The series converges absolutely and uniformly over compact sets.
\end{theorem}


\begin{proof}
Let $|z_{i}| \leq R$. From Proposition \ref{propn: majorization} we can bound
absolutely the series by
\begin{align*}
\sum_{n = 0}^{\infty}&\biggr|F^{(n)}(z_{1},\ldots,z_{n + 1}) 
\prod_{m = 1}^{n}(z - z_{m})\biggr| \\
&\leq \sum_{n = 0}^{\infty}|F|^{(n)}(R)(|z| + R)^{n} 
= \sum_{n = 0}^{\infty}\biggl(\frac{1}{n!}\sum_{m = 0}^{\infty}
\frac{|F^{(m + n)}(0)|}{m!}R^{m}\biggr)(|z| + R)^{n}
= \sum_{k = 0}^{\infty}\frac{|F^{(k)}(0)|}{k!}(|z| + 2R)^{k},
\end{align*}

where we used the binomial theorem in the last equality.
The right-hand side is a uniformly convergent series over
compact sets. It follows from Weierstrass' test that the 
convergence of the series 
\[
\sum_{n = 0}^{\infty}f^{(n)}(z_{1},\ldots,z_{n + 1})
\prod_{m = 1}^{n}(z - z_{m})
\]

is absolute and uniform over compact sets. Since the partial
sums are polynomials, in particular entire, then the limit $\wilde{F}(z)$
must be entire as well. However, from Proposition \ref{propn: interpolation_polynomials}
we know that $\wilde{F}(z_{k}) = F(z_{k})$. Thus, $F$ and $\wilde{F}$
are entire and coincide over a convergent sequence, and so $F \equiv \wilde{F}$.
\end{proof}



\end{document}